\documentclass[12pt, reqno]{amsart}
\usepackage{amsmath}
\usepackage{amstext}
\usepackage{amsbsy}
\usepackage{amsopn}
\usepackage{upref}
\usepackage{amsthm}
\usepackage{amsfonts}
\usepackage{amssymb}
\usepackage{mathrsfs}
\usepackage{times}
\usepackage{ascmac}

\allowdisplaybreaks
 \newtheorem{theorem}{Theorem}[section]
 
 \newtheorem{proposition}[theorem]{Proposition}
 \newtheorem{lemma}[theorem]{Lemma}
\theoremstyle{definition}
 \newtheorem{definition}[theorem]{Definition}
\theoremstyle{remark}
 \newtheorem{remark}[theorem]{Remark}
\newcommand{\ep}{\varepsilon}
\newcommand{\p}{\partial}
\newcommand{\RR}{\mathbb{R}}
\newcommand{\TT}{\mathbb{T}}
\newcommand{\UU}{\mathcal{U}}
\newcommand{\VV}{\mathcal{V}}
\newcommand{\WW}{\mathcal{W}}
\newcommand{\TW}{\widetilde{\mathcal{W}}}
\newcommand{\lr}[1]{\left\langle #1 \right\rangle}
\numberwithin{equation}{section}
\begin{document}
%%%%%%
%%%%%% Make Title
%%%%%%
\title[]{
A fourth-order dispersive flow equation for closed curves  
on compact Riemann surfaces
}
\author[E. ONODERA]{Eiji ONODERA}
\address[Eiji Onodera]{Department of Mathematics, 
Faculty of Science, 
Kochi University, 
Kochi 780-8520,Japan}
\email{onodera@kochi-u.ac.jp}
\begin{abstract}
A fourth-order dispersive flow equation for closed curves on 
the canonical two-dimensional unit sphere 
arises in some contexts in physics and fluid mechanics. 
In this paper, a geometric generalization of the sphere-valued model 
is considered, where the solutions are supposed to take values in 
compact Riemann surfaces. 
As a main results, 
time-local existence and the uniqueness of 
a solution to the initial value problem is established 
under the assumption that the sectional curvature of 
the Riemann surface is constant. 
The analytic difficulty comes from the
so-called loss of derivatives 
and the absence of the local smoothing effect. 
The proof is based on the geometric energy method combined with 
a kind of gauge transformation to eliminate the loss of derivatives. 
Specifically, to show the uniqueness of the solution, 
the detailed geometric analysis of 
the solvable structure for the equation is presented. 
\end{abstract}
\maketitle
%%%
%%%%
%%%%%%
%%%%%%
%%%%%% Section 1. Introduction
%%%%
%%%%
\section{Introduction}
\label{section:introduction}
%%%
Dispersive partial differential equations 
have been extensively studied in mathematical research.
Many studies have paid attention to real or 
complex-valued functions as solutions to these equations. 
However, some nonlinear dispersive partial differential 
equations in contexts in classical mechanics and fluid mechanics  
require their solutions to take values in a (curved) Riemannian
manifold.
In general, their nonlinear structures depend on the geometric 
setting of the manifold.
Therefore, concerning how to solve their initial value problem,  
geometric analysis of the relationship between their solvable structure  
and the geometric setting of the manifold plays an essential role. 
\par 
In this field, 
after the pioneering work of Koiso \cite{koiso}, 
the method of geometric analysis 
for the so-called one-dimensional Schr\"odinger flow equation,  
the higher-dimensional generalization 
and a third-order analogue 
has been developed extensively. 
Many results on how to solve their initial value
problem have been established 
mainly from the following three points of view:    
analysis of the solvable structure of 
dispersive partial differential equations(systems), 
an application of Riemannian geometry, 
and analysis of nonlinear partial differential equations 
with physical backgrounds.
See, e.g., \cite{CSU, chihara, chihara2, CO, KLPST, koiso, McGahagan, NSVZ, 
onodera0, onodera1, onodera3}, 
and references therein. 
In this paper, we study a fourth-order analogue  
whose solutions are required to take values in a compact Riemann
surface. 
This is a continuation of \cite{CO2, onodera2} 
and presents the answer to the problem suggested in \cite{chihara2}.
\par
The setting of our problem is stated as follows:  
Given a compact Riemann surface $N$ 
with the complex structure $J$ 
and with a hermitian metric $g$, 
consider the following initial value problem 
\begin{alignat}{2}
 & u_t
  =
  a\,J_u\nabla_x^3u_x
  +
  \{\lambda+ b\, g(u_x,u_x)\}J_u\nabla_xu_x
  +
  c\,g(\nabla_xu_x,u_x)J_uu_x
 \quad &\text{in}\quad
  &\mathbb{R}{\times} \mathbb{T},
\label{eq:pde}
\\
& u(0,x)
  =
  u_0(x)
\quad&\text{in}\quad &\mathbb{T}. 
\label{eq:data}
\end{alignat}
Here 
$\TT=\RR/2\pi \mathbb{Z}$ is the one-dimensional flat torus, 
$u=u(t,x):\RR\times \TT\to N$ is the unknown map 
describing the deformation of 
closed curves lying on $N$ parameterized by $t$,  
$u_0=u_0(x):\TT\to N$ is the given initial map, 
$u_t=du(\frac{\p}{\p t})$, 
$u_x=du(\frac{\p}{\p x})$, 
$du$ is the differential of the map $u$, 
$\nabla_x$ is the covariant derivative along $u$ in $x$,  
$J_u:T_uN\to T_uN$ is the complex structure at $u\in N$, 
$a$, $b$, $c$, and $\lambda$ are real constants. 
If $a,b,c=0$ and $\lambda=1$, 
\eqref{eq:pde} is reduced to the second-order dispersive equation 
of the form 
\begin{equation}
\label{eq:SM}
u_t=J_u\nabla_xu_x,
\end{equation}
which is called a one-dimensional Schr\"odinger flow equation. 
As a fourth-order analogue of \eqref{eq:SM}, 
we call \eqref{eq:pde} with $a\ne 0$ 
a fourth-order dispersive flow equation. 
Hereafter it is assumed that $a\ne 0$.   
\par 
An example of \eqref{eq:pde} with $a\ne 0$ arises in two areas of 
physics, where $N$ is supposed to be the 
canonical two-dimensional unit sphere $\mathbb{S}^2$. 
Indeed, if $N=\mathbb{S}^2$ equipped with the complex structure acting 
as $\pi/2$-degree rotation on each tangent plane  
and with the canonical metric induced from the Euclidean metric in
$\RR^3$,  
\eqref{eq:pde} is described by      
\begin{align}
 & u_t
  =
  u\wedge
  \left[
  a\,\p_x^3u_{x}
  +
  \{\lambda+(a+b)\, (u_x,u_x)\} \p_xu_{x}
  +
  (5a+c)\, (\p_xu_{x},u_x) u_{x}
  \right], 
\label{eq:pdes2}
\end{align}
where $u:\RR\times \TT\to \mathbb{S}^2\subset \RR^3$, 
$\p_x$ is the partial differential operator in $x$ acing on 
$\RR^3$-valued functions, 
$(\cdot,\cdot)$ is the inner product in $\RR^3$, 
and $\wedge$ is the exterior product in $\RR^3$. 
In particular, the $\mathbb{S}^2$-valued model 
\eqref{eq:pdes2} with 
$3a-2b+c=0$ and $\lambda=1$
models the continuum limit of the Heisenberg 
spin chain systems 
with biquadratic exchange interactions(\cite{LPD}), 
where each of $a,b,c$ is decided 
by two independent physical constants. 
Interestingly, 
the same equation can be derived from an equation 
modelling the motion of a 
vortex filament in an incompressible perfect fluid in $\RR^3$
by taking into account of the elliptical deformation effect of the core
due to the self-induced strain 
(\cite{fukumoto, FM}).
\par
For the Schr\"odinger flow equation \eqref{eq:SM} and the higher-dimensional generalization, 
almost all results on the existence of solutions have been established     
assuming essentially that $(N,J,g)$ is 
a compact K\"ahler manifold.
See, e,g, \cite{CSU, KLPST, koiso, McGahagan, NSVZ, SSB} and references therein. 
Under the assumption, the classical energy method 
combined with geometric analysis works to show the local existence results. 
On the other hand, if  $(N,J,g)$ is a compact almost hermitian manifold 
without the K\"ahler condition, 
the classical energy method breaks down, since the so-called loss of derivatives 
occurs from the covariant derivative of the almost complex structure. 
However, Chihara in \cite{chihara} overcame the difficulty by the geometric energy method 
combined with a kind of the gauge transformation acting on the pull-buck bundle. 
Indeed, he established a local existence and uniqueness result for maps from a compact Riemannian manifold into a  
compact almost hermitian manifold. 
After that, he and the author obtained similar results in \cite{CO, onodera1, onodera3, onodera2} 
for a third-order dispersive flow equation 
for maps from $\RR$ or $\TT$ into a compact almost hermitian manifold.    
\par 
In contrast, for our fourth-order 
dispersive flow equation \eqref{eq:pde}, 
we face with the difficulty due to loss of derivatives even if 
$(N,J,g)$ satisfies the K\"ahler condition, 
which is also the case for the $\mathbb{S}^2$-valued physical model \eqref{eq:pdes2}. 
If the spacial domain is the real line $\RR$ 
instead of $\TT$, 
the difficulty can be overcome 
by making use of the local dispersive smoothing effect of the equation  
in some sense. 
Besides, there is much room for the solvable structure. 
Indeed,  in \cite{CO2}, 
the local existence and the uniqueness of a solution 
to the problem on $\RR$ 
were established and were extended to 
compact K\"ahler manifolds as $N$. 
Unfortunately, however, 
the local smoothing effect is absent in our problem 
since the spacial domain $\TT$ is compact. 
In other words, the method of the proof in \cite{CO2} 
is not applicable to our problem. 
Thus the obstruction coming from the loss of derivatives  
is expected to be avoided by finding out a kind of 
special nice solvable structure of the equation. 
\par 
The previous studies of \eqref{eq:pde} on $\TT$ 
are limited as follows: 
Guo, Zeng, and Su in \cite{GZS} investigated the $\mathbb{S}^2$-valued 
physical model \eqref{eq:pdes2} 
with $3a-2b+c=0$ and $\lambda=1$ imposing 
an additional assumption $c=0$. 
Under the assumption, 
\eqref{eq:pdes2} 
is completely integrable, 
and 
they made use of some conservation laws of \eqref{eq:pdes2} 
to show the local existence of a weak solution to 
the initial value problem, 
though the uniqueness was unsolved. 
Chihara in \cite{chihara2} investigated 
fourth-order dispersive systems for $\mathbb{C}^2$-valued 
functions including a system which is 
reduced from \eqref{eq:pde} by the generalized Hasimoto transformation, 
and pointed out that the assumption that the sectional curvature of $N$ 
is constant provides the solvable structure of the initial value problem. 
To the present author's knowledge, 
though the insights seems to grasp the solvable structure
of \eqref{eq:pde}-\eqref{eq:data} essentially, 
it is nontrivial whether we can recover the solution to 
\eqref{eq:pde}-\eqref{eq:data}  
from the solution to the reduced dispersive system. 
\par 
Motivated by them, 
the present author tried to solve directly \eqref{eq:pde}-\eqref{eq:data}
imposing that the sectional curvature on $N$ is constant,  
without using the generalized Hasimoto transformation. 
Recently, he in \cite{onodera2} succeeded to 
show the local existence of a unique solution 
to the initial value
problem for the $\mathbb{S}^2$-valued model \eqref{eq:pdes2} 
without any assumption on $a,b,c,\lambda$ (except for $a\ne 0$), 
where $u_0$ is taken so that 
$u_{0x}\in H^k(\TT;\RR^3)$ with $k\geqslant 6$.  
This is proved by the energy method based on the standard 
Sobolev norm for $\RR^3$-valued functions,  
combined with a kind of gauge transformation. 
\par
The purpose of the present paper is to extend 
the results obtained in 
\cite{onodera2} for $\mathbb{S}^2$-valued model \eqref{eq:pdes2}, 
that is, to establish the 
time-local existence and uniqueness theorem for  
\eqref{eq:pde}-\eqref{eq:data} under the assumption that 
$k\geqslant 6$ and the sectional curvature on $(N,g)$ is constant.  
More precisely, our main results is stated as follows:  
\begin{theorem}
\label{theorem:uniqueness}
Suppose that $(N,J,g)$ is a compact Riemann surface whose sectional curvature is constant.
Let $k$ be an integer satisfying $k\geqslant 6$. 
Then for any 
$u_0\in C(\mathbb{T};N)$ satisfying 
$u_{0x}\in H^k(\TT;TN)$, 
there exists $T=T(\|u_{0x}\|_{H^4(\TT;TN)})>0$
such that
\eqref{eq:pde}-\eqref{eq:data} 
has a unique solution 
$u\in C([-T,T]\times \TT;N)$
satisfying 
$u_x\in 
C([-T,T];H^{k}(\TT;TN)).
$
\end{theorem}
\underline{Notation.} \  
For $\phi:\TT\to N$, 
we denote by $\Gamma(\phi^{-1}TN)$ the set of all vector fields along $\phi$. 
Let $V\in \Gamma(\phi^{-1}TN)$ and let $m$ be nonnegative integer. 
Then we say $V\in H^m(\TT;TN)$ if 
$$
\|V\|_{H^m(\TT;TN)}
:=
\sum_{\ell=0}^m
\int_{\TT}
g(\nabla_x^{\ell}V(x), \nabla_x^{\ell}V(x))\,dx
<\infty.
$$
In particular, if $m=0$, we replace $H^0(\TT;TN)$ 
with $L^2(\TT;TN)$. 
\begin{remark}
Precisely speaking, the existence time $T$ of the solution in 
Theorem~\ref{theorem:uniqueness}
depends on $a,b,c,\lambda$, and the constant sectional curvature of $(N,g)$ 
as well as $\|u_{0x}\|_{H^4(\TT;TN)}$. 
\end{remark}
\begin{remark}
The local existence of the solution  in 
Theorem~\ref{theorem:uniqueness} holds if $k\geqslant 4$. 
The assumption $k\geqslant 6$ comes from  the requirement 
to show the uniqueness.
\end{remark}
\begin{remark} 
Let $w$ be an isometric embedding of $(N,g)$ into 
some Euclidean space $\RR^d$ so that $N$ is considered as a submanifold 
of $\RR^d$. By the Gagliardo-Nirenberg inequality, 
it is found for $u_0$ in Theorem~\ref{theorem:uniqueness} 
that $u_{0x}\in H^k(\TT;TN)$ if and only if  
$(w{\circ}u_0)_x\in H^k(\TT;\RR^d)$, 
where $H^k(\TT;\RR^d)$ denotes the standard $k$-th order 
Sobolev space for $\RR^d$-valued functions on $\TT$. 
By the equivalence, Theorem~\ref{theorem:uniqueness}
actually extends the results obtained in \cite{onodera2}. 
\end{remark} 
\begin{remark}
We can extend Theorem~\ref{theorem:uniqueness} to the case 
where $(N,J,g)$ is a compact K\"ahler manifold with non-zero constant sectional curvature.  
Indeed, 
the argument using \eqref{eq:2d} and \eqref{eq:k1} in the proof 
can be replaced by that using \eqref{eq:constsec} if the curvature is not zero. 
This seems a little bit artificial and the proof is not so different. 
Thus we do not pursue that. 
\end{remark}
\begin{remark}
It is unlikely that we can remove the assumption 
on the curvature of $(N,g)$ in general.   
To see this, let $(N,g)$ be a Riemann surface whose sectional curvature 
is not necessarily constant.  
In view of \cite[Section~4]{chihara2}, 
if we can construct a sufficiently smooth 
solution $u$ to \eqref{eq:pde}-\eqref{eq:data}, 
the following necessary condition 
\begin{equation}
\int_{\TT}
\frac{\p}{\p x}
\left\{
S(u(t,x))
\right\}
g(u_x(t,x),u_x(t,x))
\,dx
=0
\label{eq:structure11}
\end{equation}
is expected to be satisfied for all existence time, 
where 
$S(u(t,x))$ denotes the sectional curvature of $(N,g)$
at $u(t,x)\in N$. 
This requires at least that the left hand side of \eqref{eq:structure11} 
is a conserved quantity in time.  
Even if \eqref{eq:structure11} is true, the initial map $u_0$ is 
required to satisfy 
\begin{equation}
\int_{\TT}
\frac{\p}{\p x}
\left\{
S(u_0(x))
\right\}
g(u_{0x}(x),u_{0x}(x))
\,dx
=0.
\label{eq:structure22}
\end{equation}
On the other hand,
\eqref{eq:structure11} and \eqref{eq:structure22} are 
obviously satisfied if 
the sectional curvature of $(N,g)$ is constant. 
\end{remark}
\par
The idea of the proof of the local existence 
comes from the following formal observation.  
Suppose that $u$ solves
\eqref{eq:pde}-\eqref{eq:data}.
If $k\geqslant 4$, $\nabla_x^ku_x$ satisfies 
\begin{align}
(\nabla_t-a\,J_u\nabla_x^4-c_1\,P_1\nabla_x^2-c_2\,P_2\nabla_x)
\nabla_x^ku_x
&=
\mathcal{O}
\left(
\sum_{m=0}^{k+2}
|\nabla_x^mu_x|_g
\right)
\label{eq:esspde}
\end{align}
where $|\cdot|_g=\left\{g(\cdot,\cdot)\right\}^{1/2}$, 
$c_1$ and $c_2$ are real constants depending on 
$a,b,c,k$ and the sectional curvature on $(N,g)$, 
and $P_1$ and $P_2$ are defined by 
\begin{align}
P_1Y
&=
g(Y,u_x)J_uu_x, 
\quad
P_2Y
=
g(\nabla_xu_x,u_x)J_uY
\nonumber
\end{align}
for any $Y\in \Gamma(u^{-1}TN)$. 
It is found that 
\eqref{eq:esspde} leads to 
the classical energy estimate 
for $\|\nabla_x^ku_x\|_{L^2(\TT;TN)}^2$ 
with loss of derivatives coming only from 
$c_1\,P_1\nabla_x^2$ and $c_2\,P_2\nabla_x$. 
Though the right hand side of \eqref{eq:esspde} includes 
$\nabla_x^2(\nabla_x^ku_x)$ and $\nabla_x(\nabla_x^ku_x)$, 
no loss of derivatives occur thanks to 
the curvature condition and 
the K\"ahler condition on $(N,J,g)$.   
To eliminate the loss of derivatives coming from 
$c_1\,P_1\nabla_x^2$ and $c_2\,P_2\nabla_x$, 
we introduce the so-called gauged function $V_k$ defined by 
\begin{align}
V_k
&=
\nabla_x^ku_x
-\frac{d_1}{2a}\,
g(\nabla_x^{k-2}u_x,J_uu_x)J_uu_x
+
\frac{d_2}{8a}\,
g(u_x,u_x)\nabla_x^{k-2}u_x, 
\label{eq:igauge}
\end{align}
where 
$d_1$ and $d_2$ are constants decided later. 
Here $V_k$ is formally expressed by  
$V_k=(I_d+\Phi_1\nabla_x^{-2}+\Phi_2\nabla_x^{-2})\nabla_x^ku_x$, 
where $I_d$ is the identity on $\Gamma(u^{-1}TN)$ and  
\begin{align}
\Phi_1Y
&=
-\frac{d_1}{2a}\,
g(Y,J_uu_x)J_uu_x, 
\quad 
\Phi_2
=
\frac{d_2}{8a}\,
g(u_x,u_x)Y 
\nonumber
\end{align}
for any $Y\in \Gamma(u^{-1}TN)$. 
Noting that 
$J_u$ commutes with $\Phi_2$  and not with $\Phi_1$, 
we see 
\begin{align}
\left[
a\,J_u\nabla_x^4, \Phi_1\nabla_x^{-2}
\right]\nabla_x^ku_x
&=
(d_1\,P_1\nabla_x^2-d_1\,P_2\nabla_x)\nabla_x^ku_x
+\text{harmless terms},
\label{eq:obs1}
\\
\left[
a\,J_u\nabla_x^4, \Phi_2\nabla_x^{-2}
\right]\nabla_x^ku_x
&=
d_2\,P_2\nabla_x\nabla_x^ku_x
+
\text{harmless terms}.
\label{eq:obs2}
\end{align}
Therefore, if we set $d_1=c_1$ and $d_2=c_1+c_2$, 
the above two commutators eliminate 
$c_1\,P_1\nabla_x^2+c_2\,P_2\nabla_x$ in the partial 
differential equation satisfied by $V_k$, 
and hence the energy estimate for  
$\|V_k\|_{L^2(\TT;TN)}^2$ works. 
The nice choice of the above gauged function 
is inspired by \cite{chihara2}. 
\par 
The strategy for the proof of the local existence of a solution 
is as follows: 
First, we construct a family 
of fourth-order parabolic regularized solutions  
$\left\{u^{\ep}\right\}_{\ep\in (0,1]}$. 
Second, we obtain $\ep$-independent uniform estimates for  
$\|u_x^{\ep}\|_{H^{k-1}(\TT;TN)}^2+\|V_k^{\ep}\|_{L^2(\TT;TN)}^2$
and the lower bound $T>0$ of 
existence time of $\left\{u^{\ep}\right\}_{\ep\in (0,1]}$, 
where $V_k^{\ep}$ is defined by \eqref{eq:igauge} 
replacing $u$ with $u^{\ep}$. 
Finally, the standard compactness argument concludes 
the existence of  
$u\in C([0,T]\times \TT;N)$ so that 
$u_x\in L^{\infty}(0,T;H^k(\TT;TN))\cap 
C([0,T];H^{k-1}(\TT;TN))$ and $u$ solves
\eqref{eq:pde}-\eqref{eq:data}. 
The two commutators \eqref{eq:obs1} and \eqref{eq:obs2} 
in the above formal observation 
will be generated essentially in the computation of the second 
and the third term of the right hand side of 
\eqref{eq:TW1}. 
One can refer to 
\cite{KLPST, koiso, McGahagan}
for tools of computation 
and 
\cite{chihara2, CO, CO2}
for the method of the gauged energy 
employed in the proof. 
\par
The strategy for the proof of the uniqueness of the solution
is stated as follows:  
Suppose that 
$u, v\in C([0,T]\times \TT;N)$ 
are solutions to \eqref{eq:pde}-\eqref{eq:data} 
satisfying 
$u_x, v_x\in 
L^{\infty}(0,T;H^6(\TT;TN))
\cap
C([0,T];H^{5}(\TT;TN))
$ 
with same initial data $u_0$. 
Their existence is ensured by the above local existence results. 
To estimate the difference between $u$ and $v$, 
we regard $u$ and $v$ as functions 
with values in some Euclidean space $\RR^d$. 
Indeed, letting
$w$ be an isometric embedding of $(N,g)$ into $\RR^d$,  
we consider $\RR^d$-valued functions 
defined as follows: 
\begin{align}
U&:=w{\circ} u, 
\quad 
V:=w{\circ} v, 
\quad
Z:=U-V, 
\nonumber
\\
\mathcal{U}&:=dw_u(\nabla_xu_x), 
\quad
\mathcal{V}:=dw_v(\nabla_xv_x), 
\quad
\WW:=\UU-\VV, 
\nonumber
\end{align} 
where $dw_p:T_pN\to T_{w{\circ}p}\RR^d\cong \RR^d$ 
is the differential of $w$ at $p\in N$. 
To complete the proof of the uniqueness,  
it suffices to show $Z=0$. 
First, as shown in \eqref{eq:WWt}, 
we obtain the classical energy estimate 
for $\|Z\|_{L^2}^2+\|Z_x\|_{L^2}^2+\|\WW\|_{L^2}^2$ 
with the loss of derivatives, 
where $\|\cdot\|_{L^2}$ expresses the standard 
$L^2$-norm for $\RR^d$-valued functions
on $\TT$. 
The loss of derivatives has similar form as 
that eliminated by the method of the gauge transformation 
in the proof of the local existence of a solution. 
Observing the analogy, we can easily find $\TW=\WW+\widetilde{\Lambda}$ 
as a gauged function of $\WW$ so that 
the energy estimate for 
$\|Z\|_{L^2}^2+\|Z_x\|_{L^2}^2+\|\TW\|_{L^2}^2$
can be closed. This shows $Z=0$. 
The precise form of $\widetilde{\Lambda}$ will be given in 
\eqref{eq:e1e2}. 
\par
In the proof of the uniqueness, 
we face with another difficulty, 
which does not appear in the proof of the local existence. 
On one hand, 
the proof of the local existence seems clear, 
thanks to the nice matching between 
the geometric formulation of \eqref{eq:pde} 
and the geometric $L^2$-norm $\|\cdot\|_{L^2(\TT;TN)}$. 
On the other hand, the proof of the uniqueness 
requires lengthier computations, 
due to the worse matching between the form of the 
equation satisfied by $U$ and 
the standard $L^2$-norm $\|\cdot\|_{L^2(\TT;\RR^d)}$.
More concretely, the most crucial part of the proof of the uniqueness
is how to derive the energy estimate for $\WW$ 
of the form \eqref{eq:WWt}. 
To derive this, the partial differential equation 
satisfied by $\mathcal{W}$ 
and the energy estimate in $L^2(\TT;\RR^d)$ 
are required. 
However, the analysis of the structure of lower order terms 
in the equation becomes complicated, 
since many terms related to 
the second fundamental form on $N$ and the derivatives
appear to describe the equation satisfied by $U$ or $V$. 
As \eqref{eq:pde} is higher-order equation than 
the Schr\"odinger flow equation or the third-order dispersive flow 
equation previously studied, 
the situation becomes worse. 
Fortunately, however,   
we can successfully formulate the  K\"ahler condition 
and the curvature condition on $(N,J,g)$
to be applicable to our problem, 
and demonstrate 
that only weak loss of derivatives 
is allowed to appear in the energy estimate 
for $\|\WW\|_{L^2(\TT;\RR^d)}^2$. 
In addition, it is to be noted that 
we does not choose $\p_xZ_x$ but choose $\WW$  
in the energy estimate. 
The choice also plays an important role (See,e.g.,Lemma~\ref{lemma:nu}) 
in our proof, 
as well as the choice of $\widetilde{\Lambda}$.   
\par 
By the way, 
the geometric formulation of \eqref{eq:pde} was originally 
proposed by \cite{onodera0}. 
Independently, Anco and Myrzakulov in \cite{AM} 
derived the equation, named a fourth-order Schr\"odinger map equation,  
for $u:\RR\times \RR\to N$ 
or  $u:\RR\times \TT\to N$
of the form 
\begin{align}
-u_t
&=J_{u}\nabla_x^3u_x
+\frac{1}{2}
\nabla_x\left\{
g(u_x,u_x)J_{u}u_x
\right\}
-\frac{1}{2}
g(J_{u}u_x,\nabla_xu_x)u_x.
\label{eq:AM}
\end{align}
Interestingly, if $N$ is a Riemann surface, 
\eqref{eq:AM} is identical with \eqref{eq:pde} 
with $a=-1$, $b=-1$, $c=-1/2$, and $\lambda=0$. 
Therefore, we immediately find that 
Theorem~\ref{theorem:uniqueness}
is valid for
the initial value problem 
also for \eqref{eq:AM}. 
\par
The organization of the present paper is as follows:
In Section~\ref{section:existence}, 
a time-local solution to \eqref{eq:pde}-\eqref{eq:data} is constructed. 
In Section~\ref{section:proof}, 
the proof of Theorem~\ref{theorem:uniqueness} is completed. 
%%%%%%%%%%%%%%%
%%%%%%%%%%%%%%%%%%Section_1
%%%%%%%%%%%%%%%%%%%%%
%
\section{Proof of the existence of a time-local solution}
\label{section:existence}
This section is devoted to the construction of a time-local
solution to \eqref{eq:pde}-\eqref{eq:data}. 
More concretely, the goal of this section is to show the 
following. 
\begin{theorem}
\label{theorem:existence}
Suppose that the sectional curvature of $(N,g)$ is constant. 
Let $k$ be an integer satisfying $k\geqslant 4$. 
Then for any 
$u_0\in C(\mathbb{T};N)$ satisfying 
$u_{0x}\in H^k(\TT;TN)$, 
there exists $T=T(\|u_{0x}\|_{H^4(\TT;TN)})>0$
such that  
\eqref{eq:pde}-\eqref{eq:data} 
has a solution 
$u\in C([-T,T]\times \TT;N)$
satisfying 
$u_x\in 
L^{\infty}(-T,T;H^{k}(\TT;TN))
\cap
C([-T,T];H^{k-1}(\TT;TN)).
$
\end{theorem}
\begin{proof}[Proof of Theorem~\ref{theorem:existence}]
Let $k\geqslant 4$ be fixed. 
It suffices to solve the problem in the positive direction in time. 
We first assume that $u_{0}\in C^{\infty}(\TT;N)$ and construct a 
local solution. 
\par  
As a beginning,  
we consider the initial value problem of the form 
\begin{alignat}{2}
 & u_t
  =
 (-\ep + a\,J_u)\nabla_x^3u_x
  \nonumber 
  \\
  &\quad \quad
  +
  b\, g(u_x,u_x)J_u\nabla_xu_x
  +
  c\,g(\nabla_xu_x,u_x)J_uu_x
  +\lambda\, J_u\nabla_xu_x
 \quad &\text{in}\quad
 &(0,\infty){\times} \mathbb{T},
\label{eq:eppde}
\\
& u(0,x)
  =
  u_0(x)
\quad&\text{in}\quad &\mathbb{T},
\label{eq:epdata}
\end{alignat}
where $\ep\in (0,1]$ is a small positive parameter. 
Thanks to the added term 
$-\ep\,\nabla_x^3u_x$, 
\eqref{eq:eppde} is a fourth-order 
quasilinear parabolic system, 
and \eqref{eq:eppde}-\eqref{eq:epdata} has a unique local smooth
solution which we will denote $u^{\ep}$.   
\begin{lemma}
\label{lemma:parabolic}  
For each $\ep\in (0,1]$, 
there exists a positive constant $T_{\ep}$ 
depending on $\ep$ and $\|u_{0x}\|_{H^4(\TT;TN)}$ 
such that  
\eqref{eq:eppde}-\eqref{eq:epdata} possesses a 
unique solution $u^{\ep}\in C^{\infty}([0,T_{\ep}]\times \TT;N)$. 
\end{lemma}
We can show Lemma~\ref{lemma:parabolic} by the mix of 
a sixth-order parabolic regularization and 
a geometric classical energy method 
without the constant curvature condition on $(N,g)$.
The proof
almost falls into the scope 
of that of \cite[Lemma~3.1]{CO2} by replacing $\RR$ with $\TT$ 
and by restricting to a compact Riemann surface as $N$. 
Thought a slight modification is required in the proof, 
the difference is not essential 
and thus we omit the detail of the proof. 
\par 
In the next step, 
letting $\left\{u^{\ep}\right\}_{\ep\in(0,1]}$ be a family of solutions 
to \eqref{eq:eppde}-\eqref{eq:epdata} constructed in 
Lemma~\ref{lemma:parabolic}, 
we obtain $\ep$-independent 
energy estimates for 
$\left\{u^{\ep}_x\right\}_{\ep\in(0,1]}$.
Precisely speaking, we obtain a uniform lower bound 
$T$ of $\left\{T_{\ep}\right\}_{\ep\in (0,1]}$ 
and show that $\left\{u^{\ep}_x\right\}_{\ep\in (0,1]}$ 
is bounded in $L^{\infty}(0,T;H^k(\TT;TN))$. 
However, the classical 
energy estimate for  $\|u^{\ep}_x\|_{H^k(\TT;TN)}$  
causes loss of derivatives. 
To overcome the difficulty, 
we introduce a gauged function $V^{\ep}_k$ defined by 
\begin{align}
V^{\ep}_k
&=
\nabla_x^ku_x^{\ep}
+
\Lambda^{\ep}
=
\nabla_x^ku_x^{\ep}
+
\Lambda^{\ep}_1
+
\Lambda^{\ep}_2,
\label{eq:V_m}
\end{align}
where 
\begin{align}
\Lambda_1^{\ep}
&=
-\frac{d_1}{2a}\,
g(\nabla_x^{k-2}u_x^{\ep},J_uu_x^{\ep})J_uu_x^{\ep}, 
\quad 
\Lambda_2^{\ep}
=
\frac{d_2}{8a}\,
g(u_x^{\ep},u_x^{\ep})\nabla_x^{k-2}u_x^{\ep}, 
\nonumber
\end{align}
and $d_1, d_2\in\RR$ are real constants which will be decided later 
depending only on $a,b,c,k$ and the constant sectional curvature 
of $(N,g)$.  
Furthermore, we introduce the associated gauged energy 
$N_k(u^{\ep}(t))$ defined by  
\begin{equation}
N_k(u^{\ep}(t))
=
\sqrt{
\|u_x^{\ep}(t)\|_{H^{k-1}(\TT;TN)}^2
+
\|V_k^{\ep}(t)\|_{L^2(\TT;TN)}^2
}.
\label{eq:N_k}
\end{equation}
We restrict the time interval on $[0,T_{\ep}^{\star}]$ 
with $T^{\star}_{\ep}$ defined by 
$$
T^{\star}_{\ep}
=
\sup
\left\{
T>0 \ | \ 
N_4(u^{\ep}(t))\leqslant 2N_4(u_0)
\quad 
\text{for all}
\quad
t\in[0,T]
\right\}. 
$$
By the Sobolev embedding, we immediately find 
that there holds
\begin{equation}
\frac{1}{C}N_k(u^{\ep}(t))
\leqslant 
\|u_x^{\ep}(t)\|_{H^k(\TT;TN)}
\leqslant 
C\,N_k(u^{\ep}(t))
\quad \
\text{for any}
\quad \
t\in [0,T_{\ep}^{\star}], 
\label{eq:timecut}
\end{equation}
with $C=C(\|u_{0x}\|_{H^4(\TT;TN)})>1$ 
being an $\ep$-independent constant. 
We shall show that there exists a constant 
$T=T(\|u_{0x}\|_{H^4(\TT;TN)})>0$ 
which is independent of $\ep\in (0,1]$ and $k$
such that $T^{\star}_{\ep}\geqslant T$ uniformly in $\ep\in (0,1]$ 
and that $\left\{N_k(u^{\ep})\right\}_{\ep\in (0,1]}$ 
is bounded in $L^{\infty}(0,T)$. 
If it is true, this together with \eqref{eq:timecut} 
implies that $\left\{u_x^{\ep}\right\}_{\ep\in (0,1]}$ is 
bounded in $L^{\infty}(0,T;H^k(\TT;TN))$. 
\par 
Having them in mind, 
let us focus on the uniform energy 
estimate for $\left\{N_k(u^{\ep})\right\}_{\ep\in (0,1]}$.  
We set $u=u^{\ep}$, $V_k=V_k^{\ep}$, 
$\Lambda=\Lambda^{\ep}$, 
$\Lambda_1=\Lambda_1^{\ep}$, $\Lambda_2=\Lambda_2^{\ep}$,
$\|\cdot\|_{H^0(\TT;TN)}=\|\cdot\|_{L^2(\TT;TN)}=\|\cdot\|_{L^2}$, 
$\|\cdot\|_{H^m(\TT;TN)}=\|\cdot\|_{H^m}$ for $m=1,\ldots,k$, 
and 
$\sqrt{g(\cdot,\cdot)}=|\cdot|_g$,  
for ease of notation. 
Since $g$ is a hermitian metric, 
$g(J_uY_1,J_uY_2)=g(Y_1,Y_2)$ holds for any 
$Y_1,Y_2\in\Gamma(u^{-1}TN)$. 
Since Riemann surfaces with hermitian metric are K\"ahler manifolds, 
$\nabla_xJ_u=J_u\nabla_x$ and $\nabla_tJ_u=J_u\nabla_t$ hold.  
We denote the sectional curvature of $(N,g)$ by $S$ which is constant.
Any positive constant which depends on  
$a$, $b$, $c$, $\lambda$, $k$, $S$, 
$\|u_{0x}\|_{H^4}$ 
and not on $\ep\in (0,1]$ will be denoted by the same $C$. 
Note that $k\geqslant 4$ and the Sobolev embedding 
$H^1(\TT)\subset C(\TT)$
yield 
$\|\nabla_x^4u_x\|_{L^{\infty}(0,T_{\ep}^{\star};L^2)}\leqslant C$
and $\|\nabla_x^mu_x\|_{L^{\infty}((0,T_{\ep}^{\star})\times \TT)}\leqslant C$ 
for $m=0,1,\ldots,3$. 
These properties will be used without any comment in this section.
\par 
We now investigate the energy estimate for $\|V_k\|_{L^2}^2$.
It follows that  
\begin{align}
\frac{1}{2}\frac{d}{dt}
\|V_k\|_{L^2}^2
&=
\int_{\TT}
g(\nabla_tV_k,V_k)
dx
\nonumber
\\
&=
\int_{\TT}
g(\nabla_t(\nabla_x^ku_x),V_k)
dx
+
\int_{\TT}
g(\nabla_t\Lambda,V_k)
dx
\nonumber
\\
&=
\int_{\TT}
g(\nabla_t(\nabla_x^ku_x),\nabla_x^ku_x)
dx
+
\int_{\TT}
g(\nabla_t(\nabla_x^ku_x),\Lambda)
dx
+
\int_{\TT}
g(\nabla_t\Lambda,V_k)
dx.
\label{eq:eqV_k}
\end{align}
To evaluate the right hand side
(denoted by RHS hereafter for short) of \eqref{eq:eqV_k}, 
we compute the partial differential 
equation satisfied by $\nabla_x^ku_x$. 
Recalling that $\nabla_xu_t=\nabla_tu_x$ and 
$(\nabla_x\nabla_t-\nabla_t\nabla_x)Y=R(u_x,u_t)Y$ 
for any $Y\in \Gamma(u^{-1}TN)$ where 
$R=R(\cdot,\cdot)$ denotes the Riemann curvature tensor on $(N,g)$, 
we have
\begin{align}
\nabla_t(\nabla_x^ku_x)
&=
\nabla_x^{k+1}u_t
+
\sum_{m=0}^{k-1}
\nabla_x^{k-1-m}
\left\{
R(u_t,u_x)\nabla_x^mu_x
\right\}
=:
\nabla_x^{k+1}u_t
+Q
.
\label{eq:na1}
\end{align}
First, we use \eqref{eq:eppde} to compute the second term of the RHS of the
above, 
which becomes 
\begin{align}
%\sum_{m=0}^{k-1}
%\nabla_x^{k-1-m}
%\left\{
%R(u_t,u_x)\nabla_x^mu_x
%\right\}
Q&=
-\ep\, 
\sum_{m=0}^{k-1}\nabla_x^{k-1-m}
\left\{
R(\nabla_x^3u_x,u_x)\nabla_x^mu_x
\right\}
\nonumber
\\
&\quad 
+a\,
\sum_{m=0}^{k-1}\nabla_x^{k-1-m}
\left\{
R(J_u\nabla_x^3u_x,u_x)\nabla_x^mu_x
\right\}
\nonumber
\\
&\quad 
+\lambda\,
\sum_{m=0}^{k-1}\nabla_x^{k-1-m}
\left\{
R(J_u\nabla_xu_x,u_x)\nabla_x^mu_x
\right\}
\nonumber
\\
&\quad 
+b\,
\sum_{m=0}^{k-1}\nabla_x^{k-1-m}
\left\{
g(u_x,u_x)R(J_u\nabla_xu_x,u_x)\nabla_x^mu_x
\right\}
\nonumber
\\
&\quad 
+c\,
\sum_{m=0}^{k-1}\nabla_x^{k-1-m}
\left\{
g(\nabla_xu_x,u_x)R(J_uu_x,u_x)\nabla_x^mu_x
\right\}. 
\nonumber
\end{align}
Thus, by using the Sobolev embedding 
and the Gagliardo-Nirenberg inequality, 
we obtain 
\begin{align}
%&\sum_{m=0}^{k-1}
%\nabla_x^{k-1-m}
%\left\{
%R(u_t,u_x)\nabla_x^mu_x
%\right\}
%\nonumber
%\\
%&
Q&=
\ep\,
\mathcal{O}
(|\nabla_x^{k+2}u_x|_g)
+a\,Q_0
+
\mathcal{O}
\left(
\sum_{m=0}^k
|\nabla_x^mu_x|_g
\right), 
\label{eq:na2}
\end{align}
where 
$$
Q_0
=
\sum_{m=0}^{k-1}\nabla_x^{k-1-m}
\left\{
R(J_u\nabla_x^3u_x,u_x)\nabla_x^mu_x
\right\}.
$$ 
Since $S$ is the constant sectional curvature of $(N,g)$,  
\begin{equation}
R(Y_1,Y_2)Y_3
=S
\left\{
g(Y_2,Y_3)Y_1
-g(Y_1,Y_3)Y_2
\right\}
\label{eq:constsec}
\end{equation}
holds for any $Y_1,Y_2,Y_3\in \Gamma(u^{-1}TN)$.  
Using the formula, $Q_0$ is expressed as follows. 
\begin{align}
Q_0
&=
S\,\sum_{m=0}^{k-1}
\nabla_x^{k-1-m}
\left\{
g(\nabla_x^mu_x,u_x)J_u\nabla_x^3u_x
-g(\nabla_x^mu_x, J_u\nabla_x^3u_x)u_x
\right\}
\nonumber
\\
&=S\,(Q_{0,1}+Q_{0,2}+Q_{0,3}), 
\label{eq:Q0}
\end{align}
where 
\begin{align}
Q_{0,1}
&=
\nabla_x^{k-1}
\left\{
g(u_x,u_x)J_u\nabla_x^3u_x
-g(u_x,J_u\nabla_x^3u_x)u_x
\right\}, 
\nonumber
\\
Q_{0,2}
&=
\nabla_x^{k-2}
\left\{
g(\nabla_xu_x,u_x)J_u\nabla_x^3u_x
-g(\nabla_xu_x,J_u\nabla_x^3u_x)u_x
\right\}, 
\nonumber
\\
Q_{0,3}
&=
\sum_{m=2}^{k-1}
\nabla_x^{k-1-m}
\left\{
g(\nabla_x^mu_x,u_x)J_u\nabla_x^3u_x
-g(\nabla_x^mu_x, J_u\nabla_x^3u_x)u_x
\right\}. 
\nonumber
\end{align}
For $Q_{0,1}$, 
the product formula implies 
\begin{align}
Q_{0,1}
&=
\sum_{\mu+\nu=0}^{k-1}
\frac{(k-1)!}{\mu!\nu!(k-1-\mu-\nu)!}
\,
g(\nabla_x^{\mu}u_x, \nabla_x^{\nu}u_x)J_u\nabla_x^{k+2-\mu-\nu}u_x
\nonumber
\\
&\quad
-
\sum_{\mu+\nu=0}^{k-1}
\frac{(k-1)!}{\mu!\nu!(k-1-\mu-\nu)!}
\,
g(\nabla_x^{\mu}u_x, J_u\nabla_x^{\nu+3}u_x)\nabla_x^{k-1-\mu-\nu}u_x
\nonumber
\\
&=
g(u_x,u_x)J_u\nabla_x^{k+2}u_x
+
2(k-1)g(\nabla_xu_x,u_x)J_u\nabla_x^{k+1}u_x
-g(u_x,J_u\nabla_x^{k+2}u_x)u_x
\nonumber
\\
&\quad 
-(k-1)g(\nabla_xu_x,J_u\nabla_x^{k+1}u_x)u_x
-(k-1)g(u_x,J_u\nabla_x^{k+1}u_x)\nabla_xu_x
\nonumber
\\
&\quad 
+
\sum_{\mu+\nu=2}^{k-1}
\frac{(k-1)!}{\mu!\nu!(k-1-\mu-\nu)!}
\,
g(\nabla_x^{\mu}u_x, \nabla_x^{\nu}u_x)J_u\nabla_x^{k+2-\mu-\nu}u_x
\nonumber
\\
&\quad
-
\sum_{\substack{\mu+\nu=0,\\ \nu\leqslant k-3}}^{k-1}
\frac{(k-1)!}{\mu!\nu!(k-1-\mu-\nu)!}
\,
g(\nabla_x^{\mu}u_x, J_u\nabla_x^{\nu+3}u_x)\nabla_x^{k-1-\mu-\nu}u_x
\nonumber
\\
&=
g(u_x,u_x)J_u\nabla_x^{k+2}u_x
+
2(k-1)g(\nabla_xu_x,u_x)J_u\nabla_x^{k+1}u_x
\nonumber
\\
&\quad
-g(u_x,J_u\nabla_x^{k+2}u_x)u_x 
-(k-1)g(\nabla_xu_x,J_u\nabla_x^{k+1}u_x)u_x
\nonumber
\\
&\quad
-(k-1)g(u_x,J_u\nabla_x^{k+1}u_x)\nabla_xu_x 
+
\mathcal{O}
\left(
\sum_{m=0}^k
|\nabla_x^mu_x|_g
\right). 
\label{eq:na3}
\end{align}
Here it is to be emphasized that  
\begin{align}
g(Y, u_x)u_x+g(Y, J_uu_x)J_uu_x&=g(u_x,u_x)Y
\label{eq:2d}
\end{align}
holds for any $Y\in \Gamma(u^{-1}TN)$, 
since $N$ is a two-dimensional real manifold. 
Using \eqref{eq:2d} with $Y=J_u\nabla_x^{k+2}u_x$, we rewrite 
the third term of the RHS of 
\eqref{eq:na3}
to have 
\begin{align}
-g(u_x,J_u\nabla_x^{k+2}u_x)u_x
&=
-g(u_x,u_x)J_u\nabla_x^{k+2}u_x
+g(J_uu_x,J_u\nabla_x^{k+2}u_x)J_uu_x
\nonumber
\\
&=
-g(u_x,u_x)J_u\nabla_x^{k+2}u_x
+g(\nabla_x^{k+2}u_x,u_x)J_uu_x. 
\label{eq:na4}
\end{align}
Substituting \eqref{eq:na4} into the RHS of \eqref{eq:na3}, we obtain 
\begin{align}
Q_{0,1}
&=
2(k-1)g(\nabla_xu_x,u_x)J_u\nabla_x^{k+1}u_x
+g(\nabla_x^{k+2}u_x,u_x)J_uu_x
\nonumber
\\
&\quad 
+(k-1)g(\nabla_x^{k+1}u_x, J_u\nabla_xu_x)u_x
+(k-1)g(\nabla_x^{k+1}u_x, J_uu_x)\nabla_xu_x
\nonumber
\\
&\quad 
+
\mathcal{O}
\left(
\sum_{m=0}^k
|\nabla_x^mu_x|_g
\right). 
\label{eq:Q01}
\end{align}
For $Q_{0,2}$, in the same way as that for $Q_{0,1}$, 
we deduce 
\begin{align}
Q_{0,2}
&=
\sum_{\mu+\nu=0}^{k-2}
\frac{(k-2)!}
{\mu!\nu!(k-2-\mu-\nu)!}
\,
g(\nabla_x^{\mu+1}u_x, \nabla_x^{\nu}u_x)J_u\nabla_x^{k+1-\mu-\nu}u_x
\nonumber
\\
&\quad 
-\sum_{\mu+\nu=0}^{k-2}
\frac{(k-2)!}
{\mu!\nu!(k-2-\mu-\nu)!}
\,
g(\nabla_x^{\mu+1}u_x, J_u\nabla_x^{\nu+3}u_x)\nabla_x^{k-2-\mu-\nu}u_x
\nonumber
\\
&=
g(\nabla_xu_x,u_x)J_u\nabla_x^{k+1}u_x
-g(\nabla_xu_x, J_u\nabla_x^{k+1}u_x)u_x
\nonumber
\\
&\quad
+
\sum_{\mu+\nu=1}^{k-2}
\frac{(k-2)!}
{\mu!\nu!(k-2-\mu-\nu)!}
\,
g(\nabla_x^{\mu+1}u_x, \nabla_x^{\nu}u_x)J_u\nabla_x^{k+1-\mu-\nu}u_x
\nonumber
\\
&\quad 
-\sum_{\substack{\mu+\nu=0,\\ \nu\leqslant k-3}}^{k-2}
\frac{(k-2)!}
{\mu!\nu!(k-2-\mu-\nu)!}
\,
g(\nabla_x^{\mu+1}u_x, J_u\nabla_x^{\nu+3}u_x)\nabla_x^{k-2-\mu-\nu}u_x
\nonumber
\\
&=
g(\nabla_xu_x,u_x)J_u\nabla_x^{k+1}u_x
+g(\nabla_x^{k+1}u_x, J_u\nabla_xu_x)u_x
+
\mathcal{O}
\left(
\sum_{m=0}^k
|\nabla_x^mu_x|_g
\right). 
\label{eq:Q02}
\end{align}
For $Q_{0,3}$, the Sobolev embedding and 
the Gagliardo-Nirenberg inequality imply 
\begin{align}
Q_{0,3}
&=\mathcal{O}
\left(
\sum_{m=0}^k
|\nabla_x^mu_x|_g
\right). 
\label{eq:Q03}
\end{align}
Collecting 
\eqref{eq:na2}, 
\eqref{eq:Q0}, 
\eqref{eq:Q01}, 
\eqref{eq:Q02}, 
and 
\eqref{eq:Q03}, 
we obtain 
\begin{align}
%&\sum_{m=0}^{k-1}
%\nabla_x^{k-1-m}
%\left\{
%R(u_t,u_x)\nabla_x^mu_x
%\right\}
%\nonumber
%\\
Q&=
\ep\,
\mathcal{O}
\left(
|\nabla_x^{k+2}u_x|_g
\right)
+aS\, 
g(\nabla_x^2(\nabla_x^ku_x), u_x)J_uu_x
\nonumber
\\
&\quad
+aS(2k-1)\,g(\nabla_xu_x,u_x)J_u\nabla_x(\nabla_x^ku_x)
+aSk\,g(\nabla_x(\nabla_x^ku_x),J_u\nabla_xu_x)u_x
\nonumber
\\
&\quad
+
aS(k-1)\,
g(\nabla_x(\nabla_x^ku_x), J_uu_x)\nabla_xu_x
+
\mathcal{O}
\left(
\sum_{m=0}^k
|\nabla_x^mu_x|_g
\right).
\label{eq:nab_2}
\end{align}
Second, we use \eqref{eq:eppde} to compute the first term of the RHS of 
\eqref{eq:na1}. 
A simple computation shows 
\begin{align}
\nabla_x^{k+1}u_t
&=
-\ep\nabla_x^4(\nabla_x^ku_x)
+a\,J_u\nabla_x^4(\nabla_x^ku_x)
+\lambda\,J_u\nabla_x^2(\nabla_x^ku_x)
+b\,Q_{1,1}+c\,Q_{1,2}, 
\label{eq:nab_1}
\end{align}
where 
\begin{align}
Q_{1,1}
&=
\nabla_x^{k+1}\left\{
g(u_x,u_x)J_u\nabla_xu_x
\right\}
\nonumber
\\
&=
\sum_{\mu+\nu=0}^{k+1}
\frac{(k+1)!}{\mu!\nu!(k+1-\mu-\nu)!}
\,
g(\nabla_x^{\mu}u_x,\nabla_x^{\nu}u_x)J_u\nabla_x^{k+2-\mu-\nu}u_x
\nonumber
\\
&=
g(u_x,u_x)J_u\nabla_x^{k+2}u_x
+2(k+1)g(\nabla_xu_x,u_x)J_u\nabla_x^{k+1}u_x
\nonumber
\\
&\quad
+2g(\nabla_x^{k+1}u_x,u_x)J_u\nabla_xu_x
\nonumber
\\
&\quad
+\sum_{\substack{\mu+\nu=2, \\ \mu,\nu\leqslant k}}^{k+1}
\frac{(k+1)!}{\mu!\nu!(k+1-\mu-\nu)!}
\,
g(\nabla_x^{\mu}u_x,\nabla_x^{\nu}u_x)J_u\nabla_x^{k+2-\mu-\nu}u_x
\nonumber
\\
&=
\nabla_x\left\{
g(u_x,u_x)J_u\nabla_x(\nabla_x^ku_x)
\right\}
+2k\,
g(\nabla_xu_x,u_x)J_u\nabla_x(\nabla_x^ku_x)
\nonumber
\\
&\quad
+2\,g(\nabla_x(\nabla_x^ku_x), u_x)J_u\nabla_xu_x
+
\mathcal{O}
\left(
\sum_{m=0}^k
|\nabla_x^mu_x|_g
\right),
\label{eq:Q11}
\intertext{and}
Q_{1,2}
&=
\nabla_x^{k+1}
\left\{
g(\nabla_xu_x,u_x)J_uu_x
\right\}
\nonumber
\\
&=
\sum_{\mu+\nu=0}^{k+1}
\frac{(k+1)!}{\mu!\nu!(k+1-\mu-\nu)!}
\,g(\nabla_x^{\mu+1}u_x, \nabla_x^{\nu}u_x)J_u\nabla_x^{k+1-\mu-\nu}u_x 
\nonumber
\\
&=
g(\nabla_xu_x,u_x)J_u\nabla_x^{k+1}u_x
+g(\nabla_x^{k+2}u_x,u_x)J_uu_x
\nonumber
\\
&\quad 
+(k+1)g(\nabla_x^{k+1}u_x,\nabla_xu_x)J_uu_x
+
g(\nabla_xu_x,\nabla_x^{k+1}u_x)J_uu_x
\nonumber
\\
&\quad 
+
(k+1)g(\nabla_x^{k+1}u_x,u_x)J_u\nabla_xu_x
\nonumber
\\
&\quad 
+
\sum_{\substack{\mu+\nu=1, \\ \mu\leqslant k-1, \\ \nu\leqslant k}}^{k+1}
\frac{(k+1)!}{\mu!\nu!(k+1-\mu-\nu)!}
\,g(\nabla_x^{\mu+1}u_x, \nabla_x^{\nu}u_x)J_u\nabla_x^{k+1-\mu-\nu}u_x 
\nonumber
\\
&=
g(\nabla_x^2(\nabla_x^ku_x),u_x)J_uu_x
+g(\nabla_xu_x,u_x)J_u\nabla_x(\nabla_x^ku_x)
\nonumber
\\
&\quad 
+(k+2)g(\nabla_x(\nabla_x^ku_x),\nabla_xu_x)J_uu_x
+(k+1)g(\nabla_x(\nabla_x^ku_x),u_x)J_u\nabla_xu_x
\nonumber
\\
&\quad 
+
\mathcal{O}
\left(
\sum_{m=0}^k
|\nabla_x^mu_x|_g
\right).
\label{eq:Q12}
\end{align}
By collecting \eqref{eq:nab_2} and \eqref{eq:nab_1} with 
\eqref{eq:Q11} and with \eqref{eq:Q12}, we have 
\begin{align}
\nabla_t(\nabla_x^ku_x)
&=
-\ep\nabla_x^4(\nabla_x^ku_x)
+
\ep\,
\mathcal{O}
\left(
|\nabla_x^{k+2}u_x|_g
\right)
\nonumber
\\
&\quad 
+a\,J_u\nabla_x^4(\nabla_x^ku_x)
+\lambda\,J_u\nabla_x^2(\nabla_x^ku_x)
+b\,\nabla_x
\left\{
g(u_x,u_x)J_u\nabla_x(\nabla_x^ku_x)
\right\}
\nonumber
\\
&\quad 
+(aS+c)\,g(\nabla_x^2(\nabla_x^ku_x), u_x)J_uu_x
\nonumber
\\
&\quad
+\left\{
aS(2k-1)+2kb+c
\right\}
\,g(\nabla_xu_x,u_x)J_u\nabla_x(\nabla_x^ku_x)
\nonumber
\\
&\quad +
\left\{
2b+(k+1)c
\right\}\,
g(\nabla_x(\nabla_x^ku_x),u_x)J_u\nabla_xu_x
\nonumber
\\
&\quad 
+
(k+2)c\,
g(\nabla_x(\nabla_x^ku_x),\nabla_xu_x)J_uu_x
\nonumber
\\
&\quad
+aSk\,g(\nabla_x(\nabla_x^ku_x),J_u\nabla_xu_x)u_x
\nonumber
\\
&\quad 
+aS(k-1)\,g(\nabla_x(\nabla_x^ku_x), J_uu_x)\nabla_xu_x
\nonumber
\\
&\quad
+\mathcal{O}
\left(
\sum_{m=0}^k
|\nabla_x^mu_x|_g
\right).
\label{eq:maya}
\end{align}
\par
Furthermore, we modify the expression of some terms including 
$\nabla_x(\nabla_x^ku_x)$ 
to detect their essential structure. 
Let $Y\in \Gamma(u^{-1}TN)$ be fixed. 
We first use \eqref{eq:2d} to see 
\begin{align}
g(u_x,u_x)J_uY
&=
g(J_uY,u_x)u_x+g(J_uY,J_uu_x)J_uu_x
\nonumber
\\
&
=
g(Y,u_x)J_uu_x-g(Y,J_uu_x)u_x.
\nonumber
\end{align} 
Acting $\nabla_x$ to both sides of the above, 
we have 
\begin{align}
2\,g(\nabla_xu_x,u_x)J_uY
&=
g(Y,\nabla_xu_x)J_uu_x
+
g(Y,u_x)J_u\nabla_xu_x
\nonumber
\\
&\quad
-g(Y,J_u\nabla_xu_x)u_x
-g(Y,J_uu_x)\nabla_xu_x. 
\label{eq:maya2} 
\end{align}
We next introduce the following expression:  
\begin{align}
A_1Y
&=
g(Y,\nabla_xu_x)J_uu_x
+
g(Y,u_x)J_u\nabla_xu_x
\nonumber
\\
&\quad
+g(Y,J_u\nabla_xu_x)u_x
+g(Y,J_uu_x)\nabla_xu_x, 
\nonumber
\\
A_2Y
&=
g(Y,J_uu_x)\nabla_xu_x
-g(Y,J_u\nabla_xu_x)u_x.  
\nonumber
\end{align} 
We find ${}^tA_1=A_1$ and ${}^tA_2=A_2$ in $T_uN$. 
More precisely we can show the following.
\begin{proposition}
Let $Y_1,Y_2\in \Gamma(u^{-1}TN)$. 
Then 
\begin{align}
g(A_iY_1,Y_2)&=g(Y_1,A_iY_2)
\label{eq:gsym}
\end{align}
holds for each $(t,x)\in [0,T_{\ep}^{*}]\times \TT$ with $i=1,2$. 
\label{proposition:gsym}
\end{proposition}
\begin{proof}[Proof of Proposition~\ref{proposition:gsym}] 
If $i=1$, \eqref{eq:gsym} immediately follows from the definition 
of $A_1$. 
If $i=2$, \eqref{eq:gsym} follows from 
\begin{equation}
\left\{g(u_x,u_x)\right\}^2
\left\{
g(A_2Y_1,Y_2)-g(Y_1,A_2Y_2)
\right\}=0,
\label{eq:gsym2}
\end{equation}
since both sides of \eqref{eq:gsym} vanish at the point $(t,x)$ 
with $u_x(t,x)=0$. 
Indeed we can show \eqref{eq:gsym2} by the following computations. 
We first write 
\begin{align}
g(u_x,u_x)A_2Y_1
&=
g(u_x,u_x)
\left\{
g(Y_1,J_uu_x)\nabla_xu_x
-g(Y_1,J_u\nabla_xu_x)u_x
\right\}
\nonumber
\\
&=
g(g(u_x,u_x)Y_1,J_uu_x)\nabla_xu_x
-g(g(u_x,u_x)Y_1,J_u\nabla_xu_x)u_x, 
\nonumber
\end{align}
and we use \eqref{eq:2d} with $Y=Y_1$ to see 
\begin{align}
g(u_x,u_x)A_2Y_1
&=
g(g(Y_1,u_x)u_x+g(Y_1,J_uu_x)J_uu_x,J_uu_x)\nabla_xu_x
\nonumber
\\
&\quad 
-g(g(Y_1,u_x)u_x+g(Y_1,J_uu_x)J_uu_x,J_u\nabla_xu_x)u_x
\nonumber
\\
&=
g(u_x,u_x)g(Y_1,J_uu_x)\nabla_xu_x
-g(u_x,J_u\nabla_xu_x)g(Y_1,u_x)u_x
\nonumber
\\
&\quad 
-g(u_x,\nabla_xu_x)g(Y_1,J_uu_x)u_x. 
\nonumber 
\end{align}
This implies 
\begin{align}
\left\{g(u_x,u_x)\right\}^2g(A_2Y_1,Y_2) 
&=
g(g(u_x,u_x)A_2Y_1, g(u_x,u_x)Y_2)
\nonumber
\\
&=
g(u_x,u_x)g(Y_1,J_uu_x)g(\nabla_xu_x,g(u_x,u_x)Y_2)
\nonumber
\\
&\quad
-g(u_x,J_u\nabla_xu_x)g(Y_1,u_x)g(u_x,g(u_x,u_x)Y_2)
\nonumber
\\
&\quad 
-g(u_x,\nabla_xu_x)g(Y_1,J_uu_x)g(u_x,g(u_x,u_x)Y_2). 
\label{eq:iri1}
\end{align}
Using \eqref{eq:2d} again with $Y=Y_2$, we see 
\begin{align}
g(u_x,g(u_x,u_x)Y_2)
&=
g(u_x,u_x)g(Y_2,u_x), 
\nonumber
\\
g(\nabla_xu_x,g(u_x,u_x)Y_2)
&=
g(\nabla_xu_x,u_x)g(Y_2,u_x)
+
g(\nabla_xu_x,J_uu_x)g(Y_2,J_uu_x).
\nonumber
\end{align}
Substituting them into \eqref{eq:iri1}, we have 
\begin{align}
&\left\{g(u_x,u_x)\right\}^2
g(A_2Y_1,Y_2)
\nonumber
\\
&=
g(u_x,u_x)g(\nabla_xu_x,J_uu_x)
\left\{
g(Y_1,J_uu_x)g(Y_2,J_uu_x)
+
g(Y_1,u_x)g(Y_2,u_x)
\right\}.
\nonumber
\end{align}
As the form of the RHS is symmetric with respect to 
$Y_1$ and $Y_2$, we immediately conclude that 
the desired property \eqref{eq:gsym2} holds.  
\end{proof} 
Using \eqref{eq:maya2} and the definition of $A_1$ and $A_2$, 
we have  
\begin{align}
&g(Y,J_uu_x)\nabla_xu_x
\nonumber
\\
&=
\frac{1}{4}
\biggl\{
g(Y,\nabla_xu_x)J_uu_x
+
g(Y,u_x)J_u\nabla_xu_x
+g(Y,J_u\nabla_xu_x)u_x
+g(Y,J_uu_x)\nabla_xu_x
\biggr\}
\nonumber
\\
&\quad
-\frac{1}{4}
\biggl\{
g(Y,\nabla_xu_x)J_uu_x
+
g(Y,u_x)J_u\nabla_xu_x
-g(Y,J_u\nabla_xu_x)u_x
-g(Y,J_uu_x)\nabla_xu_x
\biggr\}
\nonumber
\\
&\quad
+\frac{1}{2}
\biggl\{
g(Y,J_uu_x)\nabla_xu_x
-g(Y,J_u\nabla_xu_x)u_x
\biggr\}
\nonumber
\\
&=
-\frac{1}{2}\,
g(\nabla_xu_x,u_x)J_uY
+\frac{1}{4}A_1Y
+\frac{1}{2}A_2Y.
\label{eq:ayaA}
\end{align}
In the same way, we have 
\begin{align}
g(Y,J_u\nabla_xu_x)u_x
&=
-\frac{1}{2}\,
g(\nabla_xu_x,u_x)J_uY
+\frac{1}{4}A_1Y
-\frac{1}{2}A_2Y.
\label{eq:ayaB}
\end{align}
Using ${}^{t}J_u=-J_u$ in $T_uN$, \eqref{eq:gsym}, and \eqref{eq:ayaB}, 
we deduce 
\begin{align}
g(Y,u_x)J_u\nabla_xu_x
&=
{}^{t}\left(
g(\cdot,J_u\nabla_xu_x)u_x
\right)Y
\nonumber
\\
&=
-\frac{1}{2}\,
g(\nabla_xu_x,u_x)\,{}^{t}J_uY
+\frac{1}{4}\,{}^tA_1Y
-\frac{1}{2}\,{}^tA_2Y
\nonumber
\\
&=
\frac{1}{2}\,
g(\nabla_xu_x,u_x)J_uY
+\frac{1}{4}A_1Y
-\frac{1}{2}A_2Y, 
\label{eq:ayaC}
\end{align}
and 
\begin{align}
g(Y,\nabla_xu_x)J_uu_x
&=
{}^{t}\left(
g(\cdot,J_uu_x)\nabla_xu_x
\right)Y
\nonumber
\\
&=
\frac{1}{2}\,
g(\nabla_xu_x,u_x)J_uY
+\frac{1}{4}A_1Y
+\frac{1}{2}A_2Y. 
\label{eq:ayaD}
\end{align}
Applying 
\eqref{eq:ayaA}, 
\eqref{eq:ayaB}, 
\eqref{eq:ayaC}, 
and \eqref{eq:ayaD} 
to the RHS of \eqref{eq:maya}, 
we derive 
\begin{align}
\nabla_t(\nabla_x^ku_x)
&=
-\ep\nabla_x^4(\nabla_x^ku_x)
+
\ep\,
\mathcal{O}
\left(
|\nabla_x^{k+2}u_x|_g
\right)
\nonumber
\\
&\quad 
+a\,J_u\nabla_x^4(\nabla_x^ku_x)
+\lambda\,J_u\nabla_x^2(\nabla_x^ku_x)
+b\,\nabla_x
\left\{
g(u_x,u_x)J_u\nabla_x(\nabla_x^ku_x)
\right\}
\nonumber
\\
&\quad 
+c_1\,g(\nabla_x^2(\nabla_x^ku_x), u_x)J_uu_x
+c_2
\,g(\nabla_xu_x,u_x)J_u\nabla_x(\nabla_x^ku_x)
\nonumber
\\
&\quad 
+c_3\,A_1\nabla_x(\nabla_x^ku_x)
+c_4\,A_2\nabla_x(\nabla_x^ku_x)
\nonumber
\\
&\quad
+\mathcal{O}
\left(
\sum_{m=0}^k
|\nabla_x^mu_x|_g
\right),
\label{eq:ayaya}
\end{align}
where $c_1,\ldots,c_4$ are constants given by 
$a,b,c$ and $S$. 
More concretely,   
\begin{align}
c_1&=aS+c, 
\label{eq:c1}
\\
c_2&=\left\{
aS(2k-1)+2kb+c
\right\}
+
\frac{1}{2}
\left\{
2b+(k+1)c
+(k+2)c-aSk-aS(k-1)
\right\}
\nonumber
\\
&=
\left(k-\frac{1}{2}\right)aS
+(2k+1)b
+\left(k+\frac{5}{2}\right)c. 
\label{eq:c2}
\end{align}
We omit to describe the explicit form of $c_3$ and $c_4$, 
as they will not be used later. 
\par 
We are now in position to evaluate the first term of 
the RHS of \eqref{eq:eqV_k}. 
Using \eqref{eq:ayaya}, 
we have 
\begin{align}
&\int_{\TT}
g(\nabla_t(\nabla_x^ku_x), \nabla_x^ku_x)dx
\nonumber
\\
&=
-\ep
\int_{\TT}
g(\nabla_x^4(\nabla_x^ku_x), \nabla_x^ku_x)dx
+
\ep\,
\int_{\TT}
g(
\mathcal{O}
\left(
|\nabla_x^{k+2}u_x|_g
\right), 
\nabla_x^ku_x)dx
\nonumber
\\
&\quad 
+a\,
\int_{\TT}
g(J_u\nabla_x^4(\nabla_x^ku_x),  \nabla_x^ku_x)dx
+\lambda\,
\int_{\TT}
g(J_u\nabla_x^2(\nabla_x^ku_x), \nabla_x^ku_x)dx
\nonumber
\\
&\quad 
+b\,
\int_{\TT}
g(\nabla_x
\left\{
g(u_x,u_x)J_u\nabla_x(\nabla_x^ku_x)
\right\}, 
\nabla_x^ku_x)dx
\nonumber
\\
&\quad 
+c_1\,
\int_{\TT}
g(
g(\nabla_x^2(\nabla_x^ku_x), u_x)J_uu_x, 
\nabla_x^ku_x)dx
\nonumber
\\
&\quad
+c_2
\,\int_{\TT}
g(
g(\nabla_xu_x,u_x)J_u\nabla_x(\nabla_x^ku_x), 
\nabla_x^ku_x)dx
\nonumber
\\
&\quad 
+c_3\,\int_{\TT}
g(A_1\nabla_x(\nabla_x^ku_x), \nabla_x^ku_x)dx
+c_4\,\int_{\TT}
g(A_2\nabla_x(\nabla_x^ku_x),  \nabla_x^ku_x)dx
\nonumber
\\
&\quad
+\int_{\TT}
g(\mathcal{O}
\left(
\sum_{m=0}^k
|\nabla_x^mu_x|_g
\right),
 \nabla_x^ku_x)dx. 
\nonumber
\end{align}
We compute each term of the above separately. 
By integrating by parts, we obtain 
\begin{align}
&
a\,
\int_{\TT}
g(J_u\nabla_x^4(\nabla_x^ku_x),  \nabla_x^ku_x)dx
=
a\,
\int_{\TT}
g(J_u\nabla_x^2(\nabla_x^ku_x),  \nabla_x^2(\nabla_x^ku_x))dx
=0, 
\nonumber
\\
&\lambda\,
\int_{\TT}
g(J_u\nabla_x^2(\nabla_x^ku_x), \nabla_x^ku_x)dx
=
-\lambda\,
\int_{\TT}
g(J_u\nabla_x(\nabla_x^ku_x), \nabla_x(\nabla_x^ku_x))dx
=0, 
\nonumber
\\
&b\,
\int_{\TT}
g(\nabla_x
\left\{
g(u_x,u_x)J_u\nabla_x(\nabla_x^ku_x)
\right\}, 
\nabla_x^ku_x)dx
\nonumber
\\
&
=
-b\,
\int_{\TT}
g(
g(u_x,u_x)J_u\nabla_x(\nabla_x^ku_x), 
\nabla_x(\nabla_x^ku_x))dx
=0. 
\nonumber
\end{align} 
By using the Cauchy-Schwartz inequality, 
we have 
\begin{align}
\int_{\TT}
g(\mathcal{O}
\left(
\sum_{m=0}^k
|\nabla_x^mu_x|_g
\right),
 \nabla_x^ku_x)dx
\leqslant 
C\|u_x\|_{H^k}\|\nabla_x^ku_x\|_{L^2}
\leqslant 
C\|u_x\|_{H^k}^2.
\end{align}
Using the integration by parts, the Young inequality 
$AB\leqslant A^2/2+B^2/2$ for any $A,B\geqslant 0$, 
and $\ep\leqslant 1$, 
we deduce 
\begin{align}
&-\ep
\int_{\TT}
g(\nabla_x^4(\nabla_x^ku_x), \nabla_x^ku_x)dx
+
\ep\,
\int_{\TT}
g(
\mathcal{O}
\left(
|\nabla_x^{k+2}u_x|_g
\right), 
\nabla_x^ku_x)dx
\nonumber
\\
&\leqslant
-\ep
\|\nabla_x^2(\nabla_x^ku_x))\|_{L^2}^2
+
\ep\,C
\|\nabla_x^2(\nabla_x^ku_x)\|_{L^2}
\|\nabla_x^ku_x\|_{L^2}
\nonumber
\\
&\leqslant 
-\ep
\|\nabla_x^2(\nabla_x^ku_x))\|_{L^2}^2
+
\frac{\ep}{2}
\|\nabla_x^2(\nabla_x^ku_x)\|_{L^2}^2
+
\frac{\ep\,C^2}{2}
\|\nabla_x^ku_x\|_{L^2}^2
\nonumber
\\
&\leqslant 
-
\frac{\ep}{2}
\|\nabla_x^2(\nabla_x^ku_x)\|_{L^2}^2
+
\frac{C^2}{2}
\|u_x\|_{H^k}^2.
\nonumber
\end{align}
By integrating by parts 
and by using \eqref{eq:gsym}, 
we have 
\begin{align}
&c_3\,\int_{\TT}
g(A_1\nabla_x(\nabla_x^ku_x), \nabla_x^ku_x)dx
+c_4\,\int_{\TT}
g(A_2\nabla_x(\nabla_x^ku_x),  \nabla_x^ku_x)dx
\nonumber
\\
&=
-\frac{c_3}{2}
g\left(\nabla_x(A_1)\nabla_x^ku_x, \nabla_x^ku_x\right)dx
-\frac{c_4}{2}
g\left(\nabla_x(A_2)\nabla_x^ku_x, \nabla_x^ku_x\right)dx
\nonumber
\\
&\leqslant 
C\|u_x\|_{H^k}^2. 
\nonumber
\end{align}
Collecting them and noting that 
$\|u_x\|_{H^k}\leqslant CN_k(u)$ 
follows from \eqref{eq:timecut}, 
we derive 
\begin{align}
&\int_{\TT}
g(\nabla_t(\nabla_x^ku_x), \nabla_x^ku_x)dx
\nonumber
\\
&\leqslant 
-\frac{\ep}{2}\|\nabla_x^2(\nabla_x^ku_x)\|_{L^2}^2
%\nonumber
%\\
%&\quad 
+c_1\,
\int_{\TT}
g(
g(\nabla_x^2(\nabla_x^ku_x), u_x)J_uu_x, 
\nabla_x^ku_x)dx
\nonumber
\\
&\quad
+c_2
\,\int_{\TT}
g(
g(\nabla_xu_x,u_x)J_u\nabla_x(\nabla_x^ku_x), 
\nabla_x^ku_x)dx
+C\,(N_k(u))^2. 
\label{eq:V1}
\end{align}
\par
We next evaluate the second term of the RHS of \eqref{eq:eqV_k}. 
In the computation, it is to be noted that 
$\Lambda=
\mathcal{O}(|\nabla_x^{k-2}u_x|_g)$ 
and 
$$
\nabla_t(\nabla_x^ku_x)=
-\ep\,\nabla_x^4(\nabla_x^ku_x)
+
a\,J_u\nabla_x^4(\nabla_x^ku_x)
+
\mathcal{O}\left(
\sum_{m=0}^{k+2}
|\nabla_x^mu_x|_g
\right). 
$$
By noting them and  by integrating by parts, we obtain  
\begin{align}
&\int_{\TT}
g(\nabla_t(\nabla_x^ku_x), \Lambda)dx
\nonumber
\\
&\leqslant 
-\ep\,
\int_{\TT}
g(\nabla_x^4(\nabla_x^ku_x), \Lambda)dx
+
a\,
\int_{\TT}
g(J_u\nabla_x^4(\nabla_x^ku_x), \Lambda)dx
+
C\|u_x\|_{H^k}^2. 
\label{eq:V20}
\end{align}
For the first term of the RHS of \eqref{eq:V20}, 
by using $\ep\leqslant 1$, the integration by parts, 
the Young inequality 
$AB\leqslant A^2/8+2B^2$ for any $A,B\geqslant 0$, and 
$\Lambda=\mathcal{O}(|\nabla_x^{k-2}u_x|_g)$, 
we have 
\begin{align}
-\ep\,
\int_{\TT}
g(\nabla_x^4(\nabla_x^ku_x), \Lambda)dx
&=
-\ep\,
\int_{\TT}
g(\nabla_x^2(\nabla_x^ku_x), \nabla_x^2(\Lambda))dx
\nonumber
\\
&\leqslant 
\ep\|\nabla_x^2(\nabla_x^ku_x)\|_{L^2}\| \nabla_x^2(\Lambda))\|_{L^2}
\nonumber
\\
&\leqslant
\frac{\ep}{8}
\|\nabla_x^2(\nabla_x^ku_x)\|_{L^2}^2
+
2\ep\| \nabla_x^2(\Lambda))\|_{L^2}^2
\nonumber
\\
&\leqslant 
\frac{\ep}{8}
\|\nabla_x^2(\nabla_x^ku_x)\|_{L^2}^2
+C\|u_x\|_{H^k}^2. 
\label{eq:V21}
\end{align}
For the second term of the RHS of \eqref{eq:V20}, 
we compute $\nabla_x^2\Lambda$ to see 
\begin{align}
\nabla_x^2\Lambda
%&=
%\nabla_x^2\Lambda_1
%+
%\nabla_x^2\Lambda_2
%\nonumber
%\\
&=
-\frac{d_1}{2a}\nabla_x^2
\left\{
g(\nabla_x^{k-2}u_x,J_uu_x)J_uu_x
\right\}
+
\frac{d_2}{8a}
\nabla_x^2\left\{
g(u_x,u_x)\nabla_x^{k-2}u_x
\right\}
\nonumber
\\
&=
-\frac{d_1}{2a}
g(\nabla_x^ku_x,J_uu_x)J_uu_x
+
\frac{d_2}{8a}
g(u_x,u_x)\nabla_x^ku_x
\nonumber
\\
&\quad 
-\frac{d_1}{a}
g(\nabla_x^{k-1}u_x,J_u\nabla_xu_x)J_uu_x 
-\frac{d_1}{a}
g(\nabla_x^{k-1}u_x,J_uu_x)J_u\nabla_xu_x
\nonumber
\\
&\quad 
+\frac{d_2}{2a}
g(\nabla_xu_x,u_x)\nabla_x^{k-1}u_x
+
\mathcal{O}
\left(
\sum_{m=0}^{k-2}
|\nabla_x^mu_x|_g
\right). 
\nonumber
\end{align}
Thus, by integrating by parts 
and by substituting the above, 
we obtain 
\begin{align}
&a\,
\int_{\TT}
g(J_u\nabla_x^4(\nabla_x^ku_x), \Lambda)dx
\nonumber
\\
&=
a\,
\int_{\TT}
g(J_u\nabla_x^2(\nabla_x^ku_x), \nabla_x^2\Lambda)dx
\nonumber
\\
&= 
-\frac{d_1}{2}Q_{2,1}
+\frac{d_2}{8}Q_{2,2}
-d_1Q_{2,3}-d_1Q_{2,4}
+\frac{d_2}{2}Q_{2,5}
+Q_{2,6}, 
\label{eq:Q2}
\end{align}
where 
\begin{align}
Q_{2,1}
&=
\int_{\TT}
g(\nabla_x^ku_x,J_uu_x)
g(J_u\nabla_x^2(\nabla_x^ku_x),J_uu_x)\,dx,
\nonumber
\\
Q_{2,2}
&=
\int_{\TT}
g(u_x,u_x)
g(J_u\nabla_x^2(\nabla_x^ku_x), \nabla_x^ku_x)\,dx,
\nonumber
\\
Q_{2,3}
&=
\int_{\TT}
g(\nabla_x^{k-1}u_x,J_u\nabla_xu_x)
g(J_u\nabla_x^2(\nabla_x^ku_x),J_uu_x)\,dx, 
\nonumber
\\
Q_{2,4}
&=
\int_{\TT}
g(\nabla_x^{k-1}u_x,J_uu_x)
g(J_u\nabla_x^2(\nabla_x^ku_x),J_u\nabla_xu_x)\,dx,
\nonumber
\\
Q_{2,5}
&=
\int_{\TT}
g(\nabla_xu_x,u_x)
g(J_u\nabla_x^2(\nabla_x^ku_x), \nabla_x^{k-1}u_x)\,dx,
\nonumber
\\
Q_{2,6}
&=
\int_{\TT}
g(J_u\nabla_x^2(\nabla_x^ku_x), 
\mathcal{O}
\left(
\sum_{m=0}^{k-2}
|\nabla_x^mu_x|_g
\right)
)\,dx.
\nonumber
\end{align}
We compute $Q_{2,1},\ldots,Q_{2,6}$ separately.
By the integration by parts and the property of hermitian metric $g$, 
we deduce 
\begin{align}
Q_{2,1}
&=
\int_{\TT}
g(\nabla_x^ku_x,J_uu_x)
g(\nabla_x^2(\nabla_x^ku_x),u_x)\,dx,
\nonumber
\\
&=
\int_{\TT}
g(g(\nabla_x^2(\nabla_x^ku_x), u_x)J_uu_x, \nabla_x^ku_x)\,dx,
\nonumber
\\
Q_{2,2}
&=
\int_{\TT}
g(
\nabla_x\left\{
g(u_x,u_x)J_u\nabla_x(\nabla_x^ku_x)
\right\}, 
\nabla_x^ku_x
)\,dx
\nonumber
\\
&\quad 
-2\,\int_{\TT}
g(
g(\nabla_xu_x,u_x)J_u\nabla_x(\nabla_x^ku_x), 
\nabla_x^ku_x
)\,dx
\nonumber
\\
&=
-2\,\int_{\TT}
g(
g(\nabla_xu_x,u_x)J_u\nabla_x(\nabla_x^ku_x), 
\nabla_x^ku_x)\,dx, 
\nonumber
\\
Q_{2,3}
&=
\int_{\TT}
g(\nabla_x^{k-1}u_x,J_u\nabla_xu_x)
g(\nabla_x^2(\nabla_x^ku_x),u_x)\,dx
\nonumber
\\
&\leqslant 
-\int_{\TT}
g(\nabla_x^{k}u_x,J_u\nabla_xu_x)
g(\nabla_x(\nabla_x^ku_x),u_x)\,dx
+C\|u_x\|_{H^k}^2
\nonumber
\\
&=
-\int_{\TT}
g(
g(\nabla_x(\nabla_x^ku_x), u_x)J_u\nabla_xu_x, \nabla_x^ku_x
)\,dx
+C\|u_x\|_{H^k}^2, 
\nonumber
\\
Q_{2,4}
&=
\int_{\TT}
g(\nabla_x^{k-1}u_x,J_uu_x)
g(\nabla_x^2(\nabla_x^ku_x),\nabla_xu_x)\,dx
\nonumber
\\
&\leqslant 
-\int_{\TT}
g(\nabla_x^{k}u_x,J_uu_x)
g(\nabla_x(\nabla_x^ku_x),\nabla_xu_x)\,dx
+C\|u_x\|_{H^k}^2
\nonumber
\\
&=
-\int_{\TT}
g(
g(\nabla_x(\nabla_x^ku_x), \nabla_xu_x)J_uu_x, \nabla_x^ku_x
)\,dx
+C\|u_x\|_{H^k}^2, 
\nonumber
\\
Q_{2,5}
&\leqslant 
-\int_{\TT}
g(
g(\nabla_xu_x,u_x)
J_u\nabla_x(\nabla_x^ku_x), 
\nabla_x^ku_x
)\,dx
+C\|u_x\|_{H^k}^2, 
\nonumber
\\
Q_{2,6}
&\leqslant 
C\|u_x\|_{H^k}^2. 
\nonumber
\end{align} 
Applying them to \eqref{eq:Q2}, 
we obtain 
\begin{align}
a\,
\int_{\TT}
g(J_u\nabla_x^4(\nabla_x^ku_x), \Lambda)dx
&\leqslant
-\frac{d_1}{2}
\int_{\TT}
g(
g(\nabla_x^2(\nabla_x^ku_x), u_x)J_uu_x, \nabla_x^ku_x
)\,dx
\nonumber
\\
&\quad
-\frac{3d_2}{4}
\int_{\TT}
g(
g(\nabla_xu_x,u_x)
J_u\nabla_x(\nabla_x^ku_x), 
\nabla_x^ku_x
)\,dx
\nonumber
\\
&\quad 
+d_1\,
\int_{\TT}
g(
g(\nabla_x(\nabla_x^ku_x), u_x)J_u\nabla_xu_x, \nabla_x^ku_x
)\,dx
\nonumber
\\
&\quad 
+d_1\,
\int_{\TT}
g(
g(\nabla_x(\nabla_x^ku_x), \nabla_xu_x)J_uu_x, \nabla_x^ku_x
)\,dx
\nonumber
\\
&\quad 
+C\|u_x\|_{H^k}^2.
\label{eq:ri1}
\end{align}
Here we rewrite the sum of the third and the fourth term of the RHS
by using \eqref{eq:ayaC} and \eqref{eq:ayaD},  
and use the integration by parts and \eqref{eq:gsym}
to find 
\begin{align}
&d_1\,
\int_{\TT}
g(
g(\nabla_x(\nabla_x^ku_x), u_x)J_u\nabla_xu_x, \nabla_x^ku_x
)\,dx
\nonumber
\\
&\quad 
+d_1\,
\int_{\TT}
g(
g(\nabla_x(\nabla_x^ku_x), \nabla_xu_x)J_uu_x, \nabla_x^ku_x
)\,dx
\nonumber
\\
&=
d_1\,\int_{\TT}
g(g(\nabla_xu_x,u_x)J_u\nabla_x(\nabla_x^ku_x),\nabla_x^ku_x)\,dx
+
\frac{d_1}{2}
\int_{\TT}
g(A_1\nabla_x(\nabla_x^ku_x),\nabla_x^ku_x)\,dx
\nonumber
\\
&\leqslant 
d_1\,\int_{\TT}
g(g(\nabla_xu_x,u_x)J_u\nabla_x(\nabla_x^ku_x),\nabla_x^ku_x)\,dx
+
C\|u_x\|_{H^k}^2.
\label{eq:ri2}
\end{align} 
Combining \eqref{eq:ri1} and \eqref{eq:ri2}, we have 
\begin{align}
a\,
\int_{\TT}
g(J_u\nabla_x^4(\nabla_x^ku_x), \Lambda)dx
&\leqslant
-\frac{d_1}{2}
\int_{\TT}
g(
g(\nabla_x^2(\nabla_x^ku_x), u_x)J_uu_x, \nabla_x^ku_x
)\,dx
\nonumber
\\
&\quad
+\left(d_1-\frac{3d_2}{4}\right)
\int_{\TT}
g(
g(\nabla_xu_x,u_x)
J_u\nabla_x(\nabla_x^ku_x), 
\nabla_x^ku_x
)\,dx
\nonumber
\\
&\quad 
+C\|u_x\|_{H^k}^2.
\label{eq:V22}
\end{align}
Therefore, 
from  
\eqref{eq:V20}, 
\eqref{eq:V21}, 
and \eqref{eq:V22}, 
it follows that 
\begin{align}
&\int_{\TT}
g(\nabla_t(\nabla_x^ku_x), \Lambda)\,dx
\nonumber
\\
&\leqslant 
\frac{\ep}{8}
\|\nabla_x^2(\nabla_x^ku_x)\|_{L^2}^2
-\frac{d_1}{2}
\int_{\TT}
g(
g(\nabla_x^2(\nabla_x^ku_x), u_x)J_uu_x, \nabla_x^ku_x
)\,dx
\nonumber
\\
&\quad
+\left(d_1-\frac{3d_2}{4}\right)
\int_{\TT}
g(
g(\nabla_xu_x,u_x)
J_u\nabla_x(\nabla_x^ku_x), 
\nabla_x^ku_x
)\,dx
+C\|u_x\|_{H^k}^2.
\label{eq:V2e}
\end{align}
\par 
We next evaluate the third term of the RHS of \eqref{eq:eqV_k}. 
For this purpose, we compute 
$\nabla_t\Lambda$. 
Using the product formula and noting 
$\nabla_tu_x=\nabla_xu_t
=\mathcal{O}
\left(
\displaystyle 
\sum_{m=0}^4|\nabla_x^mu_x|_g
\right),
$ 
we have  
\begin{align}
\nabla_t\Lambda
&=
-\frac{d_1}{2a}
g(\nabla_t\nabla_x^{k-2}u_x,J_uu_x)J_uu_x
-\frac{d_1}{2a}
g(\nabla_x^{k-2}u_x,J_u\nabla_tu_x)J_uu_x
\nonumber
\\
&\quad
-\frac{d_1}{2a}
g(\nabla_x^{k-2}u_x,J_uu_x)J_u\nabla_tu_x
+
\frac{d_2}{8a}g(u_x,u_x)\nabla_t\nabla_x^{k-2}u_x
\nonumber
\\
&\quad 
+\frac{d_2}{4a}g(\nabla_xu_t,u_x)\nabla_x^{k-2}u_x
\nonumber
\\
&=
-\frac{d_1}{2a}
g(\nabla_t\nabla_x^{k-2}u_x,J_uu_x)J_uu_x
+
\frac{d_2}{8a}g(u_x,u_x)\nabla_t\nabla_x^{k-2}u_x
\nonumber
\\
&\quad 
+
\mathcal{O}
\left(
|\nabla_x^{k-2}u_x|_g 
\sum_{m=0}^4|\nabla_x^mu_x|_g
\right).
\nonumber
\end{align}
Thus, we have 
\begin{align}
\int_{\TT}
g(\nabla_t\Lambda,V_k)\,dx
&=
Q_{3,1}+Q_{3,2}+Q_{3,2}, 
\nonumber
\end{align}
where, 
\begin{align}
Q_{3,1}
&=
-\frac{d_1}{2a}
\int_{\TT}
g(g(\nabla_t\nabla_x^{k-2}u_x, J_uu_x)J_uu_x,V_k)\,dx,
\nonumber
\\
Q_{3,2}
&=
\frac{d_2}{8a}
\int_{\TT}
g(g(u_x,u_x)\nabla_t\nabla_x^{k-2}u_x,V_k)\,dx, 
\nonumber
\\
Q_{3,3}
&=
\int_{\TT}
g(
\mathcal{O}
\left(
|\nabla_x^{k-2}u_x|_g 
\sum_{m=0}^4|\nabla_x^mu_x|_g
\right),V_k)\,dx. 
\nonumber
\end{align}
For $Q_{3,3}$, 
since $k\geqslant 4$,  
we use the Sobolev embedding and the Cauchy-Schwartz inequality
to obtain  
\begin{align}
Q_{3,3}
&\leqslant 
C\|u_x\|_{H^k}^2. 
\label{eq:Q33}
\end{align}
For $Q_{3,1}$ and $Q_{3,2}$, 
we need to compute $\nabla_t\nabla_x^{k-2}u_x$. 
Indeed, 
by the same computation as that we obtain $\nabla_t(\nabla_x^{k}u_x)$, 
we find 
\begin{align}
\nabla_t\nabla_x^{k-2}u_x
&=
-\ep\nabla_x^4(\nabla_x^{k-2}u_x)
+a\,J_u\nabla_x^4(\nabla_x^{k-2}u_x)
+
\mathcal{O}
\left(
\sum_{m=0}^k
|\nabla_x^mu_x|_g
\right)
\nonumber
\\
&=
\ep\,
\mathcal{O}
(|\nabla_x^{k+2}u_x|_g)
+a\,J_u\nabla_x^2(\nabla_x^{k}u_x)
+
\mathcal{O}
\left(
\sum_{m=0}^k
|\nabla_x^mu_x|_g
\right).
\label{eq:k-2}
\end{align}
Applying \eqref{eq:k-2}, we deduce 
\begin{align}
Q_{3,1}
&=
-\frac{d_1}{2a}
\int_{\TT}
g(g(\nabla_t\nabla_x^{k-2}u_x, J_uu_x)J_uu_x,\nabla_x^ku_x+\Lambda)\,dx
\nonumber
\\
&\leqslant 
-\frac{d_1}{2a}
\int_{\TT}
g(g(\nabla_t\nabla_x^{k-2}u_x, J_uu_x)J_uu_x,\nabla_x^ku_x)\,dx
+C\|u_x\|_{H^k}^2
\nonumber
\\
&\leqslant 
\ep\,
\int_{\TT}
g(\mathcal{O}(|\nabla_x^{k+2}u_x|_g),\nabla_x^ku_x)\,dx
\nonumber
\\
&\quad
-\frac{d_1}{2}
\int_{\TT}
g(g(J_u\nabla_x^2(\nabla_x^ku_x), J_uu_x)J_uu_x,\nabla_x^ku_x)\,dx 
+C\|u_x\|_{H^k}^2
\nonumber
\\
&\leqslant 
\frac{\ep}{8}\|\nabla_x^2(\nabla_x^ku_x)\|_{L^2}^2
-\frac{d_1}{2}
\int_{\TT}
g(g(\nabla_x^2(\nabla_x^ku_x), u_x)J_uu_x,\nabla_x^ku_x)\,dx 
+C\|u_x\|_{H^k}^2.
\label{eq:Q31}
\end{align}
In the same way, applying \eqref{eq:k-2}, we deduce  
\begin{align}
Q_{3,2}
&=
\frac{d_2}{8a}
\int_{\TT}
g(g(u_x,u_x)\nabla_t\nabla_x^{k-2}u_x,\nabla_x^ku_x+\Lambda)\,dx 
\nonumber
\\
&\leqslant 
\frac{d_2}{8a}
\int_{\TT}
g(g(u_x,u_x)\nabla_t\nabla_x^{k-2}u_x,\nabla_x^ku_x)\,dx
+C\|u_x\|_{H^k}^2
\nonumber
\\
&\leqslant 
\ep\,
\int_{\TT}
g(\mathcal{O}(|\nabla_x^{k+2}u_x|_g),\nabla_x^ku_x)\,dx
\nonumber
\\
&\quad 
+
\frac{d_2}{8}
\int_{\TT}
g(g(u_x,u_x)J_u\nabla_x^2(\nabla_x^ku_x),\nabla_x^ku_x)\,dx
+C\|u_x\|_{H^k}^2  
\nonumber 
\\
&\leqslant 
\frac{\ep}{8}\|\nabla_x^2(\nabla_x^ku_x)\|_{L^2}^2 
+
\frac{d_2}{8}
\int_{\TT}
g(\nabla_x\left\{g(u_x,u_x)J_u\nabla_x(\nabla_x^ku_x)\right\},
\nabla_x^ku_x)\,dx
\nonumber
\\
&\quad 
-\frac{d_2}{4}
\int_{\TT}
g(g(\nabla_xu_x,u_x)J_u\nabla_x(\nabla_x^ku_x), \nabla_x^ku_x)\,dx
+C\|u_x\|_{H^k}^2 
\nonumber
\\
&=
\frac{\ep}{8}\|\nabla_x^2(\nabla_x^ku_x)\|_{L^2}^2
-\frac{d_2}{4}
\int_{\TT}
g(g(\nabla_xu_x,u_x)J_u\nabla_x(\nabla_x^ku_x), \nabla_x^ku_x)\,dx
\nonumber
\\
&\quad
+C\|u_x\|_{H^k}^2. 
\label{eq:Q32}
\end{align}
Collecting 
\eqref{eq:Q33}, 
\eqref{eq:Q31}, 
and 
\eqref{eq:Q32}, 
we obtain 
\begin{align}
&\int_{\TT}
g(\nabla_t\Lambda, V_k)\,dx
\nonumber
\\
&\leqslant 
\frac{\ep}{4}\|\nabla_x^2(\nabla_x^ku_x)\|_{L^2}^2
-\frac{d_1}{2}
\int_{\TT}
g(g(\nabla_x^2(\nabla_x^ku_x), u_x)J_uu_x,\nabla_x^ku_x)\,dx
\nonumber
\\
&\quad 
-\frac{d_2}{4}
\int_{\TT}
g(g(\nabla_xu_x,u_x)J_u\nabla_x(\nabla_x^ku_x), \nabla_x^ku_x)\,dx
+C\|u_x\|_{H^k}^2. 
\label{eq:V3e}
\end{align}
\par 
Consequently, collecting the information 
\eqref{eq:eqV_k}, 
\eqref{eq:V1}, 
\eqref{eq:V2e}, 
and 
\eqref{eq:V3e}, 
we derive 
\begin{align}
\frac{1}{2}
\frac{d}{dt}
\|V_k\|_{L^2}^2
&\leqslant 
-\frac{\ep}{8}\|\nabla_x^2(\nabla_x^ku_x)\|_{L^2}^2
+(c_1-d_1)
\int_{\TT}
g(g(\nabla_x^2(\nabla_x^ku_x), u_x)J_uu_x,\nabla_x^ku_x)\,dx
\nonumber
\\
&\quad 
+(c_2+d_1-d_2)
\int_{\TT}
g(g(\nabla_xu_x,u_x)J_u\nabla_x(\nabla_x^ku_x), \nabla_x^ku_x)\,dx
\nonumber
\\
&\quad 
+C\|u_x\|_{H^k}^2
+C(N_k(u))^2, 
\nonumber
\end{align}
where $c_1$ and $c_2$ are given by 
\eqref{eq:c1} and \eqref{eq:c2}. 
To cancell the second and the third term of the RHS of above, 
we set $d_1$ and $d_2$ so that 
\begin{align}
d_1&=c_1=aS+c, 
\nonumber
\\
d_2&=c_2+d_1
=
\left(k+\frac{1}{2}\right)aS
+(2k+1)b
+\left(k+\frac{7}{2}\right)c.
\nonumber
\end{align}
Therefore, using $\|u_x\|_{H^k}\leqslant CN_k(u)$,  
we conclude that  
\begin{align}
\frac{1}{2}
\frac{d}{dt}
\|V_k\|_{L^2}^2
&\leqslant 
-\frac{\ep}{8}\|\nabla_x^2(\nabla_x^ku_x)\|_{L^2}^2
+C(N_k(u))^2
\label{eq:VVe}
\end{align}
holds on $[0,T_{\ep}^{\star}]$
\par 
Let us now go back 
to the original purpose to derive the uniform estimate for 
$\left\{N_k(u^{\ep})\right\}_{\ep\in (0,1]}$. 
To achieve this,
it remains to consider 
the energy estimate for $\|u_x^{\ep}\|_{H^{k-1}}^2$. 
However, 
by using the integration by parts, the Sobolev embedding, 
and the Cauchy-Schwartz inequality repeatedly, 
we can easily show that    
\begin{align}
\frac{1}{2}
\frac{d}{dt}
\|u_x^{\ep}\|_{H^{k-1}}^2
&\leqslant
-\frac{\ep}{2}\sum_{m=0}^{k-1}
\|\nabla_x^{m+2}u_x^{\ep}\|_{L^2}^2
+ 
C\,(N_k(u^{\ep}))^2.   
\label{eq:k-1}
\end{align} 
Therefore, from \eqref{eq:VVe} and \eqref{eq:k-1}, 
we conclude that 
there exits a positive constant $C$
depending on $a,b,c,k,\lambda, S, \|u_{0x}\|_{H^4}$ 
and not on $\ep$ such that 
\begin{align}
\frac{d}{dt}(N_k(u^{\ep}))^2
&=
\frac{d}{dt}\left(
\|u_x^{\ep}\|_{H^{k-1}}^2
+
\|V_k^{\ep}\|_{L^2}^2
\right)
\leqslant 
C(N_k(u^{\ep}))^2
\nonumber
\end{align}
on the time-interval $[0,T_{\ep}^{\star}]$. 
This implies 
$(N_k(u^{\ep}(t)))^2
\leqslant 
(N_k(u_0))^2
e^{Ct}
$
for any $t\in [0,T_{\ep}^{\star}]$. 
Thus, by the definition of $T_{\ep}^{\star}$, 
there holds  
\begin{align}
4(N_4(u_0))^2
&=
(N_4(u^{\ep}(T_{\ep}^{\star})))^2
\leqslant 
(N_4(u_0))^2
e^{C_4T_{\ep}^{\star}} 
\nonumber
\end{align}
with $C_4>0$ which depends on 
 $a,b,c,\lambda, S, \|u_{0x}\|_{H^4}$ 
and not on $\ep$. 
This shows $e^{C_4T_{\ep}^{\star}}\geqslant 4$ 
and hence 
$T_{\ep}^{\star}\geqslant \log 4/C_4$ holds. 
Therefore, if we set $T=\log 4/C_4$, 
it follows that 
$T_{\ep}^{\star}\geqslant T$ for any $\ep\in (0,1]$ 
and 
$\left\{
N_k(u^{\ep})
\right\}_{\ep\in (0,1]}$   
is bounded in $L^{\infty}(0,T)$. 
\par
As stated before, this shows that $\left\{u_x\right\}_{\ep\in(0,1]}$ 
is bounded in $L^{\infty}(0,T;H^k(\TT;TN))$. 
Hence the standard compactness argument and the compactness of $N$ 
show the existence of a map $u\in C([0,T]\times \TT;N)$ 
and a subsequence $\left\{u^{\ep(j)}\right\}_{j=1}^{\infty}$ of 
$\left\{u^{\ep}\right\}_{\ep\in (0,1]}$
that satisfy 
\begin{alignat}{3}
&u_x^{\ep(j)}\to u_x
\quad
&\text{in}
\quad 
&C([0,T];H^{k-1}(\TT;TN)), 
\nonumber
\\
&u_x^{\ep(j)}\to u_x
\quad
&\text{in}
\quad 
&L^{\infty}(0,T;H^{k}(\TT;TN)) 
\quad
\text{weakly star}
\nonumber
\end{alignat}
as $j\to \infty$, 
and this $u$ is smooth and solves \eqref{eq:pde}-\eqref{eq:data}. 
\par 
Finally, in the general case where 
$u_0\in C(\TT;N)$ and $u_{0x}\in H^k(\TT;TN)$, 
we modify the above argument as follows:  
We take a sequence 
$\left\{u_{0}^i\right\}_{i=1}^{\infty}\subset C^{\infty}(\TT;N)$ 
such that  
\begin{equation}
u_{0x}^i
\to
u_{0x}
\quad 
\text{in}
\quad 
H^{k}(\TT;TN)
\label{eq:dense}
\end{equation}
as $i\to \infty$.  
There exist $T_i=T(\|u_{0x}^i\|_{H^4})>0$
and  
$u^i\in C^{\infty}([0,T_i]\times \TT;N)$ which satisfies 
\eqref{eq:pde} and 
$u^i(0,x)=u^i(x)$ for each $i=1,2,\ldots$, 
since $u_0^i\in C^{\infty}(\TT;N)$. 
Recalling the above argument, it is not difficult to show 
the estimate $T_i^{\star}\geqslant \log 4/C_{4,i}$
where 
$$
T^{\star}_{i}
=
\sup
\left\{
T>0 \ | \ 
N_4(u^{i}(t))\leqslant 2N_4(u_{0}^i)
\quad 
\text{for all}
\quad
t\in[0,T]
\right\}, 
$$
and $C_{4,i}>0$ depends on 
 $a,b,c,\lambda,S, \|u_{0x}^i\|_{H^4}$. 
Note that $C_{4,i}$ depends on $\|u_{0x}^i\|_{H^4}$ continuously. 
This together with \eqref{eq:dense} shows that 
there exists $C_4^{\prime}>0$ depending on 
 $a,b,c,\lambda, S, \|u_{0x}\|_{H^4}$ and not on $i$ 
such that 
$T_i^{\star}\geqslant \log 4/C_{4}^{\prime}$ for 
sufficient large $i$. 
Therefore, if we set $T=\log 4/C_4^{\prime}$,
there exists a sufficiently large $i_0\in \mathbb{N}$ 
such that   
$T_{i}^{\star}\geqslant T$ for any $i\geqslant i_0$ 
and 
$\left\{
N_k(u^{i})
\right\}_{i=i_0}^{\infty}$   
is bounded in $L^{\infty}(0,T)$. 
Therefore, by applying the compactness argument again, 
we can construct the desired 
solution to \eqref{eq:pde}-\eqref{eq:data}. 
This completes the proof. 
\end{proof}
%%%%%%%%%%%%%%%%%%%%%%%%%%%%%%%%%%%%%%
%%%%%%%%%%%%%%%%%%%%%%%%%%%%%%
%%%%%%%%%%%%%%%%%%%%%%
%%%%%%%%%%%%%
%%%%%%%
%%
%
%
%uniqueness
%
\section{Proof of Theorem~\ref{theorem:uniqueness}}
\label{section:proof}
The goal of this section is to complete the proof of
Theorem~\ref{theorem:uniqueness}. 
Throughout this section, it is assumed that $k\geqslant 6$. 
\begin{proof}[Proof of Theorem~\ref{theorem:uniqueness}]
Since $k\geqslant 6\geqslant 4$, 
Theorem~\ref{theorem:existence} established in Section~\ref{section:existence}
guarantees the existence of 
$T=T(\|u_{0x}\|_{H^4(\TT;TN)})>0$
and a map $u\in C([0,T]\times \TT;N)$ 
so that 
$u_x\in L^{\infty}(0,T;H^k(\TT;TN))
\cap
C([0,T];H^{k-1}(\TT;TN))$
and $u$ solves \eqref{eq:pde}-\eqref{eq:data} on the time-interval
$[0,T]$. 
In what follows, we shall concentrate on the proof of 
the uniqueness of the solution. 
Once the uniqueness is established, 
we can easily prove the time-continuity of $\nabla_x^ku_x$ in $L^2$
by the standard argument, which implies 
$u_x\in C([0,T];H^k(\TT;TN))$. 
In this way, the proof of Theorem~\ref{theorem:uniqueness} is completed.
\par 
Let $u,v$ be solutions constructed in Theorem~\ref{theorem:existence}. 
Then 
$u$ and $v$
solve \eqref{eq:pde}-\eqref{eq:data} 
and satisfy  
$u_x, v_x 
\in L^{\infty}(0,T;H^6(\TT;TN))
\cap
C([0,T];H^{5}(\TT;TN))$. 
We shall show $u=v$. 
For this purpose, 
fix $w$ as an isometric embedding of $(N,g)$ into 
some Euclidean space $\RR^d$ 
so that $N$ is considered as a submanifold of $\RR^d$. 
We set 
$U=w{\circ}u$, 
$V=w{\circ}v$, 
$Z=U-V$, 
$\UU=dw_u(\nabla_xu_x)$, 
$\VV=dw_v(\nabla_xv_x)$, 
and $\WW=\UU-\VV$. 
To prove $u=v$, 
it suffices to show $Z=0$. 
The proof of $Z=0$ consists of the following four steps:
\begin{enumerate}
\item[1.] Notations and tools of computations used below.
\item[2.] Analysis of the partial differential equation 
satisfied by $\UU$. 
\item[3.] Classical energy estimates for 
$\|\WW\|_{L^2(\TT;\RR^d)}$ with the loss of derivatives.
\item[4.] Energy estimates for 
$\|\TW\|_{L^2(\TT;\RR^d)}$ (defined later) 
to eliminate the loss of derivatives
\end{enumerate}
\vspace{0.3em}
\par
{\bf 1. Notations and tools of computations used below.}
\\
We state some notations and gather tools of computations 
which will be used below.
\par
The inner product and the norm in $\RR^d$
is expressed by  
$(\cdot,\cdot)$ and $|\cdot|$ respectively.  
The inner product and the norm in $L^2$
for $\RR^d$-valued functions on $\TT$
is expressed
by
$\lr{\cdot}$ 
and 
$\|\cdot\|_{L^2}$ respectively. 
That is, 
for $\phi, \psi:\TT\to \RR^d$, 
$\lr{\phi,\psi}$ and $\|\phi\|_{L^2}$ 
is given by 
$
\lr{\phi,\psi}
=
\int_{\TT}
(\phi(x),\psi(x))
\,dx$ 
and  
$\|\phi\|_{L^2}
=\sqrt{\lr{\phi,\phi}}
$
respectively. 
\vspace{0.3em}
\par 
Let $p\in N$ be a fixed point.  
We consider the orthogonal decomposition  
$
\RR^d
=
dw_p(T_pN)
\oplus
\left(
dw_p(T_pN)
\right)^{\perp}
$, 
where 
$dw_p:T_pN\to T_{w{\circ}p}\RR^d\cong \RR^d$ is 
the differential of $w:N\to \RR^d$ at $p\in N$ 
and 
$\left(
dw_p(T_pN)
\right)^{\perp}$
is the orthogonal complement of 
$dw_p(T_pN)$ in $\RR^d$.     
We denote the orthogonal projection mapping 
of $\RR^d$ onto $dw_p(T_pN)$ by $P(w{\circ}p)$ 
and define $N(w{\circ}p)$ by 
$N(w{\circ}p)=I_d-P(w{\circ}p)$, where 
$I_d$ is the identity mapping on $\RR^d$. 
Moreover, we define $J(w{\circ}p)$  
as an action on $\RR^d$ by first projecting onto 
$dw_p(T_pN)$ 
and then applying the complex structure at $p\in N$. 
More precisely, we define $J(w{\circ}p)$ by 
\begin{align}
J(w{\circ}p)&=
dw_p\circ J_{p}\circ dw_{w{\circ}p}^{-1}\circ P(w{\circ}p). 
\label{eq:J(p)}
\end{align} 
We can extend $P(\cdot)$, $N(\cdot)$, and $J(\cdot)$ 
to a smooth linear operator on $\RR^d$ so that 
$P(q)$, $N(q)$, and $J(q)$ make sense for all $q\in \RR^d$
following the argument in e.g. \cite[pp.17]{NSVZ}. 
Though $J(q)$ is not skew-symmetric and the square is not the minus 
identity in general, 
similar properties hold if $q$ is restricted to $w(N)$.
Indeed, from the definition of $P(\omega{\circ}p)$ and 
$J(\omega{\circ}p)$, it follows that  
\begin{align}
(P(w{\circ}p)Y_1,Y_2)&=(Y_1,P(w{\circ}p)Y_2),  
\label{eq:J0}
\\
(J(w{\circ}p)Y_1,Y_2)&=-(Y_1,J(w{\circ}p)Y_2),   
\label{eq:J(wp)1}
\\
\left(
J(w{\circ}p)
\right)^2Y_3
&=
-P(w{\circ}p)Y_3, 
\label{eq:J(wp)2}
\end{align}
for any $p\in N$ and $Y_1,Y_2\in \RR^d$.
\par  
Let $Y\in \Gamma(u^{-1}TN)$ be fixed. 
For $(t,x)\in [0,T]\times \TT$, 
let $\left\{\nu_3, \ldots, \nu_d\right\}$ denote a smooth local orthonormal 
frame field for the normal bundle $(dw(TN))^{\perp}$
near $U(t,x)=w{\circ}u(t,x)\in w(N)$. 
Recalling that $dw_u(\nabla_xY)$ is the $dw_u(T_uN)$-component 
of $\p_x(dw_u(Y))$, 
we see   
\begin{align}
dw_u(\nabla_xY)
&=
\p_x
\left(
dw_u(Y)
\right)
-
\sum_{k=3}^d
(\p_x(dw_u(Y)), \nu_k(U))\nu_k(U)
\nonumber
\\
&=
\p_x\left(
dw_u(Y)
\right)
+
\sum_{k=3}^d
(dw_u(Y), \p_x\left(\nu_k(U)\right))\nu_k(U)
\nonumber
\\
&=
\p_x\left(
dw_u(Y)
\right)
+
\sum_{k=3}^d
(dw_u(Y), D_k(U)U_x)\nu_k(U)
\nonumber
\\
&=
\p_x\left(
dw_u(Y)
\right)
+
A(U)(dw_u(Y),U_x), 
\label{eq:cox2}
\end{align}
where $D_k=\operatorname{grad} \nu_k$ for $k=3,\dots,d$ 
and $A(q)(\cdot,\cdot)=
\displaystyle\sum_{k=3}^d
(\cdot, D_k(q)\cdot)\nu_k(q)$ 
is the second fundamental form at $q\in w(N)$. 
In the same way, only by replacing $x$ with $t$, 
we see 
\begin{align}
dw_u(\nabla_tY)
&=
\p_t\left(
dw_u(Y)
\right)
+
\sum_{k=3}^d
(dw_u(Y), D_k(U)U_t)\nu_k(U). 
\label{eq:cot2}
\end{align}
\par
The Sobolev embedding and the Gagliardo-Nirenberg inequality 
lead to the equivalence between 
$U_x, V_x\in L^{\infty}(0,T;H^6(\TT;\RR^d))$ 
and $u_x,v_x\in L^{\infty}(0,T;H^6(\TT;TN))$. 
In particular, from the Sobolev embedding 
$H^1(\TT)$ into $C(\TT)$, 
it follows that 
$\p_x^kU_x, \p_x^kV_x\in L^{\infty}((0,T)\times \TT;\RR^d)$ 
for $k=0,1,\ldots,5$, 
which will be used below without any comments.
\par
We next observe some properties related to $\nu_k$ and $D_k$
for $k=3\ldots,d$.  
\begin{lemma}
\label{lemma:nu}
For  each $(t,x)\in [0,T]\times \TT$, 
the following properties hold. 
\begin{align}
J(U)\nu_k(U)&=0,  
\label{eq:J1}
\\
(\nu_k(U), \WW)&=
-(\nu_k(U)-\nu_k(V), \VV)
=\mathcal{O}(|Z|), 
\label{eq:key1}
\\
(\nu_k(U),\p_x\WW)
&=
-(D_k(U)U_x,\WW)
-(D_k(U)Z_x,\VV)
+
\mathcal{O}(|Z|), 
\label{eq:key2}
\\
(\nu_k(U),\p_x^2\WW)
&=
-2\,(D_k(U)U_x,\p_x\WW)
+
\mathcal{O}(|Z|+|Z_x|+|\WW|), 
\label{eq:key22}
\\
(D_k(U)Y_1,Y_2)
&=(Y_1,D_k(U)Y_2) 
\ \ \
\text{for any $Y_1,Y_2:[0,T]\times \TT\to \RR^d$}.
\label{eq:D_k}
\end{align}
\end{lemma}
\begin{remark}
In particular, 
in view of \eqref{eq:key1}, 
we find 
(see the argument to show \eqref{eq:18e} in the third step)
that the term including $\p_x^2\WW$ or $\p_x\WW$ 
can be handled as a harmless term if the vector part  
is described by $\nu_k(U)$ with some $k=3,\ldots,d$.
This is related to the reason 
why we choose $dw_u(\nabla_xu_x)-dw_v(\nabla_xv_x)$ as $\WW$.  
\end{remark}
\begin{proof}[Proof of Lemma~\ref{lemma:nu}] 
First, \eqref{eq:J1} is a direct consequence of the definition of 
$J(U)$ 
and the orthogonality  
$\nu_k(U)\perp dw_u(T_uN)$. 
Next, 
in view of  
$(\nu_k(U),\UU)=(\nu_k(V),\VV)=0$, 
we have 
\begin{align}
(\nu_k(U),\WW)
%&
=
(\nu_k(U),\UU-\VV)
%\nonumber
%\\
%&
=
-(\nu_k(U),\VV)
%\nonumber
%\\
%&
=
-(\nu_k(U)-\nu_k(V),\VV),
\nonumber
\end{align}
which shows \eqref{eq:key1}. 
Moreover, by taking the derivative of \eqref{eq:key1} in $x$, 
we have 
\begin{align}
&(\nu_k(U),\p_x\WW)
\nonumber
\\
&=
\p_x\left\{
(\nu_k(U),\WW)
\right\}
-(\p_x\left\{
\nu_k(U)
\right\},\WW)
\nonumber
\\
&=
-(\p_x\left\{
\nu_k(U)-\nu_k(V)
\right\}, \VV)
-(\nu_k(U)-\nu_k(V),\p_x\VV)
-(D_k(U)U_x,\WW)
\nonumber
\\
&=
-(D_k(U)U_x-D_k(V)V_x, \VV)
-(\nu_k(U)-\nu_k(V),\p_x\VV)
-(D_k(U)U_x,\WW)
\nonumber
\\
&=
-(D_k(U)Z_x-(D_k(U)-D_k(V))V_x, \VV)
-(\nu_k(U)-\nu_k(V),\p_x\VV)
-(D_k(U)U_x,\WW)
\nonumber
\\
&=-(D_k(U)U_x,\WW)-(D_k(U)Z_x, \VV)
+\mathcal{O}(|Z|), 
\nonumber
\end{align}
which shows \eqref{eq:key2}. 
We obtain \eqref{eq:key22}, 
by taking the derivative of \eqref{eq:key2} in $x$.  
We omit the proof of \eqref{eq:D_k}, since it has been proved 
in \cite[pp.912]{onodera1}. 
\end{proof}
\par 
The following lemma comes from the K\"ahler condition on $(N,J,g)$. 
\begin{lemma}
\label{lemma:kaehler}
(i):  
For any $Y\in \Gamma(u^{-1}TN)$, 
\begin{align}
\p_x(J(U))dw_u(Y)
&=
\sum_{k=3}^d
\left(
dw_u(Y), J(U)D_k(U)U_x
\right)
\nu_k(U).
\label{eq:kaehler1}
\end{align}
(ii): 
For any $Y:[0,T]\times \TT\to \RR^d$, 
\begin{align}
\p_x(J(U))Y
&=
\sum_{k=3}^d
\left(Y,J(U)D_k(U)U_x\right)\nu_k(U)
-\sum_{k=3}^d
(Y,\nu_k(U))J(U)D_k(U)U_x.
\label{eq:kaehler2}
\end{align}
\end{lemma}
\begin{remark}
Using \eqref{eq:kaehler2} combined with \eqref{eq:key1},  
we can handle the term $\p_x(J(U))\p_xW$ as a harmless
term in the energy estimate for $\WW$ in $L^2$. 
\end{remark}
\begin{proof}[Proof of Lemma~\ref{lemma:kaehler}]
First we show (i). For $Y\in \Gamma(u^{-1}TN)$, 
the K\"ahler condition on $(N,J,g)$ implies 
$\nabla_xJ_uY=J_u\nabla_xY$. 
Hence there holds  
\begin{equation}
dw_u(\nabla_xJ_uY)=dw_u(J_u\nabla_xY). 
\label{eq:kae}
\end{equation}
From \eqref{eq:cox2} and \eqref{eq:J1}, 
the RHS 
of \eqref{eq:kae} satisfies
\begin{align}
dw_u(J_u\nabla_xY)
&
=J(U)dw_u(\nabla_xY)
=
J(U)
\p_x(dw_u(Y)).
\label{eq:kae2}
\end{align}
On the other hand, from \eqref{eq:cox2}, 
the left hand side of \eqref{eq:kae} satisfies
\begin{align}
dw_u(\nabla_xJ_uY)
&=
\p_x\left\{dw_u(J_uY)\right\}
+
\sum_{k=3}^d
\left(
dw_u(J_uY), D_k(U)U_x
\right)
\nu_k(U)
\nonumber
\\
&=
\p_x\left\{J(U)dw_u(Y)\right\}
+
\sum_{k=3}^d
\left(
J(U)dw_u(Y), D_k(U)U_x
\right)
\nu_k(U)
\nonumber
\\
&=
\p_x(J(U))dw_u(Y)
+
J(U)\p_x(dw_u(Y))
\nonumber
\\
&\quad
+
\sum_{k=3}^d
\left(
J(U)dw_u(Y), D_k(U)U_x
\right)
\nu_k(U).
\label{eq:kae3}
\end{align}
By substituting \eqref{eq:kae2} and \eqref{eq:kae3} into 
\eqref{eq:kae}, and by using \eqref{eq:J(wp)1}, 
we have 
\begin{align}
\p_x(J(U))dw_u(Y)
&=
-\sum_{k=3}^d
\left(
J(U)dw_u(Y), D_k(U)U_x
\right)
\nu_k(U)
\nonumber
\\
&=
\sum_{k=3}^d
\left(
dw_u(Y), J(U)D_k(U)U_x
\right)
\nu_k(U), 
\nonumber
\end{align}
which shows \eqref{eq:kaehler1}. 
Next we show (ii). 
Decomposing $Y=P(U)Y+N(U)Y$ where 
$P(U)Y\in dw(T_uN)$ and 
$N(U)Y\in (dw(T_uN))^{\perp}$ for each $(t,x)$, 
we have 
\begin{align}
\p_x(J(U))Y
&=\p_x(J(U))P(U)Y
+
\p_x(J(U))N(U)Y.
\label{eq:kae4}
\end{align}
By using \eqref{eq:kaehler1} and by noting 
that $N(U)Y$ is perpendicular to  $J(U)D_k(U)U_x$, 
we find that 
the first term of the RHS of 
\eqref{eq:kae4} satisfies 
\begin{align}
\p_x(J(U))P(U)Y
&=
\sum_{k=3}^d
\left(
P(U)Y, J(U)D_k(U)U_x
\right)
\nu_k(U)
\nonumber
\\
&=
\sum_{k=3}^d
\left(
Y, J(U)D_k(U)U_x
\right)
\nu_k(U).
\label{eq:kae5}
\end{align}
Moreover, since
\begin{equation}
\p_x(J(U))\nu_k(U)
=
\p_x(J(U)\nu_k(U))
-J(U)\p_x(\nu_k(U))
=
-J(U)D_k(U)U_x
\label{eq:J2}
\end{equation}
follows from \eqref{eq:J1}, 
the second term of the RHS of \eqref{eq:kae4} satisfies 
\begin{align}
\p_x(J(U))N(U)Y
&=
\sum_{k=3}^d(Y,\nu_k(U))\p_x(J(U))\nu_k(U)
\nonumber
\\
&=
-\sum_{k=3}^d(Y,\nu_k(U))
J(U)D_k(U)U_x.
\label{eq:kae6}
\end{align}
Substituting \eqref{eq:kae5} and \eqref{eq:kae6} into \eqref{eq:kae4}, 
we obtain \eqref{eq:kaehler2}. 
\end{proof}
\par 
As in the proof of Theorem~\ref{theorem:existence}, 
we denote the sectional curvature on $(N,g)$ by $S$ 
which is supposed to be a constant. 
Recall that the Riemann curvature tensor $R$ is expressed by  
\begin{align}
R(Y_1,Y_2)Y_3
&=
S\left\{
g(Y_2,Y_3)Y_1
-
g(Y_1,Y_3)Y_2
\right\}
\label{eq:curvature}
\end{align}
for any $Y_1,Y_2,Y_3\in \Gamma(u^{-1}TN)$.  
The next lemma comes from \eqref{eq:curvature}.
\begin{lemma}
\label{lemma:curvature3}
For any $Y_1,Y_2,Y_3\in \Gamma(u^{-1}TN)$, 
\begin{align}
&dw_u
\left(
R(Y_1,Y_2)Y_3
\right)
=
\sum_k
\left(
dw_u(Y_3), D_k(U)dw_u(Y_2)
\right)
P(U)D_k(U)dw_u(Y_1)
\nonumber
\\
&\phantom{dw_u
\left(
R(Y_1,Y_2)Y_3
\right)}
\qquad 
-
\sum_k
\left(
dw_u(Y_3), D_k(U)dw_u(Y_1)
\right)
P(U)D_k(U)dw_u(Y_2),
\label{eq:curvature2}
\\
&\sum_k
\left(
dw_u(Y_3), D_k(U)dw_u(Y_2)
\right)
P(U)D_k(U)dw_u(Y_1)
\nonumber
\\
&\qquad 
-
\sum_k
\left(
dw_u(Y_3), D_k(U)dw_u(Y_1)
\right)
P(U)D_k(U)dw_u(Y_2)
\nonumber
\\
&=
S\,\left\{
(dw_u(Y_3), dw_u(Y_2))dw_u(Y_1)
-(dw_u(Y_3), dw_u(Y_1))dw_u(Y_2)
\right\}
\label{eq:curvature3}
\end{align}
\end{lemma}
\begin{proof}[Proof of Lemma~\ref{lemma:curvature3}]
We can understand that \eqref{eq:curvature2} is a kind of the expression of 
the Gauss-Codazzi formula in Riemannian geometry. 
Fix $(t,x)\in [0,T]\times \TT$. 
We take a two-parameterized smooth map   
$\gamma=\gamma(s,\sigma):(-\delta,\delta)\times (-\delta,\delta)\to N$ 
with sufficiently small 
$\delta>0$, 
and a $Y_4\in \Gamma(\gamma^{-1}TN)$ 
so that 
$\gamma(0,0)=u(t,x)$, 
$\gamma_s(0,0)=Y_1(t,x)$, 
$\gamma_\sigma(0,0)=Y_2(t,x)$, 
and 
$Y_4(0,0)=Y_3(t,x)$.
Since $R(\gamma_s,\gamma_\sigma)Y_4=\nabla_s\nabla_\sigma Y_4-\nabla_\sigma\nabla_sY_4$, 
we deduce 
\begin{align}
dw_{\gamma}
\left(
R(\gamma_s,\gamma_\sigma)Y_4
\right)
&=
dw_{\gamma}
\left(
\nabla_s\nabla_\sigma Y_4
\right)
-
dw_{\gamma}
\left(
\nabla_\sigma\nabla_sY_4
\right)
\nonumber
\\
&=
P(w{\circ}\gamma)
\p_s\left(
dw_{\gamma}(\nabla_\sigma Y_4)
\right)
-
P(w{\circ}\gamma)
\p_\sigma\left(
dw_{\gamma}(\nabla_sY_4)
\right)
\nonumber
\\
&=
P(w{\circ}\gamma)
\left\{
\p_s\left(
dw_{\gamma}(\nabla_\sigma Y_4)
\right)
-
\p_\sigma\left(
dw_{\gamma}(\nabla_sY_4)
\right)
\right\}.
\label{eq:muda}
\end{align}
Similarly to \eqref{eq:cox2} or \eqref{eq:cot2}, 
the definition of the covariant derivatives yields
\begin{align}
\p_s\left(
dw_{\gamma}(\nabla_\sigma Y_4)
\right)
&=
\p_s\left(
\p_\sigma(dw_{\gamma}(Y_4))
+
\sum_{k=3}^d
\left(
dw_{\gamma}(Y_4), D_k(w{\circ}\gamma)\p_\sigma(w{\circ}\gamma)
\right)
\nu_k(w{\circ}\gamma)
\right)
\nonumber
\\
&=
\p_s\p_\sigma(dw_{\gamma}(Y_4))
+
\sum_{k=3}^d
\p_s
\left\{
\left(
dw_{\gamma}(Y_4), D_k(w{\circ}\gamma)\p_\sigma(w{\circ}\gamma)
\right)
\right\}
\nu_k(w{\circ}\gamma)
\nonumber
\\
&\quad
+
\sum_{k=3}^d
\left(
dw_{\gamma}(Y_4), D_k(w{\circ}\gamma)\p_\sigma(w{\circ}\gamma)
\right)
D_k(w{\circ}\gamma)
\p_s(w{\circ}\gamma),
\label{eq:muda2}
\end{align}
and 
\begin{align}
\p_\sigma\left(
dw_{\gamma}(\nabla_sY_4)
\right)
&=
\p_\sigma\p_s(dw_{\gamma}(Y_4))
+
\sum_{k=3}^d
\p_\sigma
\left\{
\left(
dw_{\gamma}(Y_4), D_k(w{\circ}\gamma)\p_s(w{\circ}\gamma)
\right)
\right\}
\nu_k(w{\circ}\gamma)
\nonumber
\\
&\quad
+
\sum_{k=3}^d
\left(
dw_{\gamma}(Y_4), D_k(w{\circ}\gamma)\p_s(w{\circ}\gamma)
\right)
D_k(w{\circ}\gamma)
\p_\sigma(w{\circ}\gamma).
\label{eq:muda3}
\end{align}
By substituting \eqref{eq:muda2} and \eqref{eq:muda3} 
into \eqref{eq:muda}, and by noting 
$P(w{\circ}\gamma)\nu_k(w{\circ}\gamma)=0$, 
we have 
\begin{align}
dw_{\gamma}
\left(
R(\gamma_s,\gamma_\sigma)Y_4
\right)
&=
\sum_{k=3}^d
\left(
dw_{\gamma}(Y_4), D_k(w{\circ}\gamma)\p_\sigma(w{\circ}\gamma)
\right)
P(w{\circ}\gamma)
D_k(w{\circ}\gamma)
\p_s(w{\circ}\gamma)
\nonumber
\\
&\quad 
-
\sum_{k=3}^d
\left(
dw_{\gamma}(Y_4), D_k(w{\circ}\gamma)\p_s(w{\circ}\gamma)
\right)
P(w{\circ}\gamma)
D_k(w{\circ}\gamma)
\p_\sigma(w{\circ}\gamma).
\nonumber
\end{align}
Thus, by taking the limit $(s,\sigma)\to (0,0)$, 
we obtain 
\begin{align}
dw_{u}
\left(
R(Y_1,Y_2)Y_3
\right)
&=
\sum_{k=3}^d
\left(
dw_{u}(Y_3), D_k(w{\circ}u)dw_u(Y_2)
\right)
P(w{\circ}u)
D_k(w{\circ}u)
dw_u(Y_1)
\nonumber
\\
&\quad 
-
\sum_{k=3}^d
\left(
dw_{u}(Y_3), D_k(w{\circ}u)dw_u(Y_1)
\right)
P(w{\circ}u)
D_k(w{\circ}u)
dw_u(Y_2) 
\nonumber
\end{align}
for each $(t,x)$. 
This implies \eqref{eq:curvature2}. 
Noting that $w:(N,g)\to (\RR^d,(\cdot,\cdot))$ is isometric,  
we see that 
\eqref{eq:curvature3} follows from 
\eqref{eq:curvature} and  \eqref{eq:curvature2}.
\end{proof}
\par 
Next properties come from the assumption that 
$N$ is a two-dimensional real manifold. 
Noting that $\left\{
\frac{U_x}{|U_x|}, \frac{J(U)U_x}{|U_x|}, \nu_3(U), \ldots, \nu_d(U)
\right\}$
forms an orthonormal basis of $\RR^d$ 
for each $(t,x)$ where $U_x(t,x)\ne 0$, 
we see 
\begin{equation}
|U_x|^2Y
=
(Y,U_x)U_x
+
(Y,J(U)U_x)J(U)U_x
+
\sum_{k=3}^d
(|U_x|^2Y,\nu_k(U))\nu_k(U)
\label{eq:frame_uniqueness}
\end{equation}
holds for every $(t,x)$. 
Note also that \eqref{eq:frame_uniqueness} is valid also for $(t,x)$ 
where $U_x(t,x)=0$, as 
each of both sides of \eqref{eq:frame_uniqueness} vanishes. 
Using \eqref{eq:frame_uniqueness} with 
$J(U)Y$ instead of $Y$, we have 
\begin{align}
|U_x|^2J(U)Y
&=
(J(U)Y,U_x)U_x
+
(J(U)Y,J(U)U_x)J(U)U_x
\nonumber
\\
&\quad
+
\sum_{k=3}^d
(|U_x|^2J(U)Y,\nu_k(U))\nu_k(U)
\nonumber
\\
&=
-(Y,J(U)U_x)U_x
+(Y,U_x)J(U)U_x.
\label{eq:k1}
\end{align}
Moreover, we introduce 
$T_2(U), \ldots, T_5(U):[0,T]\times \TT\to \RR^d$ 
defined by the following.
\begin{definition}
\label{definition:symmetry_uniqueness}
For any $Y:[0,T]\times \TT\to \RR^d$,  
\begin{align}
T_2(U)Y
&=
\frac{1}{2}|U_x|^2J(U)Y, 
\label{eq:T_2(U)}
\\
T_3(U)Y
&=
\frac{1}{2}
\biggl\{
(Y,\p_xU_x)J(U)U_x
+
(Y,U_x)J(U)\p_xU_x
\nonumber
\\
&\qquad\quad
+
(Y,J(U)\p_xU_x)U_x
+
(Y,J(U)U_x)\p_xU_x
\biggr\},
\label{eq:T_3(U)}
\\
T_4(U)Y
&=
\left(
Y,\p_xU_x+\sum_{k=3}^d(U_x,D_k(U)U_x)\nu_k(U)
\right)
J(U)U_x
-
(Y,U_x)J(U)\p_xU_x, 
\label{eq:T_4(U)}
\\
T_5(U)Y
&=
\frac{1}{2}
\biggl\{
(Y,\p_xU_x)J(U)U_x
+
(Y,U_x)J(U)\p_xU_x
\nonumber
\\
&\qquad\quad
-
(Y,J(U)\p_xU_x)U_x
-
(Y,J(U)U_x)\p_xU_x
\biggr\}.
\label{eq:T_5(U)}
\end{align} 
\end{definition}
We use \eqref{eq:frame_uniqueness} or \eqref{eq:k1} to show the following. 
\begin{lemma}
\label{lemma:symmetry_uniqueness}
For any 
$Y,Y_1,Y_2:[0,T]\times \TT\to \RR^d$, 
it follows that 
\begin{align}
T_2(U)Y
&=
\frac{1}{2}
\left\{
(Y,U_x)J(U)U_x
-(Y,J(U)U_x)U_x
\right\}, 
\label{eq:tae1}
\\
\p_x(T_2(U))Y
&=
(\p_xU_x,U_x)J(U)Y
+
\frac{1}{2}
|U_x|^2
\p_x(J(U))Y, 
\label{eq:tae2}
\\
\p_x(T_2(U))Y
&=
T_5(U)Y
-\frac{1}{2}
(Y,U_x)\sum_{k=3}^d(J(U)U_x,D_k(U)U_x)\nu_k(U)
\nonumber
\\
&\quad 
+\frac{1}{2}
\sum_{k=3}^d(Y,\nu_k(U))(J(U)U_x,D_k(U)U_x)U_x, 
\label{eq:tae3}
\\
(T_3(U)Y_1, Y_2)
&=
(Y_1,T_3(U)Y_2),  
\label{eq:tae4}
\\
(T_4(U)Y_1, Y_2)
&=
(Y_1,T_4(U)Y_2).
\label{eq:tae5}
\end{align}
\end{lemma}
\begin{proof}[Proof of Lemma~\ref{lemma:symmetry_uniqueness}]
First, \eqref{eq:tae1} is a direct consequence of 
\eqref{eq:k1}. 
Second, \eqref{eq:tae2} follows from 
substituting \eqref{eq:T_2(U)} into 
$\p_x(T_2(U))Y
=
\p_x\left\{
T_2(U)Y
\right\}
-T_2(U)\p_xY$.
Third, by substituting \eqref{eq:tae1} into 
$\p_x(T_2(U))Y=\p_x\left\{T_2(U)Y\right\}
-T_2(U)\p_xY$ and by using \eqref{eq:T_5(U)}, 
we have 
\begin{align}
\p_x(T_2(U))Y
&=
T_5(U)Y
+
\frac{1}{2}
(Y,U_x)\p_x(J(U))U_x
-\frac{1}{2}
(Y,\p_x(J(U))U_x)U_x.
\label{eq:iran21}
\end{align}
Recall here that \eqref{eq:kaehler2} yields 
$\p_x(J(U))U_x
=
\displaystyle\sum_{k=3}^d
(U_x,J(U)D_k(U)U_x)\nu_k(U)$.
Substituting this into the RHS of \eqref{eq:iran21}, 
we get \eqref{eq:tae3}. 
Next, in view of \eqref{eq:T_3(U)}, 
it is immediate that 
\eqref{eq:tae4} holds. 
Finally we show \eqref{eq:tae5}. 
The proof is reduced to that of \eqref{eq:gsym} with $i=2$. 
Noting that there exists $\Xi_i\in \Gamma(u^{-1}TN)$ such that 
$dw_u(\Xi_i)=P(U)Y_i$ for each $i=1,2$, 
we have 
\begin{align}
T_4(U)Y_1
&=
(Y_1,dw_u(\nabla_xu_x))dw_u(J_uu_x)
-(Y_1,dw_u(u_x))dw_u(J_u\nabla_xu_x)
\nonumber
\\
&=
(P(U)Y_1,dw_u(\nabla_xu_x))dw_u(J_uu_x)
-(P(U)Y_1,dw_u(u_x))dw_u(J_u\nabla_xu_x)
\nonumber
\\
&=
(dw_u(\Xi_1),dw_u(\nabla_xu_x))dw_u(J_uu_x)
-(dw_u(\Xi_1),dw_u(u_x))dw_u(J_u\nabla_xu_x).
\nonumber
\end{align}
Since $w$ is an isometric, this shows 
$$
T_4(U)Y_1
=dw_u\left\{
g(\Xi_1,\nabla_xu_x)J_uu_x
-g(\Xi_1,u_x)J_u\nabla_xu_x
\right\}, 
$$
and thus we obtain 
\begin{align}
(T_4(U)Y_1,Y_2)
&=
(dw_u\left\{
g(\Xi_1,\nabla_xu_x)J_uu_x
-g(\Xi_1,u_x)J_u\nabla_xu_x
\right\},P(U)Y_2)
\nonumber
\\
&=
(dw_u\left\{
g(\Xi_1,\nabla_xu_x)J_uu_x
-g(\Xi_1,u_x)J_u\nabla_xu_x
\right\},dw_u(\Xi_2))
\nonumber
\\
&=
g(g(\Xi_1,\nabla_xu_x)J_uu_x
-g(\Xi_1,u_x)J_u\nabla_xu_x, \Xi_2)
\nonumber
\\
&=
g(A_2\Xi_1,\Xi_2).
\nonumber
\end{align}
Since $g(A_2\Xi_1,\Xi_2)=g(\Xi_1,A_2\Xi_2)$ 
follows from \eqref{eq:gsym}, 
we conclude that \eqref{eq:tae5} holds. 
\end{proof}
In what follows, for simplicity, 
we sometimes write $dw$ instead of $dw_u$ or $dw_v$ 
and use the expression $\displaystyle\sum_{k}$ 
and  $\displaystyle\sum_{k,\ell}$
instead of 
$\displaystyle\sum_{k=3}^{d}$
and 
$\displaystyle\sum_{k,\ell=3}^{d}$
respectively. 
Any confusion will not occur.
%
%
%%%%
%%%%%%
%%%%%%
%%%%%%
\vspace{0.5em}
\par
{\bf 2. Analysis of the partial differential equation 
satisfied by $\UU$. } 
\\
We compute the PDE satisfied by $\UU$.
\par 
First, we start by the computation of the PDE satisfied by $U$. 
Since $u$ satisfies \eqref{eq:pde}, 
\begin{align}
U_t
&=
dw(u_t)
\nonumber
\\
&=
a\,dw(\nabla_xJ_u\nabla_x^2u_x)
+
\lambda\,dw(J_u\nabla_xu_x)
\nonumber
\\
&\quad
+
b\,dw(g(u_x,u_x)J_u\nabla_xu_x)
+
c\,dw(g(\nabla_xu_x,u_x)J_uu_x)
\nonumber
\\
&=
a\,dw(\nabla_xJ_u\nabla_x^2u_x)
+
\lambda\,J(U)dw(\nabla_xu_x)
\nonumber
\\
&\quad 
+
b\, (dw(u_x),dw(u_x))J(U)dw(\nabla_xu_x)
+
c\, (dw(\nabla_xu_x),dw(u_x))J(U)dw(u_x)
\nonumber
\\
&=
a\,dw(\nabla_xJ_u\nabla_x^2u_x)
+
\left\{\lambda+b\,|U_x|^2\right\}
J(U)\UU 
+
c\, (\UU,U_x)J(U)U_x. 
\label{eq:u1}
\end{align}
Using \eqref{eq:cox2} and \eqref{eq:J1} repeatedly, 
we have  
\begin{align}
dw(\nabla_x^2u_x)
&=
\p_x\left(
dw(\nabla_xu_x)
\right)
+
\sum_{\ell}
(dw(\nabla_xu_x), D_{\ell}(U)U_x)\nu_{\ell}(U)
\nonumber
\\
&=
\p_x\UU
+
\sum_{\ell}
(\UU, D_{\ell}(U)U_x)\nu_{\ell}(U), 
\label{eq:u2}
\\
J(U)dw(\nabla_x^2u_x)
&=
J(U)\p_x\UU
+
\sum_{\ell}
(\UU, D_{\ell}(U)U_x)J(U)\nu_{\ell}(U)
=
J(U)\p_x\UU, 
\label{eq:u3}
\\
dw(\nabla_xJ_u\nabla_x^2u_x)
&=
\p_x
\left(
dw(J_u\nabla_x^2u_x)
\right)
+
\sum_k
(dw(J_u\nabla_x^2u_x), D_k(U)U_x)\nu_k(U)
\nonumber
\\
&=
\p_x\left(
J(U)dw(\nabla_x^2u_x)
\right)
+
\sum_k
(J(U)dw(\nabla_x^2u_x), D_k(U)U_x)\nu_k(U)
\nonumber
\\
&=
\p_x\left(
J(U)\p_x\UU
\right)
+
\sum_k
(J(U)\p_x\UU, D_k(U)U_x)\nu_k(U).
\label{eq:u4}
\end{align}
From \eqref{eq:u1} and \eqref{eq:u4}, 
we have 
\begin{align}
U_t
&=
a\,\p_x\left(
J(U)\p_x\UU
\right)
+
a\,
\sum_k
(J(U)\p_x\UU, D_k(U)U_x)\nu_k(U)
+
\mathcal{O}(|U|+|U_x|+|\UU|).
\label{eq:U_t}
\end{align}
\par 
Next, we compute the PDE 
satisfied by $\UU$. 
From 
\eqref{eq:pde}, 
\eqref{eq:curvature}, 
and
\eqref{eq:cot2}, 
it follows that 
\begin{align}
\p_t\UU
&=
\p_t(dw(\nabla_xu_x))
\nonumber
\\
&=
dw(\nabla_t\nabla_xu_x)
-
\sum_{k}
(dw(\nabla_xu_x), D_k(U)U_t)\nu_k(U)
\nonumber
\\
&=
dw(\nabla_x^2u_t+R(u_t,u_x)u_x)
-
\sum_{k}
(\UU, D_k(U)U_t)\nu_k(U)
\nonumber
\\
&=
dw(\nabla_x^2u_t+S\left\{g(u_x,u_x)u_t-g(u_t,u_x)u_x\right\})
-
\sum_{k}
(\UU, D_k(U)U_t)\nu_k(U)
\nonumber
\\
&=
dw(\nabla_x^2u_t)
+S\left\{
(dw(u_x),dw(u_x))dw(u_t)
-(dw(u_t),dw(u_x))dw(u_x)
\right\}
\nonumber
\\
&\quad 
-
\sum_{k}
(\UU, D_k(U)U_t)\nu_k(U)
\nonumber
\\
&=
dw(\nabla_x^2u_t)
+
S|U_x|^2U_t
-S(U_x,U_t)U_x
-
\sum_{k}
(\UU, D_k(U)U_t)\nu_k(U)
\nonumber
\\
&=:I+II+III+IV.
\label{eq:I-IV} 
\end{align}
We compute $II$, $III$, $IV$, and $I$ in order. 
\par
Applying \eqref{eq:U_t}, 
we have 
\begin{align}
II
&=
aS\,\p_x\left\{
|U_x|^2J(U)\p_x\UU
\right\}
-2aS\,(\p_xU_x,U_x)J(U)\p_x\UU
\nonumber
\\
&\quad 
+aS\,|U_x|^2\sum_k
(J(U)\p_x\UU, D_k(U)U_x)\nu_k(U)
+
\mathcal{O}(|U|+|U_x|+|\UU|).
\label{eq:II}
\intertext{In the same way, by noting 
$(U_x,\nu_k(U))=0$ 
and 
\eqref{eq:J0}, 
we obtain 
}
III&=
-aS\, 
\left(
U_x, 
\p_x\left(J(U)\p_x\UU\right)
\right)
U_x
+
\mathcal{O}(|U|+|U_x|+|\UU|)
\nonumber
\\
&=
-aS\, (J(U)\p_x^2\UU,U_x)U_x
-aS\, (\p_x(J(U))\p_x\UU,U_x)U_x
+\mathcal{O}(|U|+|U_x|+|\UU|)
\nonumber
\\
&=
aS\, (\p_x^2\UU,J(U)U_x)U_x
-aS\, (\p_x(J(U))\p_x\UU,U_x)U_x
+\mathcal{O}(|U|+|U_x|+|\UU|).
\label{eq:III1}
\end{align}
Furthermore, by substituting \eqref{eq:k1} with $Y=\p_x^2\UU$ 
into the first term of the RHS of \eqref{eq:III1}, 
\begin{align}
III
&=
aS\,
(\p_x^2\UU,U_x)J(U)U_x
-aS\,
|U_x|^2J(U)\p_x^2\UU
-aS\, (\p_x(J(U))\p_x\UU,U_x)U_x
\nonumber
\\
&\quad
+\mathcal{O}(|U|+|U_x|+|\UU|)
\nonumber
\\
&=
aS\,
(\p_x^2\UU,U_x)J(U)U_x
-aS\,
\p_x
\left\{
|U_x|^2J(U)\p_x\UU
\right\}
\nonumber
\\
&\quad
+2aS\,
(\p_xU_x,U_x)J(U)\p_x\UU
+aS\,
|U_x|^2
\p_x(J(U))\p_x\UU
\nonumber
\\
&\quad
-aS\, (\p_x(J(U))\p_x\UU,U_x)U_x
+\mathcal{O}(|U|+|U_x|+|\UU|).
\label{eq:III}
\end{align}
In the same way, by applying \eqref{eq:U_t},  
\begin{align}
IV&=
-a\,
\sum_k
\left(
\UU, D_k(U)\p_x\left(
J(U)\p_x\UU
\right)
\right)\nu_k(U)
\nonumber
\\
&\quad
-a\, 
\sum_k
\left(
\UU, 
D_k(U)\sum_{\ell}
(J(U)\p_x\UU, D_{\ell}(U)U_x)\nu_{\ell}(U)
\right)
\nu_k(U)
\nonumber
\\
&\quad
+
\mathcal{O}(|U|+|U_x|+|\UU|)
\nonumber 
\\
&=
-a\,
\sum_k
\left(
\UU, D_k(U)\p_x\left(
J(U)\p_x\UU
\right)
\right)\nu_k(U)
\nonumber
\\
&\quad
-a\, 
\sum_{k,\ell}
(\UU, D_k(U)\nu_{\ell}(U))
(J(U)\p_x\UU, D_{\ell}(U)U_x)
\nu_k(U)
\nonumber
\\
&\quad
+
\mathcal{O}(|U|+|U_x|+|\UU|).
\label{eq:IV}
\end{align}
Let us now move on to the computation of $I$. 
We start with 
\begin{align}
I&=dw(\nabla_x^2u_t)
\nonumber
\\
&=
a\, dw(\nabla_x^4J_u\nabla_xu_x)
+
\lambda\,
dw(\nabla_x^2J_u\nabla_xu_x)
\nonumber
\\
&\quad 
+
b\,dw
\left(
\nabla_x^2\left\{
g(u_x,u_x)J_u\nabla_xu_x
\right\}
\right)
+
c\,dw
\left(
\nabla_x^2\left\{
g(\nabla_xu_x,u_x)J_uu_x
\right\}
\right)
\nonumber
\\
&=:
I_1+I_2+I_3+I_4.
\label{eq:I_1-I_4}
\end{align}
We compute $I_2$, $I_3$, $I_4$, and $I_1$ in order. 
\par
For $I_2$, we have 
\begin{align}
I_2&=
\lambda\,
dw\left(\nabla_x(\nabla_xJ_u\nabla_xu_x)\right)
\nonumber
\\
&=
\lambda\,
\p_x\left(
dw(\nabla_xJ_u\nabla_xu_x)
\right)
+
\lambda\, 
\sum_{\ell}
\left(
dw(\nabla_xJ_u\nabla_xu_x),
D_{\ell}(U)U_x
\right)
\nu_{\ell}(U). 
\nonumber
\end{align}
Since 
\begin{align}
dw(\nabla_xJ_u\nabla_xu_x)
&=
\p_x\left(
dw(J_u\nabla_xu_x)
\right)
+
\sum_k
\left(
dw(J_u\nabla_xu_x), D_k(U)U_x
\right)
\nu_k(U)
\nonumber
\\
&=
\p_x\left(
J(U)\UU
\right)
+
\sum_k
\left(
J(U)\UU, D_k(U)U_x
\right)
\nu_k(U)
\nonumber
\\
&=
J(U)\p_x\UU
+
\p_x(J(U))\UU
+
\sum_k
\left(
J(U)\UU, D_k(U)U_x
\right)
\nu_k(U), 
\label{eq:eri1}
\end{align}
we obtain 
\begin{align}
I_2&=
\lambda\,
\p_x\left\{
J(U)\p_x\UU
\right\}
+
\lambda\,
\p_x(J(U))\p_x\UU
+
2\lambda\,
\sum_k
\left(
J(U)\p_x\UU, D_k(U)U_x
\right)
\nu_k(U)
\nonumber
\\
&\quad
+\mathcal{O}(|U|+|U_x|+|\UU|).
\label{eq:I_2} 
\end{align}
\par 
For $I_3$, we have 
\begin{align}
I_3
&=
b\,dw\left\{
g(u_x,u_x)\nabla_x^2J_u\nabla_xu_x
\right\}
+
b\,dw\left\{
2\nabla_x(g(u_x,u_x))\nabla_xJ_u\nabla_xu_x
\right\}
\nonumber
\\
&\quad
+
b\,dw\left\{
\nabla_x^2(g(u_x,u_x))J_u\nabla_xu_x
\right\}
\nonumber
\\
&=
b\,g(u_x,u_x)
dw(\nabla_x^2J_u\nabla_xu_x)
+
4b\, g(\nabla_xu_x,u_x)dw(\nabla_xJ_u\nabla_xu_x)
\nonumber
\\
&\quad 
+
2b\, g(\nabla_x^2u_x,u_x)dw(J_u\nabla_xu_x)
+
2b\, g(\nabla_xu_x,\nabla_xu_x)dw(J_u\nabla_xu_x)
\nonumber
\\
&=
b\, |U_x|^2dw(\nabla_x^2J_u\nabla_xu_x)
+
4b\, (dw(\nabla_xu_x),U_x)dw(\nabla_xJ_u\nabla_xu_x)
\nonumber
\\
&\quad
+
2b\, (dw(\nabla_x^2u_x),U_x)dw(J_u\nabla_xu_x)
+
2b\,  |dw(\nabla_xu_x)|^2dw(J_u\nabla_xu_x)
\nonumber
\\
&=
b\, |U_x|^2dw(\nabla_x^2J_u\nabla_xu_x)
+
4b\, (\UU,U_x)dw(\nabla_xJ_u\nabla_xu_x)
\nonumber
\\
&\quad
+
2b\, (dw(\nabla_x^2u_x),U_x)J(U)\UU
+
2b\,  |\UU|^2J(U)\UU. 
\nonumber
\end{align}
Here, we recall the K\"ahler condition on $(N,J,g)$ 
to see 
$dw(\nabla_x^2J_u\nabla_xu_x)=dw(\nabla_xJ_u\nabla_x^2u_x)$. 
Hence, substituting \eqref{eq:u4}, \eqref{eq:eri1},  and \eqref{eq:u2}, 
we deduce 
\begin{align}
I_3&=
b\,|U_x|^2
\p_x\left\{
J(U)\p_x\UU
\right\}
+b\,|U_x|^2
\sum_k
(J(U)\p_x\UU, D_k(U)U_x)\nu_k(U) 
\nonumber
\\
&\quad +
4b\, (\UU,U_x)J(U)\p_x\UU
+
2b\, (\p_x\UU,U_x)J(U)\UU
+\mathcal{O}(|U|+|U_x|+|\UU|)
\nonumber
\\
&=
b\,
\p_x
\left\{
|U_x|^2
J(U)\p_x\UU
\right\}
-2b\, 
(\p_xU_x,U_x)
J(U)\p_x\UU
\nonumber
\\
&\quad 
+b\,|U_x|^2 
\sum_k
(J(U)\p_x\UU, D_k(U)U_x)\nu_k(U) 
\nonumber
\\
&\quad +
4b\, (\UU,U_x)J(U)\p_x\UU
+
2b\, (\p_x\UU,U_x)J(U)\UU
+\mathcal{O}(|U|+|U_x|+|\UU|).
\label{eq:erieri}
\end{align}
Furthermore, 
by noting 
$\UU=
dw(\nabla_xu_x)
=
\p_xU_x
+
\displaystyle\sum_k
(U_x,D_k(U)U_x)\nu_k(U)$, 
we see
\begin{align}
(\UU,U_x)
&=
(\p_xU_x,U_x)
+
\sum_k(U_x,D_k(U)U_x)(\nu_k(U),U_x)
=
(\p_xU_x,U_x), 
\label{eq:eri6}
\\
J(U)\UU
&=
J(U)\left(
\p_xU_x
+
\sum_k
(U_x,D_k(U)U_x)\nu_k(U)
\right)
=
J(U)\p_xU_x.
\label{eq:eri7}
\end{align}
Collecting the information 
\eqref{eq:erieri},
\eqref{eq:eri6}, 
and \eqref{eq:eri7}, 
we obtain 
\begin{align}
I_3&=
b\,
\p_x
\left\{
|U_x|^2
J(U)\p_x\UU
\right\}
+2b\, 
(\p_xU_x,U_x)
J(U)\p_x\UU
+
2b\, (\p_x\UU,U_x)J(U)\p_xU_x
\nonumber
\\
&\quad
+b\,|U_x|^2
\sum_k
(J(U)\p_x\UU, D_k(U)U_x)\nu_k(U) 
+\mathcal{O}(|U|+|U_x|+|\UU|).
\label{eq:I_3}
\end{align}
\par
For $I_4$, we have 
\begin{align}
I_4
&=
c\,dw
\left(
g(\nabla_x^3u_x,u_x)J_uu_x
\right)
+
3c\,dw
\left(
g(\nabla_x^2u_x,\nabla_xu_x)J_uu_x
\right)
\nonumber
\\
&\quad
+
2c\,dw
\left(
g(\nabla_x^2u_x,u_x)J_u\nabla_xu_x
\right)
+
2c\,dw
\left(
g(\nabla_xu_x,\nabla_xu_x)J_u\nabla_xu_x
\right)
\nonumber
\\
&\quad
+
c\,dw
\left(
g(\nabla_xu_x,u_x)J_u\nabla_x^2u_x.
\right)
\nonumber
\\
&=
c\,
(dw(\nabla_x^3u_x),U_x)J(U)U_x
+
3c\,
(dw(\nabla_x^2u_x),\UU)J(U)U_x
\nonumber
\\
&\quad
+
2c\,
(dw(\nabla_x^2u_x),U_x)J(U)\UU
+
2c\,
|\UU|^2J(U)\UU
+
c\,
(\UU,U_x)J(U)dw(\nabla_x^2u_x).
\nonumber
\end{align}
From \eqref{eq:u2}, it follows that 
\begin{align}
dw(\nabla_x^3u_x)
&=
\p_x\left\{
dw(\nabla_x^2u_x)
\right\}
+
\sum_{\ell}
\left(
dw(\nabla_x^2u_x),D_{\ell}(U)U_x
\right)
\nu_{\ell}(U)
\nonumber
\\
&=
\p_x^2\UU
+
2\sum_{k}
\left(
\p_x\UU, D_k(U)U_x
\right)
\nu_k(U)
+
\mathcal{O}(|U|+|U_x|+|\UU|). 
\nonumber
\end{align}
This together with $(\nu_k(U),U_x)=0$ yields 
\begin{align}
(dw(\nabla_x^3u_x),U_x)
&=
(\p_x^2\UU,U_x)
+
\mathcal{O}(|U|+|U_x|+|\UU|). 
\label{eq:eri9}
\end{align}
By using \eqref{eq:u2}, \eqref{eq:eri6}, 
\eqref{eq:eri7}, and \eqref{eq:eri9}, 
we obtain 
\begin{align}
I_4
&=
c\,
(\p_x^2\UU,U_x)J(U)U_x
+
3c\,
(\p_x\UU,\UU)J(U)U_x
\nonumber
\\
&\quad
+
2c\,
(\p_x\UU,U_x)J(U)\UU
+
c\,
(\UU,U_x)J(U)\p_x\UU
+
\mathcal{O}(|U|+|U_x|+|\UU|) 
\nonumber
\\
&=
c\,
(\p_x^2\UU,U_x)J(U)U_x
+
3c\,
(\p_x\UU,\p_xU_x)J(U)U_x
\nonumber
\\
&\quad
+
2c\,
(\p_x\UU,U_x)J(U)\p_xU_x
+
c\,
(\p_xU_x,U_x)J(U)\p_x\UU
+
\mathcal{O}(|U|+|U_x|+|\UU|).  
\label{eq:I_4}
\end{align}
\par 
For $I_1=a\,dw(\nabla_x\nabla_xJ_u\nabla_x^2\nabla_xu_x)$, 
we start with 
\begin{align}
&dw(\nabla_x\nabla_xJ_u\nabla_x^2\nabla_xu_x)
\nonumber
\\
&=
\p_x\left\{
dw(\nabla_xJ_u\nabla_x^2\nabla_xu_x)
\right\}
+
\sum_{k}
\left(
dw(\nabla_xJ_u\nabla_x^2\nabla_xu_x), 
D_k(U)U_x
\right)\nu_k(U), 
\label{eq:I11}
\end{align}
and 
\begin{align}
&dw(\nabla_xJ_u\nabla_x^2\nabla_xu_x)
\nonumber
\\
&=
\p_x\left\{
dw(J_u\nabla_x^2\nabla_xu_x)
\right\}
+
\sum_{\ell}
\left(
dw(J_u\nabla_x^2\nabla_xu_x), 
D_{\ell}(U)U_x
\right)
\nu_{\ell}(U).
\label{eq:I12}
\end{align}
From \eqref{eq:J(wp)1}, \eqref{eq:kaehler2} with $Y=\p_x\UU$, 
the K\"ahler condition on $(N,J,g)$,  
and \eqref{eq:u4}, 
it follows that 
\begin{align}
dw(J_u\nabla_x^2\nabla_xu_x)
&=
J(U)\p_x^2\UU
+
\p_x(J(U))\p_x\UU
+
\sum_k
(J(U)\p_x\UU, D_k(U)U_x)\nu_k(U)
\nonumber
\\
&=
J(U)\p_x^2\UU
+
\sum_k
\left(\p_x\UU,J(U)D_k(U)U_x\right)\nu_k(U)
\nonumber
\\
&\quad
-\sum_k
(\p_x\UU,\nu_k(U))J(U)D_k(U)U_x
-
\sum_k
(\p_x\UU, J(U)D_k(U)U_x)\nu_k(U)
\nonumber
\\
&=
J(U)\p_x^2\UU
-\sum_k
(\p_x\UU,\nu_k(U))J(U)D_k(U)U_x.
\nonumber
\end{align}
Here note that 
$(\UU, \nu_k(U))=0$ holds. 
By taking the derivative of both sides in $x$, 
we see 
\begin{equation}
(\p_x\UU,\nu_k(U))
=
-(\UU,\p_x(\nu_k(U)))
=
-(\UU, D_k(U)U_x).
\nonumber
\end{equation}
Using this, 
we obtain 
\begin{align}
dw(J_u\nabla_x^2\nabla_xu_x)
&=
J(U)\p_x^2\UU
+\sum_k
(\UU,D_k(U)U_x)J(U)D_k(U)U_x.
\label{eq:I15}
\end{align}
Furthermore, by substituting \eqref{eq:I15} into 
\eqref{eq:I12}, we have  
\begin{align}
&dw(\nabla_xJ_u\nabla_x^2\nabla_xu_x) 
\nonumber
\\
&=
\p_x\left\{
J(U)\p_x^2\UU
+\sum_n
(\UU,D_n(U)U_x)J(U)D_n(U)U_x
\right\}
\nonumber
\\
&\quad
+
\sum_{\ell}
\left(
J(U)\p_x^2\UU
+\sum_n
(\UU,D_n(U)U_x)J(U)D_n(U)U_x, 
D_{\ell}(U)U_x
\right)
\nu_{\ell}(U)
\nonumber
\\
&=
\p_x\left\{
J(U)\p_x^2\UU
\right\}
-
\sum_{\ell}
\left(\p_x^2\UU,
J(U)D_{\ell}(U)U_x
\right)\nu_{\ell}(U)
\nonumber
\\
&\quad
+
\sum_{n}
(\p_x\UU, D_n(U)U_x)J(U)D_n(U)U_x
\nonumber
\\
&\quad
+\sum_{n}
\left(
\UU, 
\p_x
\left\{
D_n(U)U_x
\right\}
\right)
J(U)D_n(U)U_x
\nonumber
\\
&\quad 
+
\sum_{n}
(\UU, D_n(U)U_x)
\p_x\left\{
J(U)D_n(U)U_x
\right\}
\nonumber
\\
&\quad
+
\sum_{\ell,n}
(\UU,D_n(U)U_x)
\left(
J(U)D_n(U)U_x,D_{\ell}(U)U_x
\right)
\nu_{\ell}(U).
\label{eq:I16}
\end{align}
Therefore, by substituting \eqref{eq:I16} into \eqref{eq:I11}, 
and by using 
$\p_x^2U_x=\p_x\UU+\mathcal{O}(|U|+|U_x|+|\UU|)$, 
we deduce 
\begin{align}
I_1
&=
a\,\p_x^2\left\{
J(U)\p_x^2\UU
\right\}
-
a\,\sum_{\ell}
\left(\p_x^3\UU,
J(U)D_{\ell}(U)U_x
\right)\nu_{\ell}(U)
\nonumber
\\
&\quad
-
a\,\sum_{\ell}
\left(\p_x^2\UU,
\p_x\left\{
J(U)D_{\ell}(U)U_x
\right\}
\right)\nu_{\ell}(U)
\nonumber
\\
&\quad
-
a\,\sum_{\ell}
\left(\p_x^2\UU,
J(U)D_{\ell}(U)U_x
\right)
D_{\ell}(U)U_x
\nonumber
\\
&\quad
+
a\,\sum_{n}
(\p_x^2\UU, D_n(U)U_x)J(U)D_n(U)U_x
\nonumber
\\
&\quad
+
a\,\sum_{n}
(\p_x\UU, \p_x\left\{D_n(U)U_x\right\})J(U)D_n(U)U_x
\nonumber
\\
&\quad
+
a\,\sum_{n}
(\p_x\UU, D_n(U)U_x)\p_x\left\{J(U)D_n(U)U_x\right\}
\nonumber
\\
&\quad
+a\,\sum_{n}
\left(
\p_x\UU, 
\p_x
\left\{
D_n(U)U_x
\right\}
\right)
J(U)D_n(U)U_x
\nonumber
\\
&\quad
+
a\,\sum_{n}
\left(
\UU, 
D_n(U)\p_x^2U_x
\right)
J(U)D_n(U)U_x
\nonumber
\\
&\quad 
+
a\,\sum_{n}
(\p_x\UU, D_n(U)U_x)
\p_x\left\{
J(U)D_n(U)U_x
\right\}
\nonumber
\\
&\quad 
+
a\,\sum_{n}
(\UU, D_n(U)U_x)
J(U)
D_n(U)\p_x^2U_x
\nonumber
\\
&\quad
+
a\,\sum_{\ell,n}
(\p_x\UU,D_n(U)U_x)
\left(
J(U)D_n(U)U_x,D_{\ell}(U)U_x
\right)
\nu_{\ell}(U)
\nonumber
\\
&\quad
+
a\,\sum_{k}
\left(
\p_x\left\{
J(U)\p_x^2\UU
\right\}, 
D_k(U)U_x
\right)\nu_k(U)
\nonumber
\\
&\quad
-
a\,\sum_{k,\ell}
\left(\p_x^2\UU,
J(U)D_{\ell}(U)U_x
\right)
\left(
\nu_{\ell}(U), 
D_k(U)U_x
\right)\nu_k(U)
\nonumber
\\
&\quad
+
a\,\sum_{k,n}
(\p_x\UU, D_n(U)U_x)
\left(
J(U)D_n(U)U_x, 
D_k(U)U_x
\right)\nu_k(U)
\nonumber
\\
&\quad 
+
\mathcal{O}
(|U|,|U_x|,|\UU|)
\nonumber
\\
&=
a\,\p_x^2\left\{
J(U)\p_x^2\UU
\right\}
-
2a\,
F_1(\p_x^3\UU)
-
a\,F_2(\p_x^2\UU)
+
a\,F_3(\p_x^2\UU)
\nonumber
\\
&\quad
+
2a\,F_4(\p_x\UU)
+
2a\,F_5(\p_x\UU)
+a\,F_6(\p_x\UU)
+
a\,F_7(\p_x\UU)
\nonumber
\\
&\quad
+
\sum_k
\mathcal{O}
\left(
|U|+|U_x|+|\UU|+|\p_x\UU|+|\p_x^2\UU|
\right)\nu_k(U)
\nonumber
\\
&\quad 
+
\mathcal{O}
(|U|+|U_x|+|\UU|), 
\label{eq:I20}
\end{align}
where for any $Y:[0,T]\times \TT\to \RR^d$, 
\begin{align}
F_1(Y)&=\sum_{k}
\left(Y,
J(U)D_{k}(U)U_x
\right)\nu_{k}(U),
\nonumber
\\
F_2(Y)&=\sum_{k}
\left(Y,
J(U)D_{k}(U)U_x
\right)
D_{k}(U)U_x,
\nonumber
\\
F_3(Y)&=\sum_{k}
(Y, D_k(U)U_x)J(U)D_k(U)U_x,
\nonumber
\\
F_4(Y)&=\sum_{k}
(Y, \p_x\left\{D_k(U)U_x\right\})J(U)D_k(U)U_x, 
\nonumber
\\
F_5(Y)&=\sum_{k}
(Y, D_k(U)U_x)\p_x\left\{J(U)D_k(U)U_x\right\},
\nonumber
\\
F_6(Y)&=\sum_{k}
\left(
\UU, 
D_k(U)Y
\right)
J(U)D_k(U)U_x,
\nonumber
\\
F_7(Y)&=\sum_{k}
(\UU, D_k(U)U_x)
J(U)
D_k(U)Y.
\nonumber
\end{align}
Combining 
\eqref{eq:II},
\eqref{eq:III}, 
\eqref{eq:IV}, 
\eqref{eq:I_2}, 
\eqref{eq:I_3},
\eqref{eq:I_4}, 
and 
\eqref{eq:I20}, 
we derive 
\begin{align}
\p_x\UU
&=
I_1+I_2+I_3+I_4+II+III+IV
\nonumber
\\
&=
a\,\p_x^2\left\{
J(U)\p_x^2\UU
\right\}
-
2a\,F_1(\p_x^3\UU)
+
\lambda\,
\p_x\left\{
J(U)\p_x\UU
\right\}
\nonumber
\\
&\quad
+
(b+aS-aS)\,
\p_x
\left\{
|U_x|^2
J(U)\p_x\UU
\right\}
+
(c+aS)\,
(\p_x^2\UU,U_x)J(U)U_x
\nonumber
\\
&\quad 
-
a\,F_2(\p_x^2\UU)
+
a\,F_3(\p_x^2\UU)
+
2a\,F_4(\p_x\UU)
+
2a\,F_5(\p_x\UU)
\nonumber
\\
&\quad
+a\,F_6(\p_x\UU)
+
a\,F_7(\p_x\UU)
+(2b+c-2aS+2aS)\, 
(\p_xU_x,U_x)
J(U)\p_x\UU
\nonumber
\\
&\quad
+
(2b+2c)\, (\p_x\UU,U_x)J(U)\p_xU_x 
+
3c\,
(\p_x\UU,\p_xU_x)J(U)U_x
\nonumber
\\
&\quad
-aS\, (\p_x(J(U))\p_x\UU,U_x)U_x
+aS\,
|U_x|^2
\p_x(J(U))\p_x\UU
\nonumber
\\
&\quad
+\lambda\, 
\p_x(J(U))\p_x\UU
+
r(U,U_x,\UU,\p_x\UU,\p_x^2\UU) 
+
\mathcal{O}
(|U|+|U_x|+|\UU|), 
\label{eq:eqUU}
\end{align}
where 
\begin{equation}
r(U,U_x,\UU,\p_x\UU,\p_x^2\UU)
=
\sum_k
\mathcal{O}
\left(
|U|+|U_x|+|\UU|+|\p_x\UU|+|\p_x^2\UU|
\right)\nu_k(U).
\nonumber
\end{equation}
\vspace{0.5em}
\\
{\bf 3. Classical energy estimates for 
$\|\WW\|_{L^2(\TT;\RR^d)}$ with the loss of derivatives}
\\ 
We compute $\p_t\WW=\p_t\UU-\p_t\VV$
and next evaluate the classical energy estimate for $\WW$ in $L^2$. 
Obviously, $\VV$ also satisfies \eqref{eq:eqUU} 
replacing $\UU$ with $\VV$. 
Hence, by using the mean value formula, 
we obtain 
\begin{align}
\p_t\WW
&=
a\,\p_x^2\left\{
J(U)\p_x^2\WW
\right\} 
-
2a\,F_1(\p_x^3\WW)
+
\lambda\, 
\p_x\left\{
J(U)\p_x\WW
\right\}
\nonumber
\\
&\quad
+
b\,
\p_x
\left\{
|U_x|^2
J(U)\p_x\WW
\right\}
+
(c+aS)\,
(\p_x^2\WW,U_x)J(U)U_x
\nonumber
\\
&\quad 
-
a\,F_2(\p_x^2\WW)
+
a\,F_3(\p_x^2\WW)
+
2a\,F_4(\p_x\WW)
+
2a\,F_5(\p_x\WW)
\nonumber
\\
&\quad
+a\,F_6(\p_x\WW)
+
a\,F_7(\p_x\WW)
+(2b+c)\, 
(\p_xU_x,U_x)
J(U)\p_x\WW
\nonumber
\\
&\quad
+
(2b+2c)\, (\p_x\WW,U_x)J(U)\p_xU_x 
+
3c\,
(\p_x\WW,\p_xU_x)J(U)U_x
\nonumber
\\
&\quad
-aS\, (\p_x(J(U))\p_x\WW,U_x)U_x
+aS\,
|U_x|^2
\p_x(J(U))\p_x\WW
+\lambda\, 
\p_x(J(U))\p_x\WW
\nonumber
\\
&\quad
+
r(U,U_x,\UU,\p_x\UU,\p_x^2\UU)
-
r(V,V_x,\VV,\p_x\VV,\p_x^2\VV)
\nonumber
\\
&\quad 
+
\mathcal{O}
(|Z|+|Z_x|+|\WW|).
\label{eq:eqWW}
\end{align} 
Note that 
$F_1(\cdot),\ldots,F_7(\cdot)$ 
should be expressed globally not by using 
local orthonormal frame.  
It is possible   
by using the second fundamental form on $w(N)$ 
and the derivatives, 
or by following the argument in \cite{onodera2} to use 
the partition of unity on $w(N)$. 
However, for simplicity and for better understandings, 
we will continue to use the local expression 
without loss of generality. 
\vspace{0.3em}  
\par
We  move on to the classical energy estimate for 
$\|\WW\|_{L^2}^2$. 
Since $k\geqslant 6$, 
$\WW\in L^{\infty}(0,T;H^5)\cap C([0,T];H^4)\cap C^1([0,T];L^2)$.  
This together with \eqref{eq:eqWW} implies 
\begin{align}
&\frac{1}{2}
\frac{d}{dt} 
\|\WW\|_{L^2}^2
\nonumber
\\
&=
\lr{\p_t\WW,\WW}
\nonumber
\\
&=
a\,
\lr{
\p_x^2\left\{
J(U)\p_x^2\WW
\right\}, 
\WW} 
-
2a\,
\lr{
F_1(\p_x^3\WW)
, \WW
}
+
\lambda\,
\lr{
\p_x\left\{
J(U)\p_x\WW
\right\}
,\WW
}
\nonumber
\\
&\quad
+
b\,
\lr{
\p_x
\left\{
|U_x|^2
J(U)\p_x\WW
\right\}
,\WW
}
+
(c+aS)\,
\lr{
(\p_x^2\WW,U_x)J(U)U_x, 
\WW
}
\nonumber
\\
&\quad 
-
a\,
\lr{
F_2(\p_x^2\WW), 
\WW
}
+
a\,
\lr{
F_3(\p_x^2\WW), 
\WW
}
+
2a\,
\lr{
F_4(\p_x\WW), 
\WW
}
\nonumber
\\
&\quad
+
2a\,
\lr{
F_5(\p_x\WW), 
\WW
}
+a\,
\lr{
F_6(\p_x\WW), 
\WW
}
+
a\,
\lr{
F_7(\p_x\WW), 
\WW
}
\nonumber
\\
&\quad
+(2b+c)\, 
\lr{
(\p_xU_x,U_x)
J(U)\p_x\WW, 
\WW
}
+
(2b+2c)\, 
\lr{
(\p_x\WW,U_x)J(U)\p_xU_x, 
\WW
}
\nonumber
\\
&\quad 
+
3c\,
\lr{
(\p_x\WW,\p_xU_x)J(U)U_x, 
\WW
}
-aS\, 
\lr{
(\p_x(J(U))\p_x\WW,U_x)U_x, 
\WW
}
\nonumber
\\
&\quad
+aS\,
\lr{
|U_x|^2
\p_x(J(U))\p_x\WW, 
\WW
}
+\lambda\, 
\lr{
\p_x(J(U))\p_x\WW, 
\WW
}
\nonumber
\\
&\quad
+
\lr{
r(U,U_x,\UU,\p_x\UU,\p_x^2\UU)
-
r(V,V_x,\VV,\p_x\VV,\p_x^2\VV), 
\WW
}
\nonumber
\\
&\quad 
+
\lr{
\mathcal{O}
(|Z|+|Z_x|+|\WW|), 
\WW
}.
\label{eq:EEn}
\end{align}
\par
Let us compute the RHS of the above 
term by term. 
First, by integrating by parts, it is immediate to see
\begin{align}
a\,
\lr{
\p_x^2\left\{
J(U)\p_x^2\WW
\right\}, 
\WW}
&=
a\,
\lr{
J(U)\p_x^2\WW, 
\p_x^2\WW
}
=0,
\nonumber
\\
\lambda\,
\lr{
\p_x\left\{
J(U)\p_x\WW
\right\}
,\WW
}
&=
-\lambda\,
\lr{
J(U)\p_x\WW
,\p_x\WW
}
=0, 
\nonumber
\\
b\,
\lr{
\p_x
\left\{
|U_x|^2
J(U)\p_x\WW
\right\}
,\WW
}
&=
-b\,
\lr{
|U_x|^2
J(U)\p_x\WW
,\p_x\WW
}
=0.
\nonumber
\end{align}
\par 
Next, 
by the Cauchy-Schwartz inequality, 
there holds 
\begin{align}
\lr{
\mathcal{O}
(|Z|+|Z_x|+|\WW|), 
\WW
}
&\leqslant 
\|\mathcal{O}
(|Z|+|Z_x|+|\WW|)\|_{L^2}
\|\WW\|_{L^2} 
\nonumber
\\
&
\leqslant 
C\left\{
\|Z\|_{L^2}^2
+\|Z_x\|_{L^2}^2
+\|\WW\|_{L^2}^2
\right\}
\nonumber
\end{align}
for some $C>0$. 
Here and hereafter, 
various positive constants depending on 
$\|u_x\|_{L^{\infty}(0,T;H^6)}$ and $\|v_x\|_{L^{\infty}(0,T;H^6)}$
will be denoted by the same $C$
without any comments. 
Besides, we use the notation $D(t)$ so that the square is defined by 
$$D(t)^2=\|Z\|_{L^2}^2
+\|Z_x\|_{L^2}^2
+\|\WW\|_{L^2}^2.$$ 
\par
Next, by noting 
\begin{align}
&
r(U,U_x,\UU,\p_x\UU,\p_x^2\UU)
-
r(V,V_x,\VV,\p_x\VV,\p_x^2\VV)
\nonumber
\\
&=
\sum_k
\mathcal{O}
\left(
|Z|+|Z_x|+|\WW|+|\p_x\WW|+|\p_x^2\WW|
\right)\nu_k(U)
\nonumber
\\
&\quad
+\sum_k
\mathcal{O}
\left(
|U|+|U_x|+|\UU|+|\p_x\UU|+|\p_x^2\UU|
\right)(\nu_k(U)-\nu_k(V)), 
\nonumber
\end{align}
we use \eqref{eq:key1} obtained in Lemma~\ref{lemma:nu}, 
$\p_xZ_x=\WW+\mathcal{O}(|Z|+|Z_x|)$,  
and the integration by parts, 
to obtain    
\begin{align}
&\lr{
r(U,U_x,\UU,\p_x\UU,\p_x^2\UU)
-
r(V,V_x,\VV,\p_x\VV,\p_x^2\VV), 
\WW
}
\nonumber
\\
&\leqslant 
\lr{
\sum_k
\mathcal{O}
\left(
|Z|+|Z_x|+|\WW|+|\p_x\WW|+|\p_x^2\WW|
\right)\nu_k(U),
\WW
}
+C\,D(t)^2
\nonumber
\\
&=
\int_{\TT}
\sum_k
\mathcal{O}
\left(
|Z|+|Z_x|+|\WW|+|\p_x\WW|+|\p_x^2\WW|
\right)
\mathcal{O}(|Z|)\,dx
+C\,D(t)^2
\nonumber
\\
&=
\int_{\TT}
\sum_k
\mathcal{O}
\left(
|Z|+|Z_x|+|\WW|
\right)
\mathcal{O}(|Z|+|Z_x|+|\WW|)\,dx
+C\,D(t)^2
\nonumber
\\
&\leqslant 
C\,D(t)^2.
\label{eq:18e}
\end{align}
\par
In the next computation, 
the K\"ahler condition on $(N,J,g)$ 
plays the crucial parts. 
Indeed, we 
apply \eqref{eq:kaehler2} with $Y=\p_x\WW$ and use \eqref{eq:key2}  
to obtain   
\begin{align}
\p_x(J(U))\p_x\WW
&=
\sum_k
\left(\p_x\WW,J(U)D_k(U)U_x\right)\nu_k(U)
-\sum_k
(\p_x\WW,\nu_k(U))J(U)D_k(U)U_x
\nonumber
\\
&=
\sum_k
\left(\p_x\WW,J(U)D_k(U)U_x\right)\nu_k(U)
+\mathcal{O}(|Z|+|Z_x|+|\WW|). 
\label{eq:0yumi}
\end{align}
By using \eqref{eq:0yumi} and \eqref{eq:key1}, we see  
\begin{align}
(\p_x(J(U))\p_x\WW,U_x)
&=
\sum_k
\left(\p_x\WW,J(U)D_k(U)U_x\right)(\nu_k(U), U_x)
+\mathcal{O}(|Z|+|Z_x|+|\WW|)
\nonumber
\\
&=
\mathcal{O}(|Z|+|Z_x|+|\WW|),
\label{eq:1yumi}
\\
(\p_x(J(U))\p_x\WW,\WW)
&=
\sum_k
\left(\p_x\WW,J(U)D_k(U)U_x\right)(\nu_k(U), \WW)
+\mathcal{O}(|Z|+|Z_x|+|\WW|)
\nonumber
\\
&=
\sum_k
\left(\p_x\WW,J(U)D_k(U)U_x\right)\mathcal{O}(|Z|)
+
\mathcal{O}(|Z|+|Z_x|+|\WW|).
\label{eq:2yumi}
\end{align}
Thus, by using \eqref{eq:1yumi} and the Cauchy-Schwartz inequality, 
we have
\begin{align}
-aS\, 
\lr{
(\p_x(J(U))\p_x\WW,U_x)U_x, 
\WW
}
&\leqslant 
C\,D(t)^2. 
\nonumber
\end{align}
In the same manner,
by using \eqref{eq:2yumi}, the Cauchy-Schwartz inequality, 
together with the integration by parts, 
we deduce 
\begin{align}
aS\,
\lr{
|U_x|^2
\p_x(J(U))\p_x\WW, 
\WW}
+\lambda\, 
\lr{
\p_x(J(U))\p_x\WW, 
\WW
} 
&\leqslant 
C\,D(t)^2.  
\nonumber
\end{align}
Collecting them, we obtain 
\begin{align}
&\frac{1}{2}\frac{d}{dt}
\|\WW\|_{L^2}^2
\nonumber
\\
&\leqslant
(c+aS)\,
\lr{
(\p_x^2\WW,U_x)J(U)U_x, 
\WW
}
+(2b+c)\, 
\lr{
(\p_xU_x,U_x)
J(U)\p_x\WW, 
\WW
}
\nonumber
\\
&\quad
+
(2b+2c)\, 
\lr{
(\p_x\WW,U_x)J(U)\p_xU_x, 
\WW
}  
+
3c\,
\lr{
(\p_x\WW,\p_xU_x)J(U)U_x, 
\WW
}
\nonumber
\\
&\quad 
+\sum_{i=1}^7R_i
+C\,D(t)^2, 
\label{eq:nn}
\end{align}
where 
\begin{alignat}{2}
R_1&=-
2a\,
\lr{
F_1(\p_x^3\WW)
, \WW
},
R_2=-
a\,
\lr{
F_2(\p_x^2\WW), 
\WW
},
R_3=a\,
\lr{
F_3(\p_x^2\WW), 
\WW
},
\nonumber
\\
R_4&=2a\,
\lr{
F_4(\p_x\WW), 
\WW
},
R_5
=2a\,
\lr{
F_5(\p_x\WW), 
\WW
},
R_6=a\,
\lr{
F_6(\p_x\WW), 
\WW
},
\nonumber
\\
R_7&=a\,
\lr{
F_7(\p_x\WW), 
\WW
}.
\nonumber
\end{alignat}
\par 
In what follows, we need to compute more carefully. 
Let us consider $R_1$. 
We start by integrating by parts to see  
\begin{align}
R_1
&=
2a\,
\lr{
\sum_{k}
\left(\p_x^2\WW,
\p_x\left\{J(U)D_{k}(U)U_x\right\}
\right)\nu_{k}(U)
, \WW
}
\nonumber
\\
&\quad
+
2a\,
\lr{
\sum_{k}
\left(\p_x^2\WW,
J(U)D_{k}(U)U_x
\right)
D_k(U)U_x
, \WW
}
\nonumber
\\
&\quad
+
2a\,
\lr{
\sum_{k}
\left(\p_x^2\WW,
J(U)D_{k}(U)U_x
\right)
\nu_{k}(U)
, \p_x\WW
}.
\nonumber
\end{align}
By applying \eqref{eq:key1} to 
the first term of the RHS of the above and 
by applying \eqref{eq:key2} to the third term of the 
RHS of the above, we have 
\begin{align}
R_1
&=
2a\,
\int_{\TT}
\sum_{k}
\left(\p_x^2\WW,
\p_x\left\{J(U)D_{k}(U)U_x\right\}
\right)
\mathcal{O}(|Z|)
\,dx
\nonumber
\\
&\quad
+
2a\,
\lr{
\sum_{k}
\left(\p_x^2\WW,
J(U)D_{k}(U)U_x
\right)
D_k(U)U_x
, \WW
}
\nonumber
\\
&\quad
-
2a\,
\lr{
\sum_{k}
\left(\p_x^2\WW,
J(U)D_{k}(U)U_x
\right)
D_k(U)U_x
, \WW
}
\nonumber
\\
&\quad
-
2a\,
\lr{
\sum_{k}
\left(\p_x^2\WW,
J(U)D_{k}(U)U_x
\right)
D_k(U)Z_x
, \VV
}
\nonumber
\\
&\quad
-
2a\,
\int_{\TT}
\sum_{k}
\left(\p_x^2\WW,
J(U)D_{k}(U)U_x
\right)
\mathcal{O}(|Z|)
\,dx
\nonumber
\\
&= 
2a\,
\int_{\TT}
\sum_{k}
\left(\p_x^2\WW,
\p_x\left\{J(U)D_{k}(U)U_x\right\}
\right)
\mathcal{O}(|Z|)
\,dx
\nonumber
\\
&\quad
-
2a\,
\lr{
\sum_{k}
\left(\p_x^2\WW,
J(U)D_{k}(U)U_x
\right)
D_k(U)Z_x
, \VV
}
\nonumber
\\
&\quad
-
2a\,
\int_{\TT}
\sum_{k}
\left(\p_x^2\WW,
J(U)D_{k}(U)U_x
\right)
\mathcal{O}(|Z|)
\,dx.
\nonumber
\\
\intertext{Here, 
by integrating by parts,  
the first and the third term are bounded by 
$C
\,D(t)^2
$. 
Therefore we obtain  
}
R_1
&\leqslant 
-
2a\,
\lr{
\sum_{k}
\left(\p_x^2\WW,
J(U)D_{k}(U)U_x
\right)
D_k(U)Z_x
, \VV
} 
+C\,D(t)^2. 
\nonumber
\end{align}
Furthermore, 
by using integration by parts, 
$\p_xZ_x=\WW+\mathcal{O}(|Z|+|Z_x|)$,   
\eqref{eq:J(wp)1},
and  
\eqref{eq:D_k},  
we deduce 
\begin{align}
R_1
&\leqslant 
2a\,
\lr{
\sum_{k}
\left(\p_x\WW,
J(U)D_{k}(U)U_x
\right)
D_k(U)\p_xZ_x
, \VV
}
+C\,D(t)^2
\nonumber
\\
&\leqslant 
2a\,
\lr{
\sum_{k}
\left(\p_x\WW,
J(U)D_{k}(U)U_x
\right)
D_k(U)(\WW+\mathcal{O}(|Z|+|Z_x|)), \VV
}
+C\,D(t)^2
\nonumber
\\
&\leqslant 
2a\,
\lr{
\sum_{k}
\left(\p_x\WW,
J(U)D_{k}(U)U_x
\right)
D_k(U)\WW
, \VV
}
+C\,D(t)^2
\nonumber
\\
&\leqslant 
2a\,
\lr{
\sum_{k}
\left(\p_x\WW,
J(U)D_{k}(U)U_x
\right)
D_k(U)\WW
, \UU
}
\nonumber
\\
&\quad
-
2a\,
\lr{
\sum_{k}
\left(\p_x\WW,
J(U)D_{k}(U)U_x
\right)
D_k(U)\WW
, \WW
}
+C\,D(t)^2
\nonumber
\\
&\leqslant 
2a\,
\lr{
\sum_{k}
\left(\p_x\WW,
J(U)D_{k}(U)U_x
\right)
D_k(U)\WW
, \UU
}
+C\,D(t)^2
\nonumber
\\
&=
-2a\,
\lr{
\sum_{k}
\left(J(U)\p_x\WW,
D_{k}(U)U_x
\right)
D_k(U)\UU
, \WW
}
+C\,D(t)^2
\nonumber
\\
&\leqslant 
2a\,
\lr{
\sum_{k}
\left(J(U)\WW,
D_{k}(U)U_x
\right)
D_k(U)\UU
, \p_x\WW
}
+C\,D(t)^2
\nonumber
\\
&=:
R_{11}+R_{12}+C\,D(t)^2, 
\label{eq:2e1}
\end{align}
where 
\begin{align}
R_{11}
&=2a\,
\lr{
\sum_{k}
\left(J(U)\WW,
D_{k}(U)U_x
\right)
P(U)D_k(U)\UU
, \p_x\WW
}, 
\nonumber
\\
R_{12}
&=2a\,
\lr{
\sum_{k}
\left(J(U)\WW,
D_{k}(U)U_x
\right)
N(U)D_k(U)\UU
, \p_x\WW
}.
\nonumber
\end{align}
For $R_{12}$, recall \eqref{eq:key2} to see
$$(N(U)D_k(U)\UU
, \p_x\WW)
=\sum_{\ell}
(D_k(U)\UU,\nu_{\ell}(U))(\nu_{\ell}(U),\p_x\WW)
=
\mathcal{O}
\left(
|Z|+|Z_x|+|\WW|
\right).
$$
This shows $R_{12}\leqslant C\,D(t)^2$. 
For $R_{11}$, by using \eqref{eq:J0} and \eqref{eq:D_k}, 
we see 
$$
R_{11}
=2a\,
\lr{
\sum_{k}
\left(J(U)\WW,
D_{k}(U)U_x
\right)
D_k(U)P(U)\p_x\WW, \UU
}. 
$$
Since $(N(U)D_k(U)P(U)\p_x\WW, \UU)=0$, we have 
$$
R_{11}
=2a\,
\lr{
\sum_{k}
\left(J(U)\WW,
D_{k}(U)U_x
\right)
P(U)D_k(U)P(U)\p_x\WW, \UU
}. 
$$
Applying \eqref{eq:curvature3} in Lemma~\ref{lemma:curvature3} with 
$dw_u(Y_1)=P(U)\p_x\WW$, $dw_u(Y_2)=U_x$, and with 
$dw_u(Y_3)=J(U)\WW$, we obtain 
\begin{align}
R_{11}
&=
2a\,
\lr{
\sum_{k}
\left(J(U)\WW,
D_{k}(U)P(U)\p_x\WW
\right)
P(U)D_k(U)U_x, \UU
}
\nonumber
\\
&\quad
+2aS\,
\lr{
(J(U)\WW, U_x)P(U)\p_x\WW,
\UU
}
-2aS
\lr{
(J(U)\WW, P(U)\p_x\WW)U_x,
\UU
}
\nonumber
\\
&=:R_{111}
+R_{112}+R_{113}.
\nonumber 
\end{align}
Here we recall \eqref{eq:key2} to see   
\begin{equation}
N(U)\p_x\WW
=\sum_k
(\p_x\WW,\nu_k(U))\nu_k(U)
=
\mathcal{O}(|Z|+|Z_x|+|\WW|).
\label{eq:key5}
\end{equation}  
This implies 
$P(U)\p_x\WW=\p_x\WW+\mathcal{O}(|Z|+|Z_x|+|\WW|).$
Using this \eqref{eq:J0}, \eqref{eq:J(wp)1} and 
$P(U)\UU=\UU$,  we obtain  
\begin{align}
R_{111}
&=
-2a\, 
\lr{
\sum_{k}
\left(\WW,
J(U)D_{k}(U)P(U)\p_x\WW
\right)
D_k(U)U_x, P(U)\UU
}
\nonumber
\\
&\leqslant
-2a\, 
\lr{
\sum_{k}
\left(\WW,
J(U)D_{k}(U)\p_x\WW
\right)
D_k(U)U_x, \UU
} 
+C\,D(t)^2
\nonumber
\\
&=
-2a\, 
\lr{
\sum_{k}
(\UU, D_k(U)U_x)
J(U)D_{k}(U)\p_x\WW, 
\WW
} 
+C\,D(t)^2.
\nonumber
\end{align}
In the same way, 
using \eqref{eq:J0}, \eqref{eq:J(wp)1} 
and 
$
\UU=\p_xU_x+
\sum_{\ell}(U_x,D_{\ell}(U)U_x)\nu_{\ell}(U)$,
we obtain 
\begin{align}
R_{112}
&\leqslant 
-2aS\,
\lr{
(\p_x\WW,\p_xU_x)J(U)U_x,\WW
}
+C\,D(t)^2, 
\nonumber
\\
R_{113}
&\leqslant
2aS\,
\lr{
(\p_xU_x,U_x)J(U)\p_x\WW,\WW
}
+C\,D(t)^2.
\nonumber 
\end{align}
Collecting them, 
we obtain 
\begin{align}
R_1&=R_{111}+R_{112}+R_{113}+R_{12}+C\,D(t)^2
\nonumber
\\
&\leqslant 
-2a\, 
\lr{
\sum_{k}
(\UU, D_k(U)U_x)
J(U)D_{k}(U)\p_x\WW, 
\WW
} 
\nonumber
\\
&\quad
-2aS\,
\lr{
(\p_x\WW,\p_xU_x)J(U)U_x,\WW
}
+
2aS\,
\lr{
(\p_xU_x,U_x)J(U)\p_x\WW,\WW
}
\nonumber
\\
&\quad
+C\,D(t)^2.
\label{eq:2e}
\end{align}
The first term of the RHS of \eqref{eq:2e}
is cancelled with 
the same term appearing from the computation of
$R_6+R_7$. 
\par 
We compute 
$R_6+R_7
=
a\,\lr{F_6(\p_x\WW),\WW}
+a\,\lr{F_7(\p_x\WW),\WW}$. 
By noting $J(U)=J(U)P(U)$ and 
by applying \eqref{eq:curvature3} 
with 
$dw_u(Y_1)=U_x$, 
$dw_u(Y_2)=P(U)\p_x\WW$ 
and with 
$dw_u(Y_3)=\UU$, 
we obtain 
\begin{align}
F_6(\p_x\WW)
&=
\sum_{k}
\left(
\UU, 
D_k(U)\p_x\WW
\right)
J(U)D_k(U)U_x
\nonumber
\\
&=
J(U)\sum_{k}
\left(
\UU, 
D_k(U)\p_x\WW
\right)
P(U)D_k(U)U_x
\nonumber
\\
&=
J(U)\sum_{k}
\left(
\UU, 
D_k(U)P(U)\p_x\WW
\right)
P(U)D_k(U)U_x
\nonumber
\\
&\quad
+
J(U)\sum_{k}
\left(
\UU, 
D_k(U)N(U)\p_x\WW
\right)
P(U)D_k(U)U_x
\nonumber
\\
&=
J(U)\sum_{k}
\left(
\UU, 
D_k(U)U_x
\right)
P(U)D_k(U)P(U)\p_x\WW
\nonumber
\\
&\quad
+S\,J(U)
\left\{
(\UU,P(U)\p_x\WW)U_x
-(\UU,U_x)P(U)\p_x\WW
\right\}
\nonumber
\\
&\quad
+
J(U)\sum_{k}
\left(
\UU, 
D_k(U)N(U)\p_x\WW
\right)
P(U)D_k(U)U_x.
\nonumber
\end{align}
Furthermore, 
we use $J(U)=J(U)P(U)$  
and \eqref{eq:key5} to obtain 
\begin{align}
F_6(\p_x\WW)
&=
\sum_{k}
\left(
\UU, 
D_k(U)U_x
\right)
J(U)D_k(U)P(U)\p_x\WW
\nonumber
\\
&\quad
+S\,(\p_x\WW,\UU)J(U)U_x
-S\,(\UU,U_x)J(U)\p_x\WW
+\mathcal{O}
(|Z|+|Z_x|+|\WW)
\nonumber
\\
&=
\sum_{k}
\left(
\UU, 
D_k(U)U_x
\right)
J(U)D_k(U)\p_x\WW
\nonumber
\\
&\quad
+S\,(\p_x\WW,\p_xU_x)J(U)U_x
-S\,(\p_xU_x,U_x)J(U)\p_x\WW
\nonumber
\\
&\quad
+\mathcal{O}
(|Z|+|Z_x|+|\WW)
\nonumber
\\
&=
F_7(\p_x\WW)
+S\,(\p_x\WW,\p_xU_x)J(U)U_x
-S\,(\p_xU_x,U_x)J(U)\p_x\WW
\nonumber
\\
&\quad
+\mathcal{O}
(|Z|+|Z_x|+|\WW). 
\nonumber
\end{align}
Hence we obtain 
\begin{align}
R_6
+
R_7
&\leqslant
2a\,
\lr{
\sum_{k}
\left(
\UU, 
D_k(U)U_x
\right)
J(U)D_k(U)\p_x\WW
,\WW
}
\nonumber
\\
&\quad
+aS\,
\lr{
(\p_x\WW,\p_xU_x)J(U)U_x
,\WW
}
\nonumber
\\
&\quad
-aS\,
\lr{
(\p_xU_x,U_x)J(U)\p_x\WW
,\WW}
+C\,D(t)^2.
\label{eq:1011e}
\end{align}
Combining \eqref{eq:2e} and \eqref{eq:1011e}, and 
using \eqref{eq:D_k}, 
we have 
\begin{align}
&R_1
+
R_6
+
R_7
\nonumber
\\
&\leqslant
-aS\,
\lr{
(\p_x\WW,\p_xU_x)J(U)U_x
,\WW} 
+aS\,
\lr{
(U_x,\p_xU_x)J(U)\p_x\WW
,\WW}
+C\,D(t)^2.
\label{eq:RR167}
\end{align}
\par
Next, we compute 
$R_2+R_3=-a\lr{F_2(\p_x^2\WW),\WW}+a\lr{F_3(\p_x^2\WW),\WW}$. 
As in deriving \eqref{eq:2e}, 
we use \eqref{eq:J(wp)1}, \eqref{eq:D_k}, 
and \eqref{eq:curvature3} with 
$dw_u(Y_1)=U_x$, 
$dw_u(Y_2)=J(U)\p_x^2\WW$ 
and with 
$dw_u(Y_3)=U_x$,
to deduce 
\begin{align}
F_2(\p_x^2\WW)
&=
\sum_k
(\p_x^2\WW,J(U)D_k(U)U_x)D_k(U)U_x
\nonumber
\\
&=
\sum_k
(\p_x^2\WW,J(U)D_k(U)U_x)N(U)D_k(U)U_x
\nonumber
\\
&\quad
+
\sum_k
(\p_x^2\WW,J(U)D_k(U)U_x)P(U)D_k(U)U_x
\nonumber
\\
&=
\sum_k
(\p_x^2\WW,J(U)D_k(U)U_x)N(U)D_k(U)U_x
\nonumber
\\
&\quad
-
\sum_k
(U_x, D_k(U)J(U)\p_x^2\WW)P(U)D_k(U)U_x
\nonumber
\\
&=
\sum_k
(\p_x^2\WW,J(U)D_k(U)U_x)N(U)D_k(U)U_x
\nonumber
\\
&\quad
-
\sum_k
(U_x, D_k(U)U_x)P(U)D_k(U)J(U)\p_x^2\WW
\nonumber
\\
&\quad
-
S\,
\left\{
(U_x,J(U)\p_x^2\WW)U_x
-(U_x,U_x)J(U)\p_x^2\WW
\right\}
\nonumber
\\
&=
-\sum_k
(U_x, D_k(U)U_x)D_k(U)J(U)\p_x^2\WW
\nonumber
\\
&\quad
+S\,(\p_x^2\WW,J(U)U_x)U_x
+S\,|U_x|^2J(U)\p_x^2\WW
\nonumber
\\
&\quad 
+\sum_k
(U_x, D_k(U)U_x)N(U)D_k(U)J(U)\p_x^2\WW
\nonumber
\\
&\quad 
+
\sum_k
(\p_x^2\WW,J(U)D_k(U)U_x)N(U)D_k(U)U_x
\nonumber
\\
&=
-\sum_k
(U_x, D_k(U)U_x)D_k(U)J(U)\p_x^2\WW
\nonumber
\\
&\quad
+S\,(\p_x^2\WW,J(U)U_x)U_x
+S\,|U_x|^2J(U)\p_x^2\WW
\nonumber
\\
&\quad 
+\sum_{\ell}
\mathcal{O}
\left(
|\p_x^2\WW|
\right)\nu_{\ell}(U),
\label{eq:6e1}
\\
\intertext{and in the same way 
we use \eqref{eq:J(wp)1}, \eqref{eq:D_k}, 
and \eqref{eq:curvature3} with 
$dw_u(Y_1)=U_x$, 
$dw_u(Y_2)=\p_x^2\WW$ 
and with 
$dw_u(Y_3)=U_x$,
to deduce }
F_3(\p_x^2\WW)
&=
\sum_{k}
(\p_x^2\WW, D_k(U)U_x)J(U)D_k(U)U_x
\nonumber
\\
&=
J(U)\sum_{k}
(U_x, D_k(U)\p_x^2\WW)P(U)D_k(U)U_x
\nonumber
\\
&=
J(U)\sum_{k}
(U_x, D_k(U)N(U)\p_x^2\WW)P(U)D_k(U)U_x
\nonumber
\\
&\quad
+
J(U)\sum_{k}
(U_x, D_k(U)P(U)\p_x^2\WW)P(U)D_k(U)U_x
\nonumber
\\
&=
J(U)\sum_{k}
(U_x, D_k(U)N(U)\p_x^2\WW)P(U)D_k(U)U_x
\nonumber
\\
&\quad
+
J(U)\sum_{k}
(U_x, D_k(U)U_x)P(U)D_k(U)P(U)\p_x^2\WW
\nonumber
\\
&\quad
+S\,J(U)
\left\{
(U_x, P(U)\p_x^2\WW)U_x
-(U_x,U_x)P(U)\p_x^2\WW
\right\}
\nonumber
\\
&=
\sum_{k}
(U_x, D_k(U)N(U)\p_x^2\WW)J(U)D_k(U)U_x
\nonumber
\\
&\quad
+
\sum_{k}
(U_x, D_k(U)U_x)J(U)D_k(U)P(U)\p_x^2\WW
\nonumber
\\
&\quad
+S\,(\p_x^2\WW,U_x)J(U)U_x
-S\,|U_x|^2J(U)\p_x^2\WW
\nonumber
\\
&=
\sum_{k}
(U_x, D_k(U)U_x)J(U)D_k(U)\p_x^2\WW
\nonumber
\\
&\quad
+S\,(\p_x^2\WW,U_x)J(U)U_x
-S\,|U_x|^2J(U)\p_x^2\WW
\nonumber
\\
&\quad 
+
\sum_{k}
(U_x, D_k(U)N(U)\p_x^2\WW)J(U)D_k(U)U_x
\nonumber
\\
&\quad
-
\sum_{k}
(U_x, D_k(U)U_x)J(U)D_k(U)N(U)\p_x^2\WW.
\label{eq:6e2}
\end{align}
Here, we use \eqref{eq:key22} to see
\begin{align}
N(U)\p_x^2\WW
&=
\sum_{\ell}
(\p_x^2\WW,\nu_{\ell}(U))\nu_{\ell}(U)
\nonumber
\\
&=
-2\sum_{\ell}
(\p_x\WW, D_{\ell}(U)U_x)\nu_{\ell}(U)
+
\mathcal{O}(|Z|+|Z_x|+|\WW|).
\label{eq:6e22}
\end{align}
By substituting \eqref{eq:6e22} into 
the fourth and fifth term of the RHS of \eqref{eq:6e2},
we have 
\begin{align}
F_3(\p_x^2\WW)
&=
\sum_{k}
(U_x, D_k(U)U_x)J(U)D_k(U)\p_x^2\WW
\nonumber
\\
&\quad
+S\,(\p_x^2\WW,U_x)J(U)U_x
-S\,|U_x|^2J(U)\p_x^2\WW
\nonumber
\\
&\quad 
-2
\sum_{k,\ell}
(\p_x\WW, D_{\ell}(U)U_x)
(U_x, D_k(U)\nu_{\ell}(U))J(U)D_k(U)U_x
\nonumber
\\
&\quad
+2
\sum_{k,\ell}
(\p_x\WW, D_{\ell}(U)U_x)
(U_x, D_k(U)U_x)J(U)D_k(U)\nu_{\ell}(U)
\nonumber
\\
&\quad
+
\mathcal{O}
(|Z|+|Z_x|+|\WW|).
\label{eq:6e23}
\end{align}
Thus, from \eqref{eq:6e1}, \eqref{eq:6e23}, 
and \eqref{eq:k1}, 
it follows that 
\begin{align}
&
-F_2(\p_x^2\WW)
+F_3(\p_x^2\WW)
\nonumber
\\
&=
\sum_k
(U_x, D_k(U)U_x)D_k(U)J(U)\p_x^2\WW
+
\sum_{k}
(U_x, D_k(U)U_x)J(U)D_k(U)\p_x^2\WW
\nonumber
\\
&\quad
-S\,(\p_x^2\WW,J(U)U_x)U_x
+S\,(\p_x^2\WW,U_x)J(U)U_x
-2S\,|U_x|^2J(U)\p_x^2\WW
\nonumber
\\
&\quad 
-2
\sum_{k,\ell}
(\p_x\WW, D_{\ell}(U)U_x)
(U_x, D_k(U)\nu_{\ell}(U))J(U)D_k(U)U_x
\nonumber
\\
&\quad
+2
\sum_{k,\ell}
(\p_x\WW, D_{\ell}(U)U_x)
(U_x, D_k(U)U_x)J(U)D_k(U)\nu_{\ell}(U)
\nonumber
\\
&\quad
+\sum_{\ell}
\mathcal{O}
\left(
|\p_x^2\WW|
\right)\nu_{\ell}(U)
+
\mathcal{O}
(|Z|+|Z_x|+|\WW|)
\nonumber
\\
&=
\sum_k
(U_x, D_k(U)U_x)(D_k(U)J(U)+J(U)D_k(U))\p_x^2\WW
\nonumber
\\
&\quad
-S\,|U_x|^2J(U)\p_x^2\WW
\nonumber
\\
&\quad 
-2
\sum_{k,\ell}
(\p_x\WW, D_{\ell}(U)U_x)
(U_x, D_k(U)\nu_{\ell}(U))J(U)D_k(U)U_x
\nonumber
\\
&\quad
+2
\sum_{k,\ell}
(\p_x\WW, D_{\ell}(U)U_x)
(U_x, D_k(U)U_x)J(U)D_k(U)\nu_{\ell}(U)
\nonumber
\\
&\quad
+\sum_{\ell}
\mathcal{O}
\left(
|\p_x^2\WW|
\right)\nu_{\ell}(U)
+
\mathcal{O}
(|Z|+|Z_x|+|\WW|).
\label{eq:6e3}
\end{align}
Therefore, from  \eqref{eq:6e3} and 
$$|U_x|^2J(U)\p_x^2\WW
=\p_x
\left\{
|U_x|^2J(U)\p_x\WW
\right\}
-2(\p_xU_x,U_x)J(U)\p_x\WW
-|U_x|^2\p_x(J(U))\p_x\WW,
$$ 
we see that 
$R_2+R_3
=-a\lr{F_2(\p_x^2\WW),\WW}
+a\lr{F_3(\p_x^2\WW),\WW}$ 
is evaluated as follows:
\begin{align}
R_2
+
R_3
&\leqslant 
a\, 
\left\langle
\p_x\biggl\{
\sum_k
(U_x, D_k(U)U_x)(D_k(U)J(U)+J(U)D_k(U))\p_x\WW
\biggr\}, 
\WW
\right\rangle
\nonumber
\\
&\quad
-a\, 
\left\langle
\p_x\biggl\{
\sum_k
(U_x, D_k(U)U_x)(D_k(U)J(U)+J(U)D_k(U))
\biggr\}
\p_x\WW, 
\WW
\right\rangle
\nonumber
\\
&\quad
-aS\,
\lr{
\p_x\biggl\{
|U_x|^2J(U)\p_x\WW
\biggr\}, \WW
}
+2aS\,
\lr{
(\p_xU_x,U_x)J(U)\p_x\WW
,\WW
}
\nonumber
\\
&\quad
+aS\,
\lr{
|U_x|^2\p_x(J(U))\p_x\WW
,\WW
}
\nonumber
\\
&\quad
-2a\,
\lr{
\sum_{k,\ell}
(\p_x\WW, D_{\ell}(U)U_x)
(U_x, D_k(U)\nu_{\ell}(U))J(U)D_k(U)U_x,\WW
}
\nonumber
\\
&\quad
+2a\,
\lr{
\sum_{k,\ell}
(\p_x\WW, D_{\ell}(U)U_x)
(U_x, D_k(U)U_x)J(U)D_k(U)\nu_{\ell}(U), \WW
}
\nonumber
\\
&\quad 
+
\lr{
\sum_{\ell}
\mathcal{O}
\left(|\p_x^2\WW|\right)\nu_{\ell}(U), 
\WW
}
\nonumber
\\
&\quad
+C\,D(t)^2.
\label{eq:6e4}
\end{align}
Note here that  
$$
((J(U)D_k(U)+D_k(U)J(U))Y_1,Y_2)
=
-(Y_1, (J(U)D_k(U)+D_k(U)J(U))Y_2)
$$
holds for any $Y_1,Y_2:[0,T]\times \TT\to \RR^d$.
This implies that the first term of the RHS of \eqref{eq:6e4}
vanishes. 
Indeed, the integration by parts yields 
\begin{align}
&
a\,\left\langle
\p_x\biggl\{
\sum_k
(U_x, D_k(U)U_x)(D_k(U)J(U)+J(U)D_k(U))\p_x\WW
\biggr\}, 
\WW
\right\rangle
\nonumber
\\
&=
-a\,\left\langle
\sum_k
(U_x, D_k(U)U_x)(D_k(U)J(U)+J(U)D_k(U))\p_x\WW, 
\p_x\WW
\right\rangle
\nonumber
\\
&=0. 
\nonumber
\end{align}
In addition, the third term of the RHS of \eqref{eq:6e4}
vanishes by integrating by parts.
Beside,
due to the presence of $N(U)$, we can bound 
the eighth term of the RHS of \eqref{eq:6e4} 
by $C\,D(t)^2$ 
using the same argument to show \eqref{eq:18e}.
Consequently, we derive 
\begin{align}
R_2+R_3
&\leqslant 
-a\, 
\left\langle
\p_x\biggl\{
\sum_k
(U_x, D_k(U)U_x)(D_k(U)J(U)+J(U)D_k(U))
\biggr\}
\p_x\WW, 
\WW
\right\rangle
\nonumber
\\
&\quad
+2aS\,
\lr{
(\p_xU_x,U_x)J(U)\p_x\WW
,\WW
}
\nonumber
\\
&\quad
-2a\,
\lr{
\sum_{k,\ell}
(\p_x\WW, D_{\ell}(U)U_x)
(U_x, D_k(U)\nu_{\ell}(U))J(U)D_k(U)U_x,\WW
}
\nonumber
\\
&\quad
+2a\,
\lr{
\sum_{k,\ell}
(\p_x\WW, D_{\ell}(U)U_x)
(U_x, D_k(U)U_x)J(U)D_k(U)\nu_{\ell}(U), \WW
}
\nonumber
\\
&\quad
+
C\,D(t)^2.
\label{eq:67e}
\end{align}
The third and fourth term and the bad part of the first term 
of the RHS of \eqref{eq:67e} will be cancelled with the same term 
appearing in the computation of 
$R_4+R_5$.  
\par 
Let us next compute $R_4+R_5$. 
To begin with, 
we introduce 
$T_1(U)$ 
which is  defined by 
\begin{align}
T_1(U)Y
&=
\sum_k
(Y,D_k(U)U_x)J(U)D_k(U)U_x
\nonumber
\end{align}
for any $Y:[0,T]\times \TT\to \RR^d$. 
Substituting this with $Y=\p_x^2\WW$ 
into the RHS of 
$\p_x(T_1(U))\p_x\WW=\p_x\left\{T_1(U)\p_x\WW\right\}-T_1(U)\p_x^2\WW$, 
we can write
\begin{align}
R_4
+
R_5
&=
2a\,
\lr{
\p_x(T_1(U))\p_x\WW
,\WW}. 
\label{eq:89e1}
\end{align}
On the other hand, 
using \eqref{eq:curvature3} with 
$dw_u(Y_1)=U_x$, 
$dw_u(Y_2)=P(U)Y$ 
and with 
$dw_u(Y_3)=U_x$,
we find that  
$T_1(U)Y$ has the following another expression. 
\begin{align}
T_1(U)Y
&=
J(U)\sum_k
(Y,D_k(U)U_x)P(U)D_k(U)U_x
\nonumber
\\
&=
J(U)\sum_k
(P(U)Y,D_k(U)U_x)P(U)D_k(U)U_x
\nonumber
\\
&\quad
+
J(U)\sum_k
(N(U)Y,D_k(U)U_x)P(U)D_k(U)U_x
\nonumber
\\
&=
J(U)\sum_k
(U_x, D_k(U)P(U)Y)P(U)D_k(U)U_x
\nonumber
\\
&\quad
+
J(U)\sum_k
(N(U)Y,D_k(U)U_x)P(U)D_k(U)U_x
\nonumber
\\
&=
J(U)\sum_k
(U_x, D_k(U)U_x)P(U)D_k(U)P(U)Y
\nonumber
\\
&\quad
+
S\,J(U)
\left\{
(U_x,P(U)Y)U_x
-|U_x|^2P(U)Y
\right\}
\nonumber
\\
&\quad
+
J(U)\sum_k
(N(U)Y,D_k(U)U_x)P(U)D_k(U)U_x
\nonumber
\\
&=
\sum_k
(U_x, D_k(U)U_x)J(U)D_k(U)P(U)Y
\nonumber
\\
&\quad
+
S\,(U_x,P(U)Y)J(U)U_x
-S|U_x|^2J(U)Y
\nonumber
\\
&\quad
+
\sum_k
(N(U)Y,D_k(U)U_x)J(U)D_k(U)U_x
\nonumber
\\
&=
\sum_k
(U_x, D_k(U)U_x)J(U)D_k(U)Y
\nonumber
\\
&\quad
-\sum_k
(U_x, D_k(U)U_x)J(U)D_k(U)N(U)Y
\nonumber
\\
&\quad
+
S\,(Y,U_x)J(U)U_x
-S|U_x|^2J(U)Y
\nonumber
\\
&\quad
+
\sum_k
(N(U)Y,D_k(U)U_x)J(U)D_k(U)U_x
\nonumber
\end{align}
for any $Y:[0,T]\times \TT\to \RR^d$. 
If we adopt this formulation, 
we have  
\begin{align}
\p_x(T_1(U))Y
&=
\p_x\left\{
\sum_k
(U_x, D_k(U)U_x)J(U)D_k(U)
\right\}
Y
\nonumber
\\
&\quad
-\p_x\left\{
\sum_k
(U_x, D_k(U)U_x)J(U)D_k(U)\right\}
N(U)Y
\nonumber
\\
&\quad
-\sum_k
(U_x, D_k(U)U_x)J(U)D_k(U)\p_x(N(U))Y
\nonumber
\\
&\quad
+S\,(Y,\p_xU_x)J(U)U_x
+S\,(Y,U_x)\p_x(J(U))U_x
+S\,(Y,U_x)J(U)\p_xU_x
\nonumber
\\
&\quad 
-2S(\p_xU_x,U_x)J(U)Y-S|U_x|^2\p_x(J(U))Y
\nonumber
\\
&\quad
+
\sum_k
(\p_x(N(U))Y,D_k(U)U_x)J(U)D_k(U)U_x
\nonumber
\\
&\quad
+
\sum_k
(N(U)Y,\p_x\left\{D_k(U)U_x\right\})J(U)D_k(U)U_x
\nonumber
\\
&\quad
+
\sum_k
(N(U)Y,D_k(U)U_x)\p_x\left\{J(U)D_k(U)U_x\right\}.
\label{eq:89e2}
\end{align}
By substituting \eqref{eq:89e2} into \eqref{eq:89e1}, 
we have 
\begin{align}
R_4
+
R_5
&=
2a\,
\lr{
\p_x\biggl\{
\sum_k
(U_x, D_k(U)U_x)J(U)D_k(U)
\biggr\}
\p_x\WW,
\WW
}
\nonumber
\\
&\quad
-2a\,
\lr{
\p_x
\biggl\{
\sum_k
(U_x, D_k(U)U_x)J(U)D_k(U)
\biggr\}
N(U)\p_x\WW,
\WW
}
\nonumber
\\
&\quad
-2a\,
\lr{
\sum_k
(U_x, D_k(U)U_x)J(U)D_k(U)\p_x(N(U))\p_x\WW, 
\WW
}
\nonumber
\\
&\quad
+2aS\,
\lr{
(\p_x\WW,\p_xU_x)J(U)U_x,
\WW
}
+2aS\,
\lr{
(\p_x\WW,U_x)\p_x(J(U))U_x,
\WW
}
\nonumber
\\
&\quad
+2aS\,
\lr{
(\p_x\WW,U_x)J(U)\p_xU_x,
\WW
}
-4aS\,
\lr{
(\p_xU_x,U_x)J(U)\p_x\WW,
\WW
}
\nonumber
\\
&\quad
-2aS\,
\lr{
|U_x|^2\p_x(J(U))\p_x\WW,
\WW
}
\nonumber
\\
&\quad
+
2a\,
\lr{
\sum_k
(\p_x(N(U))\p_x\WW,D_k(U)U_x)J(U)D_k(U)U_x,\WW
}
\nonumber
\\
&\quad
+
2a\,
\lr{
\sum_k
(N(U)\p_x\WW,\p_x\left\{D_k(U)U_x\right\})J(U)D_k(U)U_x,\WW
}
\nonumber
\\
&\quad
+
2a\,
\lr{
\sum_k
(N(U)\p_x\WW,D_k(U)U_x)\p_x\left\{J(U)D_k(U)U_x\right\},
\WW
}.
\label{eq:89e3}
\end{align}
The second, the tenth, and the eleventh term of the 
RHS of \eqref{eq:89e3} are bounded by 
$C\,D(t)^2$, in view of \eqref{eq:key5}.  
For the fifth term of the RHS of \eqref{eq:89e3}, 
we use \eqref{eq:kaehler2} and $(U_x,\nu_k(U))=0$ 
to see 
\begin{align}
\p_x(J(U))U_x
&=
\sum_k
\left(U_x,J(U)D_k(U)U_x\right)\nu_k(U),
\label{eq:kaehler4}
\end{align}
which combined with \eqref{eq:key1} implies
$(\p_x(J(U))U_x,\WW)=\mathcal{O}(|Z|)$. 
Therefore, by the integration by parts, 
the fifth term of the RHS of \eqref{eq:89e3} 
is bounded by 
$C\,D(t)^2$.
In the same way, we use \eqref{eq:kaehler2}, \eqref{eq:key1}, 
and \eqref{eq:key2}, 
to have  
\begin{align}
&(\p_x(J(U))\p_x\WW, \WW)
\nonumber
\\
&=
\sum_k
\left(\p_x\WW,J(U)D_k(U)U_x\right)(\nu_k(U),\WW) 
-\sum_k
(\p_x\WW,\nu_k(U))(J(U)D_k(U)U_x,\WW)
\nonumber
\\
&=
\sum_k
\left(\p_x\WW,J(U)D_k(U)U_x\right)\mathcal{O}(|Z|)
+
\mathcal{O}
\left(
(|Z|+|Z_x|+|\WW|)|\WW|\right). 
\nonumber
\end{align}
Thus the integration by parts shows that 
the eighth term of the RHS of \eqref{eq:89e3} 
is bounded by $C\,D(t)^2$.
For, the third and the ninth term of the RHS of 
\eqref{eq:89e3},  
in view of \eqref{eq:key5}, we have 
\begin{align}
\p_x(N(U))\p_x\WW
&=
\sum_{\ell}
(\p_x\WW,D_{\ell}(U)U_x)\nu_{\ell}(U)
+
\sum_{\ell}
(\p_x\WW,\nu_{\ell}(U))D_{\ell}(U)U_x
\nonumber
\\
&=
\sum_{\ell}
(\p_x\WW,D_{\ell}(U)U_x)\nu_{\ell}(U)
+
\mathcal{O}
(|Z|+|Z_x|+|\WW|),
\nonumber
\end{align}
which implies  
\begin{align}
&-2a\,
\lr{
\sum_k
(U_x, D_k(U)U_x)J(U)D_k(U)\p_x(N(U))\p_x\WW, 
\WW
}
\nonumber
\\
&\leqslant 
-2a\,
\lr{
\sum_{k,\ell}
(\p_x\WW, D_{\ell}(U)U_x)
(U_x, D_k(U)U_x)
J(U)D_k(U)\nu_{\ell}(U), 
\WW
}
+
C\,D(t)^2,
\nonumber
\\
\intertext{and}
&
2a\,
\lr{
\sum_k
(\p_x(N(U))\p_x\WW,D_k(U)U_x)J(U)D_k(U)U_x,\WW
}
\nonumber
\\
&\leqslant 
2a\,
\lr{
\sum_{k,\ell}
(\p_x\WW, D_{\ell}(U)U_x)
(\nu_{\ell}(U),D_k(U)U_x)J(U)D_k(U)U_x,\WW
}
+C\,D(t)^2.
\nonumber
\end{align}
Collecting them, we derive 
\begin{align}
R_4
+
R_5
&\leqslant 
2a\,
\lr{
\p_x\biggl\{
\sum_k
(U_x, D_k(U)U_x)J(U)D_k(U)
\biggr\}
\p_x\WW,
\WW
}
\nonumber
\\
&\quad
-2a\,
\lr{
\sum_{k,\ell}
(\p_x\WW, D_{\ell}(U)U_x)
(U_x, D_k(U)U_x)
J(U)D_k(U)\nu_{\ell}(U), 
\WW
}
\nonumber
\\
&\quad
+
2a\,
\lr{
\sum_{k,\ell}
(\p_x\WW, D_{\ell}(U)U_x)
(\nu_{\ell}(U),D_k(U)U_x)J(U)D_k(U)U_x,\WW
}
\nonumber
\\
&\quad
+2aS\,
\lr{
(\p_x\WW,\p_xU_x)J(U)U_x,
\WW
}
+2aS\,
\lr{
(\p_x\WW,U_x)J(U)\p_xU_x,
\WW
}
\nonumber
\\
&\quad
-4aS\,
\lr{
(\p_xU_x,U_x)J(U)\p_x\WW,
\WW
}
+
C\,D(t)^2.
\label{eq:89e}
\end{align}
\par 
Therefore, by \eqref{eq:67e} and \eqref{eq:89e}, 
we obtain 
\begin{align}
&R_2+R_3+R_4+R_5
\nonumber
\\
&\leqslant 
a\,
\lr{
\p_x\biggl\{
\sum_k
(U_x, D_k(U)U_x)(J(U)D_k(U)-D_k(U)J(U))
\biggr\}
\p_x\WW,
\WW
}
\nonumber
\\
&\quad
+2aS\,
\lr{
(\p_x\WW,\p_xU_x)J(U)U_x,
\WW
}
+2aS\,
\lr{
(\p_x\WW,U_x)J(U)\p_xU_x,
\WW
}
\nonumber
\\
&\quad
-2aS\,
\lr{
(\p_xU_x,U_x)J(U)\p_x\WW,
\WW
}
+
C\,D(t)^2.
\label{eq:R2345}
\end{align}
Note here that  
\begin{align}
&\left(
\sum_k(\UU,D_k(U)U_x)(J(U)D_k(U)-D_k(U)J(U))Y_1,Y_2
\right)
\nonumber
\\
&=
\left(
Y_1, \sum_k(\UU,D_k(U)U_x)(J(U)D_k(U)-D_k(U)J(U))Y_2
\right)
\label{eq:sym41}
\end{align}
holds for any $Y_1,Y_2:[0,T]\times \TT\to \RR^d$. 
This follows immediately from \eqref{eq:J(wp)1} and 
\eqref{eq:D_k}. 
By taking the derivative of both sides of \eqref{eq:sym41} in $x$, 
we see that  
\begin{align}
&\left(
\p_x\left\{
\sum_k(\UU,D_k(U)U_x)(J(U)D_k(U)-D_k(U)J(U))
\right\}
Y_1,Y_2
\right)
\nonumber
\\
&=
\left(
Y_1, \p_x\left\{
\sum_k(\UU,D_k(U)U_x)(J(U)D_k(U)-D_k(U)J(U))
\right\}
Y_2
\right)
\nonumber 
\end{align}
holds for any $Y_1,Y_2:[0,T]\times \TT\to \RR^d$. 
Hence, by integrating by parts,  
we see that 
the first term of the RHS of \eqref{eq:R2345}
is bounded by $C\,D(t)^2$. 
Therefore we obtain 
\begin{align}
&R_2+R_3+R_4+R_5
\nonumber
\\
&\leqslant 
2aS\,
\lr{
(\p_x\WW,\p_xU_x)J(U)U_x,
\WW
}
+2aS\,
\lr{
(\p_x\WW,U_x)J(U)\p_xU_x,\WW
}
\nonumber
\\
&\quad
-2aS\,
\lr{
(\p_xU_x,U_x)J(U)\p_x\WW,\WW
}
+
C\,D(t)^2.
\label{eq:RR2345}
\end{align}
\par 
Gathering the information 
\eqref{eq:nn}, 
\eqref{eq:RR167}, 
and 
\eqref{eq:RR2345},
we derive  
\begin{align}
\frac{1}{2}
\frac{d}{dt}
\|\WW\|_{L^2}^2
&\leqslant 
(c+aS)\,
\lr{
(\p_x^2\WW,U_x)J(U)U_x, 
\WW
}
\nonumber
\\
&\quad 
+(-aS+2b+c)\, 
\lr{
(\p_xU_x,U_x)
J(U)\p_x\WW, 
\WW
}
\nonumber
\\
&\quad
+
(2aS+2b+2c)\, 
\lr{
(\p_x\WW,U_x)J(U)\p_xU_x, 
\WW
}
\nonumber
\\
&\quad 
+
(aS+3c)\,
\lr{
(\p_x\WW,\p_xU_x)J(U)U_x, 
\WW
}
+C\,D(t)^2.
\label{eq:WWt2}
\end{align}
Furthermore we rewrite the third and fourth term of the RHS 
of \eqref{eq:WWt2} recalling Definition~\ref{definition:symmetry_uniqueness} 
and Lemma~\ref{lemma:symmetry_uniqueness}. 
From Definition~\ref{definition:symmetry_uniqueness}, 
it follows that 
\begin{align} 
&(\p_x\WW,U_x)J(U)\p_xU_x
\nonumber
\\
&=
\frac{1}{2}
\left(
T_3(U)-T_4(U)+T_5(U)
\right)\p_x\WW
+
\frac{1}{2}
\sum_k
(U_x,D_k(U)U_x)
(\p_x\WW,\nu_k(U))J(U)U_x. 
\nonumber
\end{align}
Using \eqref{eq:tae2} and \eqref{eq:tae3} with $Y=\p_x\WW$,  
we see 
\begin{align}
T_5(U)\p_x\WW
&=
(\p_xU_x,U_x)J(U)\p_x\WW
+\frac{1}{2}
|U_x|^2\p_x(J(U))\p_x\WW
\nonumber
\\
&\quad
+\frac{1}{2}(\p_x\WW,U_x)
\sum_k
(J(U)U_x,D_k(U)U_x)\nu_k(U)
\nonumber
\\
&\quad 
-\frac{1}{2}
\sum_k
(\p_x\WW,\nu_k(U))
(J(U)U_x,D_k(U)U_x)U_x.
\nonumber
\end{align}
Substituting this and using \eqref{eq:key2}, 
we obtain 
\begin{align}
&(\p_x\WW,U_x)J(U)\p_xU_x
\nonumber
\\
&=
\frac{1}{2}(\p_xU_x,U_x)J(U)\p_x\WW
+
\frac{1}{2}
\left(
T_3(U)-T_4(U)
\right)\p_x\WW
\nonumber
\\
&\quad
+
\frac{1}{4}
|U_x|^2\p_x(J(U))\p_x\WW
+\frac{1}{4}(\p_x\WW,U_x)
\sum_k
(J(U)U_x,D_k(U)U_x)\nu_k(U)
\nonumber
\\
&\quad
+\mathcal{O}(|Z|+|Z_x|+|\WW|). 
\label{eq:tae6}
\end{align}
In the same way,  by using Definition~\ref{definition:symmetry_uniqueness} 
and Lemma~\ref{lemma:symmetry_uniqueness},
we obtain 
\begin{align}
&(\p_x\WW,\p_xU_x)J(U)U_x
\nonumber
\\
&=
\frac{1}{2}(\p_xU_x,U_x)J(U)\p_x\WW
+
\frac{1}{2}
\left(
T_3(U)+T_4(U)
\right)\p_x\WW
\nonumber
\\
&\quad
+
\frac{1}{4}
|U_x|^2\p_x(J(U))\p_x\WW
+\frac{1}{4}(\p_x\WW,U_x)
\sum_k
(J(U)U_x,D_k(U)U_x)\nu_k(U)
\nonumber
\\
&\quad
+\mathcal{O}(|Z|+|Z_x|+|\WW|).
\label{eq:tae7}
\end{align}
Thanks to \eqref{eq:tae4} and \eqref{eq:tae5}, 
we can easily show 
$\lr{T_i(U)\p_x\WW,\WW}
\leqslant 
C\,D(t)^2$ with $i=3,4$,  
by integrating by parts. 
Besides, 
it is now immediate to see 
\begin{align}
&\lr{
|U_x|^2\p_x(J(U))\p_x\WW
,
\WW
}
\leqslant 
C\,D(t)^2,
\nonumber
\\
&\lr{
(\p_x\WW,U_x)
\sum_k
(J(U)U_x,D_k(U)U_x)\nu_k(U)
\WW
}
\leqslant 
C\,D(t)^2
\nonumber
\end{align}
by the argument 
using \eqref{eq:kaehler2} and \eqref{eq:key1}, 
Therefore, 
we substitute \eqref{eq:tae6} and \eqref{eq:tae7} 
into \eqref{eq:WWt2} to derive   
\begin{align}
&\frac{1}{2}
\frac{d}{dt}
\|\WW\|_{L^2}^2
\nonumber
\\
&\leqslant 
(c+aS)\,
\lr{
(\p_x^2\WW,U_x)J(U)U_x, 
\WW
} 
+(-aS+2b+c)\, 
\lr{
(\p_xU_x,U_x)
J(U)\p_x\WW, 
\WW
}
\nonumber
\\
&\quad
+
(aS+b+c)\, 
\lr{
(\p_xU_x,U_x)J(U)\p_x\WW, 
\WW
}
\nonumber
\\
&\quad 
+
\frac{aS+3c}{2}\,
\lr{
(\p_xU_x,U_x)J(U)\p_x\WW, 
\WW
}
+C\,D(t)^2.
\nonumber
\\
&=
(c+aS)\,
\lr{
(\p_x^2\WW,U_x)J(U)U_x, 
\WW
}
\nonumber
\\
&\quad 
+\frac{aS+6b+7c}{2}\, 
\lr{
(\p_xU_x,U_x)
J(U)\p_x\WW, 
\WW
}
+C\,D(t)^2.
\label{eq:WWt}
\end{align}  
Even if we use the integration parts, 
the first and the second term of 
the RHS of \eqref{eq:WWt} cannot be bounded by 
$C\,D(t)^2$.  
Fortunately, however, we will find in the next step
that the two terms can be eliminated essentially
by introducing a gauged function.
\vspace{0.5em}
\\
{\bf 4. Energy estimates for  
$\|\TW\|_{L^2(\TT;\RR^d)}$  
to eliminate the loss of derivatives.} 
\\
We introduce the function $\TW$ 
which is defined by   
\begin{align}
\widetilde{\WW}
&=
\WW+\widetilde{\Lambda},
\label{eq:sag1}
\end{align}
where 
\begin{align}
\widetilde{\Lambda}
&=
-\frac{e_1}{2a}(Z,J(U)U_x)J(U)U_x
+
\frac{e_2}{8a}|U_x|^2Z,
\label{eq:sag2}
\\
e_1&=aS+c, 
\quad 
e_2=e_1+\frac{aS+6b+7c}{2}.
\label{eq:e1e2}
\end{align}
Moreover, we introduce the energy $\widetilde{D}(t)$ whose square is 
defined by
\begin{align}
\widetilde{D}(t)^2
&=
\|Z(t)\|_{L^2}^2
+\|Z_x(t)\|_{L^2}^2
+\|\TW(t)\|_{L^2}^2.
\label{eq:menergy2}
\end{align}
Since $u$ and $v$ satisfy the same initial value, 
$\widetilde{D}(0)=0$ holds. 
We shall show that  
there exists a positive constant $C$ such that
\begin{equation}
\frac{1}{2}
\frac{d}{dt}\widetilde{D}(t)^2
\leqslant C\, 
\widetilde{D}(t)^2
\label{eq:De}
\end{equation}
for all $t\in (0,T)$. 
If it is true,  
\eqref{eq:De} together with $\widetilde{D}(0)=0$ shows $\widetilde{D}(t)\equiv 0$. 
This implies $Z=0$.  
\par
In the proof of \eqref{eq:De}, 
by integrating by parts repeatedly,
it is now not difficult to obtain the following estimate 
permitting the loss of derivatives of order one:  
\begin{equation}
\frac{1}{2}\frac{d}{dt}
\left\{
\|Z(t)\|_{L^2}^2
+
\|Z_x(t)\|_{L^2}^2
\right\}
\leqslant 
C\, \widetilde{D}(t)^2.  
\label{eq:notmainineq}
\end{equation}
Having them in mind, 
we hereafter concentrate on how to derive 
the estimate of the form
\begin{equation}
\frac{1}{2}\frac{d}{dt}
\|\TW(t)\|_{L^2}^2
\leqslant 
C\, \widetilde{D}(t)^2. 
\label{eq:mainineq}
\end{equation}
For this purpose, 
we begin with 
\begin{align}
\frac{1}{2}
\frac{d}{dt}
\|\TW\|_{L^2}^2
&=
\lr{
\p_t\TW,
\TW
}
\nonumber
\\
&=
\lr{
\p_t\WW,
\TW
}
+
\lr{
\p_t\widetilde{\Lambda},
\TW
}
\nonumber
\\
&=
\lr{
\p_t\WW,\WW
}
+
\lr{
\p_t\WW,\widetilde{\Lambda}
}
+
\lr{
\p_t\widetilde{\Lambda},\TW
}.
\label{eq:TW1}
\end{align}
The first term of the RHS of \eqref{eq:TW1} 
has already been investigated to satisfy 
\eqref{eq:WWt}.
Hence we compute the second and the third term 
of the RHS of \eqref{eq:TW1} below. 
Observing $\widetilde{\Lambda}=\mathcal{O}(|Z|)$, 
we see 
$\TW=\WW+\mathcal{O}(|Z|)$, 
$\p_x\TW=\p_x\WW+\mathcal{O}(|Z|+|Z_x|)$, 
and 
$\p_x^2\TW=\p_x^2\WW+\mathcal{O}(|Z|+|Z_x|+|\TW|)$, 
which will be often used without comments. 
\par
We start the computation of 
$\lr{
\p_t\widetilde{\Lambda},\TW}$ 
by investigating $\p_t\widetilde{\Lambda}$. 
A simple computation shows    
\begin{align}
\p_t\widetilde{\Lambda}
&=
-\frac{e_1}{2a}(Z_t,J(U)U_x)J(U)U_x
+
\frac{e_2}{8a}|U_x|^2Z_t
+\mathcal{O}(|Z).
\label{eq:aya1}
\end{align} 
Recalling \eqref{eq:U_t}, we see  
\begin{align}
Z_t
&=
a\,\p_x\left(
J(U)\p_x\WW
\right)
+
a\,
\sum_k
(J(U)\p_x\WW, D_k(U)U_x)\nu_k(U)
+
\mathcal{O}(|Z|+|Z_x|+|\TW|)
\label{eq:aya2}
\\
&=
a\,J(U)\p_x^2\WW
+a\,\p_x(J(U))\p_x\WW
+
a\,
\sum_k
(J(U)\p_x\WW, D_k(U)U_x)\nu_k(U)
\nonumber
\\
&\quad
+
\mathcal{O}(|Z|+|Z_x|+|\TW|).
\label{eq:aya3}
\end{align}
By using \eqref{eq:aya3}, we see
\begin{align}
&-\frac{e_1}{2a}(Z_t,J(U)U_x)J(U)U_x
\nonumber
\\
&=
-\frac{e_1}{2}(J(U)\p_x^2\WW,J(U)U_x)J(U)U_x
-\frac{e_1}{2}(\p_x(J(U))\p_x\WW,J(U)U_x)J(U)U_x
\nonumber
\\
&\quad 
-\frac{e_1}{2}
\sum_k
(J(U)\p_x\WW, D_k(U)U_x)(\nu_k(U),J(U)U_x)J(U)U_x
+
\mathcal{O}(|Z|+|Z_x|+|\TW|).
\nonumber
\end{align}
The third term of the RHS vanishes, since $(\nu_k(U),J(U)U_x)=0$. 
By noting \eqref{eq:0yumi}, we see that 
the second term of the RHS is $\mathcal{O}(|Z|+|Z_x|+|\TW|)$.
Thus we have 
\begin{align}
-\frac{e_1}{2a}(Z_t,J(U)U_x)J(U)U_x
&=
-\frac{e_1}{2}(\p_x^2\WW,U_x)J(U)U_x
+\mathcal{O}(|Z|+|Z_x|+|\TW|).
\label{eq:aya4}
\end{align}
On the other hand, by using \eqref{eq:aya2}, 
we obtain 
\begin{align}
\frac{e_2}{8a}|U_x|^2Z_t 
&=
\frac{e_2}{8}|U_x|^2
\p_x\left(
J(U)\p_x\WW
\right)
+
\frac{e_2}{8}|U_x|^2
\sum_k
(J(U)\p_x\WW, D_k(U)U_x)\nu_k(U)
\nonumber
\\
&\quad
+\mathcal{O}(|Z|+|Z_x|+|\TW|)
\nonumber
\\
&=
\frac{e_2}{8}
\p_x\left\{
|U_x|^2
J(U)\p_x\WW
\right\}
-
\frac{e_2}{4}
(\p_xU_x,U_x)J(U)\p_x\WW
\nonumber
\\
&\quad
+
\sum_k
\mathcal{O}
\left(
|\p_x\WW|
\right)
\nu_k(U)
+\mathcal{O}(|Z|+|Z_x|+|\TW|).
\label{eq:aya5}
\end{align}
By substituting \eqref{eq:aya4} and \eqref{eq:aya5} 
into \eqref{eq:aya1}, 
we obtain 
\begin{align}
\p_t\widetilde{\Lambda}
&=
-\frac{e_1}{2}(\p_x^2\WW,U_x)J(U)U_x
+
\frac{e_2}{8}
\p_x\left\{
|U_x|^2
J(U)\p_x\WW
\right\}
-
\frac{e_2}{4}
(\p_xU_x,U_x)J(U)\p_x\WW
\nonumber
\\
&\quad
+
\sum_k
\mathcal{O}
\left(
|\p_x\WW|
\right)
\nu_k(U)
+\mathcal{O}(|Z|+|Z_x|+|\TW|).
\nonumber
\end{align}
This shows that 
\begin{align}
\lr{
\p_t\widetilde{\Lambda},\TW
}
&=
-\frac{e_1}{2}
\lr{
(\p_x^2\WW,U_x)J(U)U_x,\WW+\mathcal{O}(|Z|)
}
\nonumber
\\
&\quad
+\frac{e_2}{8}
\lr{
\p_x\left\{
|U_x|^2
J(U)\p_x\WW
\right\},\WW+\mathcal{O}(|Z|)
}
\nonumber
\\
&\quad
-\frac{e_2}{4}
\lr{
(\p_xU_x,U_x)J(U)\p_x\WW,\WW+\mathcal{O}(|Z|)
}
\nonumber
\\
&\quad
+
\lr{ 
\sum_k
\mathcal{O}
\left(
|\p_x\WW|
\right)
\nu_k(U)
, \WW+\mathcal{O}(|Z|)
}
\nonumber
\\
&\quad 
+\lr{
\mathcal{O}(|Z|+|Z_x|+|\TW|),\WW+\mathcal{O}(|Z|)
}
\nonumber
\\
&\leqslant 
-\frac{e_1}{2}
\lr{
(\p_x^2\WW,U_x)J(U)U_x,\WW
}
+\frac{e_2}{8}
\lr{
\p_x\left\{
|U_x|^2
J(U)\p_x\WW
\right\},\WW
}
\nonumber
\\
&\quad
-\frac{e_2}{4}
\lr{
(\p_xU_x,U_x)J(U)\p_x\WW,\WW
}
\nonumber
\\
&\quad
+
\lr{ 
\sum_k
\mathcal{O}
\left(
|\p_x\WW|
\right)
\nu_k(U), \WW
} 
+C\,\widetilde{D}(t)^2
\nonumber
\\
&\leqslant
-\frac{e_1}{2}
\lr{
(\p_x^2\WW,U_x)J(U)U_x,\WW
}
-\frac{e_2}{4}
\lr{
(\p_xU_x,U_x)J(U)\p_x\WW,\WW
}
\nonumber
\\
&\quad 
+C\,\widetilde{D}(t)^2.
\label{eq:tom1}
\end{align}
\par
We next compute 
$\lr{\p_t\WW,\widetilde{\Lambda}}$. 
Observing \eqref{eq:eqWW}, we see  
\begin{align}
\p_t\WW
&=a\,\p_x^2\left\{
J(U)\p_x^2\WW
\right\} 
-
2a\,
\sum_{k}
\left(\p_x^3\WW,
J(U)D_{k}(U)U_x
\right)\nu_{k}(U)
\nonumber
\\
&\quad
+
\mathcal{O}(|Z|+|Z_x|+|\WW|+|\p_x\WW|+|\p_x^2\WW|).
\nonumber
\end{align}
By using this and  by noting $\widetilde{\Lambda}=\mathcal{O}(|Z|)$, 
we integrate by parts to obtain 
\begin{align}
\lr{
\p_t\WW,\widetilde{\Lambda}
}
&\leqslant 
R_8+R_9
+C\,\widetilde{D}(t)^2, 
\label{eq:ma1}
\end{align}
where 
\begin{align}
R_8&=a\,
\lr{
\p_x^2\left\{
J(U)\p_x^2\WW
\right\},\widetilde{\Lambda}
}, 
\nonumber
\\
R_9
&=
-2a\,
\lr{
\sum_{k}
\left(\p_x^3\WW,
J(U)D_{k}(U)U_x
\right)\nu_{k}(U),\widetilde{\Lambda}
}.
\nonumber
\end{align}
For $R_9$, noting $\widetilde{\Lambda}=\mathcal{O}(|Z|)$,
we use the integration by parts 
and $(\nu_k(U),J(U)U_x)=0$ to obtain 
\begin{align}
R_9
&\leqslant 
-2a\,(-1)^3
\lr{
\sum_{k}
\left(\WW,
J(U)D_{k}(U)U_x
\right)\nu_{k}(U),\p_x^3\widetilde{\Lambda}
}
+
C\,\widetilde{D}(t)^2
\nonumber
\\
&\leqslant
2a\,
\lr{
\sum_{k}
\left(\WW,
J(U)D_{k}(U)U_x
\right)\nu_{k}(U),
-\frac{e_1}{2a}(\p_x^3Z,J(U)U_x)J(U)U_x
+
\frac{e_2}{8a}|U_x|^2\p_x^3Z
}
\nonumber
\\
&\quad 
+
C\,\widetilde{D}(t)^2
\nonumber
\\
&=\frac{e_2}{4}
\lr{
\sum_{k}
\left(\WW,
J(U)D_{k}(U)U_x
\right)\nu_{k}(U),
|U_x|^2\p_x^3Z
} 
+
C\,\widetilde{D}(t)^2.
\nonumber
\end{align}
Furthermore,
since 
$\p_x^3Z=\p_x^2Z_x
=
\p_x\WW
+
\mathcal{O}(|Z|+|Z_x|+|\WW|)
=\p_x\WW+\mathcal{O}(|Z|+|Z_x|+|\TW|)$
and 
$(\nu_k(U),\p_x\WW)=\mathcal{O}(|Z|+|Z_x|+|\WW|)$, 
we have 
\begin{align}
R_9
&\leqslant 
\frac{e_2}{4}
\lr{
\sum_{k}
\left(\WW,
J(U)D_{k}(U)U_x
\right)\nu_{k}(U),
|U_x|^2\p_x\WW
} 
+
C\,\widetilde{D}(t)^2
\leqslant 
C\,\widetilde{D}(t)^2.
\label{eq:ma2}
\end{align}
For $R_8$, 
we begin with   
\begin{align}
R_8
&=
-\frac{e_1}{2}\,
\lr{
\p_x^2\left\{
J(U)\p_x^2\WW
\right\}, (Z,J(U)U_x)J(U)U_x
}
+
\frac{e_2}{8}\,
\lr{
\p_x^2\left\{
J(U)\p_x^2\WW
\right\}, |U_x|^2Z
}
\nonumber
\\
&=:R_{81}+R_{82}.
\nonumber
\end{align}
The integration by parts implies 
\begin{align}
R_{81}
&=
-\frac{e_1}{2}\,
\lr{
J(U)\p_x^2\WW, 
\p_x^2\left\{(Z,J(U)U_x)J(U)U_x\right\}
}
\nonumber
\\
&\leqslant 
-\frac{e_1}{2}\,
\lr{
J(U)\p_x^2\WW, 
(\p_xZ_x,J(U)U_x)J(U)U_x
}
\nonumber
\\
&\quad
-e_1\,
\lr{
J(U)\p_x^2\WW, 
(Z_x,\p_x\left\{J(U)U_x\right\})J(U)U_x
}
\nonumber
\\
&\quad  
-e_1\,
\lr{
J(U)\p_x^2\WW, 
(Z_x,J(U)U_x)\p_x\left\{J(U)U_x\right\}
}
+C\,\widetilde{D}(t)^2
\nonumber
\\
&\leqslant 
-\frac{e_1}{2}\,
\lr{
J(U)\p_x^2\WW, 
(\p_xZ_x,J(U)U_x)J(U)U_x
}
\nonumber
\\
&\quad
+e_1\,
\lr{
J(U)\p_x\WW, 
(\p_xZ_x,\p_x\left\{J(U)U_x\right\})J(U)U_x
}
\nonumber
\\
&\quad  
+e_1\,
\lr{
J(U)\p_x\WW, 
(\p_xZ_x,J(U)U_x)\p_x\left\{J(U)U_x\right\}
}
+C\,\widetilde{D}(t)^2
\nonumber
\\
&=:R_{83}+R_{84}+R_{85}+C\,\widetilde{D}(t)^2.
\label{eq:ma4}
\end{align}
Since $\UU=\p_xU_x+\displaystyle\sum_{k}(U_x,D_k(U)U_x)\nu_k(U)$ 
and $\VV=\p_xV_x+\displaystyle\sum_{k}(V_x,D_k(V)V_x)\nu_k(V)$,  
we see 
$$\p_xZ_x=\WW-\sum_{k}(Z_x,D_k(U)U_x)\nu_k(U)
-\sum_k(V_x,D_k(U)Z_x)\nu_k(U)
+
\mathcal{O}(|Z|), 
$$
and thus
$
(\p_xZ_x, J(U)U_x)
=
(\WW, J(U)U_x)
+\mathcal{O}(|Z|)$. 
Substituting this, using \eqref{eq:J(wp)1} and \eqref{eq:J(wp)2}, 
and integrating by parts, 
we have 
\begin{align}
R_{83}&=
-\frac{e_1}{2}\,
\lr{
\p_x^2\WW, 
(\p_xZ_x,J(U)U_x)U_x
}
\nonumber
\\
&\leqslant 
-\frac{e_1}{2}\,
\lr{
\p_x^2\WW, 
(\WW,J(U)U_x)U_x
}
-\frac{e_1}{2}\,
\lr{
\p_x^2\WW, 
\mathcal{O}(|Z|)
}
+C\,\widetilde{D}(t)^2
\nonumber
\\
&\leqslant 
-\frac{e_1}{2}\,
\lr{
(\p_x^2\WW, U_x)J(U)U_x, \WW
}
+C\,\widetilde{D}(t)^2.
\label{eq:ma5}
\end{align}
From \eqref{eq:kaehler4}, it follows that 
\begin{equation}
\p_x\left\{J(U)U_x\right\}
=
J(U)\p_xU_x+\sum_k(U_x,J(U)D_k(U)U_x)\nu_k(U).
\nonumber
\end{equation}
Using this, $\p_xZ_x=\WW+\mathcal{O}(|Z|+|Z_x|)$, 
and \eqref{eq:key1}, 
we see 
\begin{align}
&(\p_xZ_x, \p_x\left\{J(U)U_x\right\})
\nonumber
\\
&=
(\p_xZ_x, J(U)\p_xU_x)
+\sum_k(U_x,J(U)D_k(U)U_x)(\nu_k(U),\p_xZ_x)
+
\mathcal{O}(|Z|+|Z_x|)
\nonumber
\\
&=
(\WW, J(U)\p_xU_x)
+\sum_k(U_x,J(U)D_k(U)U_x)(\nu_k(U),\WW)
+
\mathcal{O}(|Z|+|Z_x|)
\nonumber
\\
&=(\WW, J(U)\p_xU_x)+
\mathcal{O}(|Z|+|Z_x|).
\nonumber
\end{align}
This implies 
\begin{align}
R_{84}&=
e_1\,
\lr{
\p_x\WW, 
(\p_xZ_x,\p_x\left\{J(U)U_x\right\})U_x
}
\nonumber
\\
&\leqslant 
e_1\,
\lr{
\p_x\WW, 
(\WW, J(U)\p_xU_x)U_x
}
+C\,\widetilde{D}(t)^2
\nonumber
\\
&= 
e_1\,
\lr{
(\p_x\WW, U_x)
J(U)\p_xU_x, \WW
}
+C\,\widetilde{D}(t)^2.
\label{eq:ma6}
\end{align}
In the same way, we use \eqref{eq:J(wp)2} to see   
\begin{align}
&(J(U)\p_x\WW, \p_x\left\{J(U)U_x\right\})
\nonumber
\\
&=
(J(U)\p_x\WW, J(U)\p_xU_x)
+
\sum_k(U_x,J(U)D_k(U)U_x)
(J(U)\p_x\WW, \nu_k(U))
\nonumber
\\
&=
(\p_x\WW, P(U)\p_xU_x)
\nonumber
\\
&=
(\p_x\WW, \p_xU_x+\sum_k(U_x, D_k(U)U_x)\nu_k(U))
\nonumber
\\
&=
(\p_x\WW,\p_xU_x)
+
\mathcal{O}(|Z|+|Z_x|+|\WW|).
\nonumber
\end{align}
Substituting this, we obtain 
\begin{align}
R_{85}
&=e_1\,
\lr{
J(U)\p_x\WW, 
(\p_xZ_x,J(U)U_x)\p_x\left\{J(U)U_x\right\}
}
\nonumber
\\
&\leqslant
e_1\,
\lr{
\p_x\WW, 
(\p_xZ_x,J(U)U_x)\p_xU_x
}
+C\,\widetilde{D}(t)^2
\nonumber
\\
&\leqslant
e_1\,
\lr{
\p_x\WW, 
(\WW,J(U)U_x)\p_xU_x
}
+C\,\widetilde{D}(t)^2
\nonumber
\\
&=
e_1\,
\lr{
(\p_x\WW, \p_xU_x)J(U)U_x,\WW
}
+C\,\widetilde{D}(t)^2.
\label{eq:ma7}
\end{align}
Substituting \eqref{eq:ma5}, \eqref{eq:ma6}, and \eqref{eq:ma7}
into \eqref{eq:ma4}, 
we see $R_{81}=R_{83}+R_{84}+R_{85}+C\,\widetilde{D}(t)^2$
is bounded as follows:
\begin{align}
R_{81}
&\leqslant 
-\frac{e_1}{2}\,
\lr{
(\p_x^2\WW, U_x)J(U)U_x, \WW
}
+
e_1\,
\lr{
(\p_x\WW, U_x)
J(U)\p_xU_x, \WW
}
\nonumber
\\
&\quad
+
e_1\,
\lr{
(\p_x\WW, \p_xU_x)J(U)U_x,\WW
}
+C\,\widetilde{D}(t)^2.
\nonumber
\end{align}
Furthermore, by applying \eqref{eq:tae6} and \eqref{eq:tae7} 
to the second and the third term of the RHS of above, 
and by using 
$\lr{T_3(U)\p_x\WW,\WW}\leqslant C\,\widetilde{D}(t)^2$ again, 
we deduce 
\begin{align}
R_{81}
&\leqslant 
-\frac{e_1}{2}\,
\lr{
(\p_x^2\WW, U_x)J(U)U_x, \WW
}
+
e_1\,
\lr{
(\p_xU_x,U_x)J(U)\p_x\WW, \WW
}
\nonumber
\\
&\quad
+
e_1\,
\lr{
T_3(U)\p_x\WW,\WW
}
+C\,\widetilde{D}(t)^2
\nonumber
\\
&\leqslant 
-\frac{e_1}{2}\,
\lr{
(\p_x^2\WW, U_x)J(U)U_x, \WW
}
+
e_1\,
\lr{
(\p_xU_x,U_x)J(U)\p_x\WW, \WW
}
+C\,\widetilde{D}(t)^2.
\label{eq:ma8}
\end{align}
For $R_{82}$, the integration by parts and the same argument as above 
lead to  
\begin{align}
R_{82}
&=
\frac{e_2}{8}
\lr{
J(U)\p_x^2\WW, \p_x^2\left\{|U_x|^2Z\right\}
}
\nonumber
\\
&=
\frac{e_2}{8}
\lr{
J(U)\p_x^2\WW, |U_x|^2\p_xZ_x+4(\p_xU_x,U_x)Z_x+\mathcal{O}(|Z|)
}
\nonumber
\\
&\leqslant 
\frac{e_2}{8}
\lr{
|U_x|^2J(U)\p_x^2\WW, \WW
} 
+
\frac{e_2}{2}
\lr{
J(U)\p_x^2\WW,(\p_xU_x,U_x)Z_x 
}
+C\,\widetilde{D}(t)^2
\nonumber
\\
&\leqslant 
\frac{e_2}{8}
\lr{
\p_x
\left\{|U_x|^2J(U)\p_x\WW\right\}, \WW
}
-
\frac{e_2}{4}
\lr{
(\p_xU_x,U_x)J(U)\p_x\WW, \WW
}
\nonumber
\\
&\quad
-
\frac{e_2}{8}
\lr{
|U_x|^2\p_x(J(U))\p_x\WW, \WW
} 
-
\frac{e_2}{2}
\lr{
J(U)\p_x\WW,(\p_xU_x,U_x)\WW
}
+C\,\widetilde{D}(t)^2
\nonumber
\\
&\leqslant 
-
\frac{3e_2}{4}
\lr{
(\p_xU_x,U_x)J(U)\p_x\WW, \WW
} 
+C\,\widetilde{D}(t)^2.
\label{eq:ma9}
\end{align}
Therefore, from \eqref{eq:ma8}, and \eqref{eq:ma9}, 
it follows that $R_8=R_{81}+R_{82}$ is bounded as follows:
\begin{align}
R_8
&\leqslant 
-\frac{e_1}{2}\,
\lr{
(\p_x^2\WW, U_x)J(U)U_x, \WW
}
+
\left(e_1-\frac{3e_2}{4}\right)\,
\lr{
(\p_xU_x,U_x)J(U)\p_x\WW, \WW
}
\nonumber
\\
&\quad
+C\,\widetilde{D}(t)^2.
\label{eq:ma10}
\end{align}
Consequently, by substituting \eqref{eq:ma2} and \eqref{eq:ma10} 
into \eqref{eq:ma1}, 
we have 
\begin{align}
\lr{
\p_t\WW, \widetilde{\Lambda}
}
&\leqslant 
-\frac{e_1}{2}\,
\lr{
(\p_x^2\WW, U_x)J(U)U_x, \WW
}
\nonumber
\\
&\quad
+
\left(e_1-\frac{3e_2}{4}\right)\,
\lr{
(\p_xU_x,U_x)J(U)\p_x\WW, \WW
}
+C\,\widetilde{D}(t)^2.
\label{eq:tom2}
\end{align}
\par
Collecting the information 
\eqref{eq:TW1}, \eqref{eq:WWt}, \eqref{eq:tom1}, \eqref{eq:tom2}, 
and \eqref{eq:e1e2}, 
we conclude 
\begin{align}
\frac{1}{2}
\frac{d}{dt}
\|\TW\|_{L^2}^2
&\leqslant 
(c+aS-e_1)\,
\lr{
(\p_x^2\WW,U_x)J(U)U_x, 
\WW
}
\nonumber
\\
&\ \ 
+
\left(
\frac{aS+6b+7c}{2}+e_1-e_2
\right) 
\lr{
(\p_xU_x,U_x)
J(U)\p_x\WW, 
\WW
}
+C\,\widetilde{D}(t)^2
\nonumber
\\
&=
C\,\widetilde{D}(t)^2, 
\nonumber
\end{align}
which is the desired result \eqref{eq:mainineq}. 
\end{proof}
%\vspace{2em}
%\\
%
{\bf Acknowledgements.} \\
The author would like to thank Hiroyuki Chihara 
for valuable comments and encouragement. 
Thanks to his comments in \cite{chihara2}, 
the proof of Theorem~\ref{theorem:uniqueness} and
\ref{theorem:existence} 
is improved to be 
comprehensible.
This work is supported by 
JSPS Grant-in-Aid for 
Young Scientists (B) \#24740090.

%%%%%% End
%%%%%%
%
%%%%%%
%%%%%%
%%%%%% References
%%%%%%

%%%%%%
%%%%%%
%%%%%%
%%%%%%

\begin{thebibliography}{10}
\bibitem{AM}
S.C.~Anco  and R.~Myrzakulov,
\textit{ Integrable generalizations of Schr\"odinger maps 
and Heisenberg spin models from Hamiltonian 
flows of curves and surfaces.} 
J.\ Geom.\ Phys. \textbf{60} (2010), 1576--1603.

\bibitem{CSU}
N.-H.~Chang, J.~Shatah, and K.~Uhlenbeck,
\textit{Schr\"odinger maps.} 
Comm.\ Pure Appl.\ Math. \textbf{53} (2000), 590--602.

\bibitem{chihara} 
H.~Chihara, 
\textit{Schr\"odinger flow into almost Hermitian manifolds.} 
Bull.\ Lond.\ Math.\ Soc. \textbf{45} (2013), 37--51. 

\bibitem{chihara2} 
H.~Chihara, 
\textit{Fourth-order dispersive systems on the one-dimensional torus.} 
J.\ Pseudo-Differ.\ Oper.\ Appl. \textbf{6} (2015), 237--263. 


\bibitem{CO}
H.~Chihara and E.~Onodera, 
{\it A third order dispersive flow for closed curves into almost Hermitian manifolds}, 
J. Funct.\ Anal.\ {\bf 257}, 388--404, (2009).

\bibitem{CO2}
H.~Chihara and E.~Onodera, 
\textit{A fourth-order dispersive flow into K\"ahler manifolds.} 
Z.\ Anal.\ Anwend.\ \textbf{34} (2015), 221--249. 


\bibitem{fukumoto}
Y.~Fukumoto, 
\textit{Three-dimensional motion of a vortex filament and its relation 
to the localized induction hierarchy.}
Eur.\ Phys.\ J. B \textbf{29} (2002), 167--171.

\bibitem{FM}
Y.~Fukumoto and H.~K.~Moffatt, 
\textit{Motion and expansion of a viscous vortex ring. Part I. 
A higher-order asymptotic formula for the velocity.}
J.\ Fluid.\ Mech.  \textbf{417}  (2000), 1--45.

\bibitem{GZS} 
B.~Guo, M.~Zeng, and F.~Su, 
\textit{Periodic weak solutions for a classical 
one-dimensional isotropic biquadratic Heisenberg spin chain.} 
J. Math.\ Anal.\ Appl.  \textbf{330} (2007), 729--739.

\bibitem{KLPST}
C.~Kenig, T.~Lamm, D.~Pollack, G.~Staffilani, and T.~Toro, 
\textit{The Cauchy problem for Schr\"odinger flows.}
Discrete Contin.\ Dyn.\ Syst.  \textbf{27} (2010), 389--439.

\bibitem{koiso}
N.~Koiso, 
\textit{The vortex filament equation and 
a semilinear Schr\"odinger equation in a Hermitian symmetric space.}
Osaka J. Math.  \textbf{34} (1997), 199--214.

\bibitem{LPD}
M.~Lakshmanan, K.~Porsezian, and M.~Daniel, 
\textit{Effect of discreteness on the continuum limit of the Heisenberg
spin chain.}
Phys.\ Lett.\ A. \textbf{133} (1988), 483--488.

\bibitem{McGahagan}
H.~McGahagan,
\textit{An approximation scheme for Schr\"odinger maps.} 
Comm.\ Partial Differential Equations
\textbf{32} (2007), 375--400.

\bibitem{NSVZ}
A.~Nahmod, J.~Shatah, L.~Vega, and C.~Zeng, 
\textit{Schr\"odinger maps and their associated frame systems.} 
Int.\ Math.\ Res.\ Not.\ IMRN 2007, no. 21, Art. ID rnm088, 29 pp.

\bibitem{onodera0} 
E.~Onodera, 
\textit{Generalized Hasimoto transform of one-dimensional 
dispersive flows into compact Riemann surfaces.} 
SIGMA Symmetry Integrability Geom.\ Methods Appl. 
\textbf{4} (2008), article No. 044, 10 pages.

\bibitem{onodera1}
E.~Onodera, 
\textit{A third-order dispersive flow for closed curves 
into K\"ahler manifolds.} 
J. Geom.\ Anal.  \textbf{18} (2008), 889--918.

\bibitem{onodera3}
E.~Onodera, 
\textit{A remark on the global existence of a third order dispersive 
flow into locally Hermitian symmetric spaces.}
Comm. Partial Differential Equations  \textbf{35} (2010), 1130--1144.

\bibitem{onodera2}
E.~Onodera, 
\textit{The initial value problem for a fourth-order 
dispersive closed curve flow on the two-sphere.} 
to appear in Proc. Roy. Soc. Edinburgh Sect. A. 

\bibitem{SSB}
P.L.~Sulem, C.~Sulem, and C.Bardos, 
\textit{On the continuous limit for a system of classical spins.} 
Comm.\ Math.\ Phys.  \textbf{107} (1986), 431--454.

\end{thebibliography}
\end{document}